\documentclass[11pt]{amsart}

\usepackage[overload]{textcase}
\usepackage{comment}
\usepackage{hyperref}
\usepackage[capitalise,noabbrev]{cleveref}
\usepackage{amsmath}
\usepackage{amssymb}
\usepackage{enumitem}
\usepackage{thmtools}

\usepackage[labelformat=simple]{subcaption}
\usepackage{tikz}
\usepackage{tkz-graph}
\tikzstyle{VertexStyle} = [shape = circle, draw, fill]
\tikzset{pre/.style={-}}    
\tikzstyle{every node}=[circle, inner sep=0pt, minimum width=4pt]
\usetikzlibrary{backgrounds}
\newtheorem{theorem}{Theorem}[section]
\newtheorem{lemma}[theorem]{Lemma}
\newtheorem{corollary}[theorem]{Corollary}
\newtheorem{proposition}[theorem]{Proposition}
\newtheorem{sublemma}{}[theorem]

\theoremstyle{definition}
\newtheorem{definition}{Definition}[section]

\newcommand{\dY}{$\Delta$-$Y$~exchange}
\newcommand{\Yd}{$Y$-$\Delta$~exchange}
\newcommand{\ba}{\backslash}

\DeclareMathOperator{\cl}{cl}
\DeclareMathOperator{\fcl}{fcl}
\DeclareMathOperator{\co}{co}
\DeclareMathOperator{\si}{si}
\newcommand{\cocl}{\cl^*}
\newcommand{\clstar}{\cl^{(*)}}
\newcommand{\seq}[1]{[#1]}

\newcommand{\unfortunate}{$N$-grounded}

\newcommand{\pspider}{elongated-quad $3$-separator}

\newcommand{\twisted}{skew-whiff $3$-separator}
\newcommand{\spikelike}{spike-like $3$-separator}

\newcommand{\tvamoslike}{twisted cube-like $3$-separator}
\newcommand{\psep}{particular $3$-separator}
\newcommand{\Psep}{Particular $3$-separator}

\newenvironment{slproof}[1][Subproof]{\begin{proof}[#1]}{\end{proof}}

\crefformat{sublemma}{#2\rm{#1}#3}
\crefrangeformat{sublemma}{#3\rm{#1}#4--#5\rm{#2}#6}
\crefmultiformat{sublemma}{#2\rm{#1}#3}{ and~#2\rm{#1}#3}{, #2\rm{#1}#3}{ and~#2\rm{#1}#3}
\crefrangeformat{subsublemma}{#3\rm{#1}#4--#5\rm{#2}#6}
\crefmultiformat{subsublemma}{#2\rm{#1}#3}{ and~#2\rm{#1}#3}{, #2\rm{#1}#3}{ and~#2\rm{#1}#3}
\crefrangeformat{enumi}{#3\rm{#1}#4--#5\rm{#2}#6}
\crefname{enumi}{}{}
\crefname{sublemma}{}{}
\Crefname{sublemma}{Claim}{Claims}

\setenumerate{label=\rm(\roman*),midpenalty=2}

\begin{document}

\title[$N$-detachable pairs I: unveiling $X$]{$N$-detachable pairs in $3$-connected matroids I: unveiling $X$}

\thanks{The authors were supported by the New Zealand Marsden Fund.}

\author{Nick Brettell \and Geoff Whittle \and Alan Williams}
\address{Department of Computer Science, Durham University, United Kingdom}
\email{nbrettell@gmail.com}
\address{School of Mathematics and Statistics, Victoria University of Wellington, New Zealand}
\email{geoff.whittle@vuw.ac.nz}
\email{ayedwilliams@gmail.com}


\date{\today}

\maketitle

\begin{abstract}
  Let $M$ be a $3$-connected matroid, and let $N$ be a $3$-connected minor of $M$.
  We say that a pair $\{x_1,x_2\} \subseteq E(M)$ is \emph{$N$-detachable} if 
  one of the matroids $M/x_1/x_2$ or $M \backslash x_1 \backslash x_2$ is both $3$-connected and has an $N$-minor.
  This is the first in a series of three papers where we describe the structures that arise when $M$ has no $N$-detachable pairs.
  In this paper, we prove that if no $N$-detachable pair can be found in $M$, then either
  $M$ has a $3$-separating set, which we call $X$, with certain strong structural properties, or
  $M$ has one of three particular $3$-separators that can appear in a matroid with no $N$-detachable pairs.
\end{abstract}

\section{Introduction}

Let $M$ be a $3$-connected matroid, and let $N$ be a $3$-connected minor of $M$.
We say that a pair $\{x_1,x_2\} \subseteq E(M)$ is \emph{$N$-detachable} if 
either 
\begin{enumerate}[label=\rm(\alph*)]
  \item $M/x_1/x_2$ is $3$-connected and has an $N$-minor, or
  \item $M \ba x_1 \ba x_2$ is $3$-connected and has an $N$-minor.
\end{enumerate}

This is the first in a series of three papers where we describe the structures that
arise when it is not possible to find an $N$-detachable pair. As a consequence, we show
that if $M$ has at least ten more elements than $N$, then
 either $M$ has an $N$-detachable pair after possibly performing a single
$\Delta$-$Y$ or $Y$-$\Delta$ exchange, or $M$ is essentially $N$ with a single 
spike attached. More precisely, we have the following theorem.
Formal definitions of $\Delta$-$Y$ exchange and ``spike-like $3$-separator'' are given 
in Section~\ref{presec}.

\begin{theorem}
  Let $M$ be a $3$-connected matroid, 
  and let $N$ be a $3$-connected minor of $M$ such that $|E(N)| \ge 4$ and $|E(M)|-|E(N)| \ge 10$.
  Then either
  \begin{enumerate}
    \item $M$ has an $N$-detachable pair,
    \item there is a matroid $M'$ obtained by performing a single $\Delta$-$Y$ or $Y$-$\Delta$ exchange on $M$ such that $M'$ has 
      an $N$-detachable pair, or
    \item there is a \spikelike~$P$ of $M$ such that at most one element of $E(M)-E(N)$ is not in $P$.
  \end{enumerate}
  \label{maintheorem}
\end{theorem}

In fact, we prove a stronger result that requires only that $|E(M)| - |E(N)| \ge 5$, but a handful of additional highly structured outcomes involving particular
$3$-separators of bounded size
arise. Describing these requires some preparation and we defer the full statement of the
stronger theorem until the third paper. 

These papers had their genesis in the Ph.D.\ thesis of Alan Williams~\cite{Williams2015} where
the problem of finding a detachable pair without worrying about keeping a minor was solved.
In essence, the strategy here follows the strategy of \cite{Williams2015}, but with the additional
responsibility of always taking care to keep the minor.

\subsection*{Background and motivation}
The proof of \cref{maintheorem} is long; much longer than we originally anticipated. 
Without a solid motivation, the case for going to the trouble of proving it is weak indeed.
In fact, the motivation is clear. It comes from a desire to find exact excluded-minor 
characterisations of certain minor-closed classes of representable matroids. 
What follows is a discussion of that motivation.

To some extent, progress in matroid theory
can be measured by success in finding excluded-minor characterisations of classes of
matroids. Results to date include Tutte's 
excluded-minor characterisation of binary and regular matroids \cite{tutte1958homotopy}; Bixby's
and, independently, Seymour's excluded-minor characterisation of ternary matroids \cite{bixby1979reid,seymour1979matroid};
Geelen, Gerards and Kapoor's excluded-minor characterisation of GF$(4)$-representable
matroids \cite{Geelen2000excluded}; and Hall, Mayhew and van Zwam's excluded-minor characterisation of the
near-regular matroids, that is, the matroids representable over all fields 
with at least three elements \cite{hall2011excluded}. 
Recently Geelen, Gerards and Whittle announced a proof
of Rota's Conjecture \cite{geelen2014solving}.  However, their techniques 
are extremal and give no insight into how one might find the exact list of
excluded minors for such classes. Extending the range of known exact excluded-minor theorems
for basic classes of matroids remains a problem of genuine interest and, indeed, a significant
challenge that tests the state of the art of techniques in matroid theory.

At this stage we need to note that regular matroids, and many other naturally arising
classes of representable matroids such as near-regular, 
dyadic and $\sqrt[6]{1}$-matroids \cite{whittle1997matroids},
can be described as classes of matroids representable
over an algebraic structure called a {\em partial field}. Of course, a field is an example
of a partial field, and classes of 
matroids representable over partial fields enjoy many of the properties that hold for matroids
representable over fields \cite{pendavingh2010lifts,pendavingh2010confinement,semple1996partial}. 

The immediate problem that looms large is that of finding the excluded minors for
the class of GF$(5)$-representable matroids. While this problem is beyond the range of
current techniques, a road map for an attack is outlined in \cite{pendavingh2010confinement}. In essence,
this road map reduces the problem to a finite sequence of problems of the following type.
We have the class of $\mathbb P$-representable matroids for some
fixed partial field $\mathbb P$. We have a $3$-connected matroid $N$ with the property that
every $\mathbb P$-representation of $N$ extends {\em uniquely} to a $\mathbb P$-representation
of any  $3$-connected $\mathbb P$-representable matroid  having $N$ as a minor. Such a matroid 
$N$ is
called a {\em strong stabilizer} for the class of $\mathbb P$-representable matroids. 
With these ingredients,
the goal is to bound the size of an excluded minor for the class of $\mathbb P$-representable
matroids having 
the strong stabilizer $N$ as a minor.
This situation is a more general version of the one that arises in the proof of the excluded-minor characterisation 
of GF$(4)$-representable matroids \cite{Geelen2000excluded}. There, the partial field is GF$(4)$
and the fixed minor $N$ is $U_{2,4}$.

For all of the classes described above we may attempt to generalise the 
strategy developed by Geelen, Gerards and Kapoor. We have an excluded minor $M$, with 
strong stabilizer $N$. We wish to bound the size of $M$ relative to $N$. 
Assume, for a contradiction, that $M$ is large relative to $N$.
It is proved in  \cite{Geelen2000excluded,stabilizers} that in this case, up
to duality, one can find a pair of elements $x,y\in E(M)$ such that $M\ba x$, $M\ba y$
and $M\ba x\ba y$ have $N$-minors and are $3$-connected up to series pairs. Finding such a 
pair is the first step in the proofs given in \cite{Geelen2000excluded,hall2011excluded}. But there is the rub.
The possible presence of series pairs leads to a major complication in the subsequent analysis.
The current proofs for the excluded-minor characterisations of both GF$(4)$-representable and
near-regular matroids could be significantly shortened if we could replace
``$3$-connected up to series pairs'' by ``$3$-connected'' in the initial step. That is
precisely what Theorem~\ref{maintheorem} enables us to do. 

If we are to succeed in finding the excluded minors for the classes of matroids 
that would lead to an exact solution to Rota's Conjecture for GF$(5)$, eliminating
unnecessary technicalities in the analyses becomes more than just a convenience; it
becomes absolutely essential. Eliminating unnecessary technicalities is what this paper achieves. It gives a feasible first
step on the way to an explicit characterisation of the excluded minors for these classes.

Note that outcomes (ii) and (iii) of Theorem~\ref{maintheorem} do not limit its applicability for
finding excluded-minor characterisations of matroids representable over partial fields.
It is known that excluded minors for a partial field are closed under the $\Delta$-$Y$ 
exchange \cite{oxley2000generalized}. Moreover, it is not difficult to show that excluded minors have bounded-size spike-like $3$-separators.

Theorem~\ref{maintheorem} has already been applied to make further progress on excluded-minor problems. 
For a fixed matroid $N$, a matroid $M$ is $N$-{\em fragile} if, for all $e\in E(M)$,
at most one of  $M\ba e$ and $M/e$ has an $N$-minor. It is shown in \cite{bcosw2019} that if
$M$ is a sufficiently large excluded minor for a partial field~$\mathbb P$ with a strong stabilizer $N$ as
a minor, then $M$ is $\Delta$-$Y$-equivalent to a matroid from which an $N$-fragile matroid can be obtained by deleting two elements. The proof of this result makes essential use of Theorem~\ref{maintheorem}.
In essence, this reduces the problem of bounding the size of
an excluded minor to understanding the class of $\mathbb P$-representable
$N$-fragile matroids. In general this appears to be a difficult problem, but progress
has been made for two genuinely interesting classes.

The Hydra-5 partial field captures the first layer of the hierarchy of 
GF$(5)$-representable partial fields mentioned above. The 2-regular partial field
has the property 2-regular-representable matroids are representable over all fields
of size at least four. It turns out that $U_{2,5}$ is a strong stabilizer for
both these partial fields. Moreover, the $U_{2,5}$-fragile matroids that are either 
2-regular or Hydra-5 representable are known \cite{clark2015structure}. Using this, it is 
possible to obtain 
an explicit bound for the size of an excluded minor for either of these partial fields \cite{ben}.
The current bound is too large to enable an exhaustive search for the excluded minors.
It is hoped that, in the not too distant future, we can refine this bound and 
obtain an explicit list of the excluded minors. There would be some satisfaction in 
achieving this. Finding the excluded minors would be an important first step on the
way to getting the excluded minors for GF$(5)$. Having said that, it seems likely that,
in the end, combinatorial explosion will make the full solution impossible. Nonetheless it
is interesting to know just where the boundary of infeasibility lies. 

On the other hand, obtaining the excluded minors for the 2-regular matroids would be a
significant step towards understanding the matroids representable over all fields of size
at least 4. We know that this class contains the class of 2-regular matroids.
The excluded
minors for 2-regular that belong to the class would be interesting indeed and it is likely
that they could be exploited to obtain an explicit description of the class of matroids
representable over all fields of size at least four. Indeed it would also be a significant
step towards understanding the classes that arise when one considers matroids representable
over sets of fields that contain GF$(4)$. This would generalise analogous results for
GF$(2)$ and GF$(3)$ \cite{tutte1958homotopy,whittle1997matroids}.

\subsection*{The structure of these papers}
We now outline the approach taken to prove \cref{maintheorem} in this series of papers.
As is traditional, we begin by recalling Seymour's Splitter Theorem.

\begin{theorem}[Seymour's Splitter Theorem~\cite{pds}]
  Let $M$ be a $3$-connected matroid that is not a wheel or a whirl, 
  and let $N$ be a $3$-connected proper minor of $M$.
  Then there exists an element $e \in E(M)$ such that $M/e$ or $M\ba e$ is $3$-connected and has an $N$-minor.
\end{theorem}

By Seymour's Splitter Theorem, we may assume, up to duality, that there is an element $d \in E(M)$ such that $M \ba d$ is $3$-connected and has an $N$-minor.
Let $d' \in E(M \ba d)$ such that $M \ba d \ba d'$ has an $N$-minor.
If $M \ba d \ba d'$ is $3$-connected, then $\{d,d'\}$ is an $N$-detachable pair.
On the other hand, if $M \ba d \ba d'$ is not $3$-connected, then $M \ba d \ba d'$ has a $2$-separation $(Y,Z)$ where 
the $N$-minor is primarily contained in one side of the $2$-separation, $Z$ say.

The main result of this first paper of the series shows that if $M$ were to have no $N$-detachable pairs, and $|Y| \ge 4$, then either $Y$ contains a $3$-separating set $X$ with a number of strong structural properties, or $Y \cup d$ contains one of the handful of \psep s that can appear in a matroid with no $N$-detachable pairs (we describe these \psep s in \cref{sec-problematic}).
On the journey towards the proof of this result, we prove a number of lemmas about the existence of $N$-detachable pairs when $M$ or $M^*$ contains one of a few special structures: namely, triangles (\cref{sectris}), a $U_{3,5}$-restriction (\cref{secplanes}), or a single-element extension of a flan (\cref{secflans}). 
The proof of the main result is in \cref{secunveil}.

In the second paper, we further analyse this structured set $X$, and show that if we cannot find an $N$-detachable pair, then $X \cup d$ is contained in one of the handful of \psep s that can appear in a matroid with no $N$-detachable pairs.
In the third paper, the main hurdle that remains is handling the case where for any pair $\{d,d'\}$ such that $M \ba d \ba d'$ has an $N$-minor, the $2$-separation $(Y,Z)$ in $M \ba d \ba d'$ has $|Y|<4$. 

\section{Preliminaries}
\label{presec}

The notation and terminology in the paper follow Oxley~\cite{oxbook}.
We write $x \in \clstar(Y)$ to denote that either $x \in \cl(Y)$ or $x \in \cocl(Y)$.
The phrase ``by orthogonality'' refers to the fact that a circuit and a cocircuit cannot intersect in exactly one element.
For a set~$X$ and element~$e$, we write $X \cup e$ instead of $X \cup \{e\}$, and $X-e$ instead of $X-\{e\}$.
We say that $X$ meets $Y$ if $X \cap Y \neq \emptyset$.
We denote $\{1,2,\dotsc,n\}$ by $\seq{n}$.

\subsection*{Connectivity}
Let $M$ be a matroid with ground set $E$.  The \emph{connectivity function} of $M$, denoted by $\lambda_M$, is defined as follows, for all subsets $X$ of \nopagebreak $E$:
\begin{align*}
  \lambda_M(X) = r(X) + r(E - X) - r(M).
\end{align*}
A subset $X$ or a partition $(X, E-X)$ of $E$ is \emph{$k$-separating} if $\lambda_M(X) \leq k-1$.
A $k$-separating partition $(X,E-X)$ is a \emph{$k$-separation} if $|X| \ge k$ and $|E-X|\ge k$.
A $k$-separating set $X$, a $k$-separating partition $(X,E-X)$ or a $k$-separation $(X,E-X)$ is \emph{exact} if $\lambda_M(X) = k-1$.
The matroid $M$ is \emph{$n$-connected} if, for all $k < n$, it has no $k$-separations.
When a matroid is $2$-connected, we simply say it is \emph{connected}.

The connectivity functions of a matroid and its dual are equal; that is,
$\lambda_M(X) = \lambda_{M^*}(X)$.  In fact, it is easily shown that
\begin{align*}
  \lambda_M(X) = r(X) + r^*(X) - |X|.
\end{align*}
\subsection*{Spike-like $3$-separators}
Let $M$ be a matroid with ground set $E$.
We say that a $4$-element set $Q \subseteq E$ is a \emph{quad} if it is both a circuit and a cocircuit of $M$.

\begin{definition}
  \label{def-spike-like}
  Let $P \subseteq E$ be an exactly $3$-separating set of $M$.
  If there exists a partition $\{L_1,\dotsc,L_t\}$ of $P$ with $t\geq 3$ such that 
  \begin{enumerate}[label=\rm(\alph*)]
    \item $|L_i|=2$ for each $i\in\{1,\dotsc,t\}$, and
    \item $L_i\cup L_j$ is a quad for all distinct $i,j\in\{1,\dotsc,t\}$, 
  \end{enumerate}
  then $P$ is a \emph{\spikelike} of $M$.
\end{definition}

To illustrate the necessity for outcome (iii) of \cref{maintheorem} we describe the construction of a matroid that satisfies neither (i) nor (ii) of the theorem.
Let $F_7$ be a copy of the Fano matroid with a triangle $\{x,y,z\}$.
Let $F_7'$ be the matroid obtained from $F_7$ by adding elements $y'$ and $z'$ in parallel with $y$ and $z$ respectively, and relabelling the element~$x$ as $t$.
Now let $S$ be a spike with tip $t$, where $r(S) \ge 4$, and let $T=\{t,y',z'\}$ be a leg of $S$.  Let $M = P_T(F_7',S) \ba T$, the generalised parallel connection of $S$ and $F_7'$ along $T$ with the elements $T$ removed.  Then $M$ has no $F_7$-detachable pairs.
Alternatively,
let $F_7''$ be the matroid obtained from $F_7$ by adding elements $y'$ and $z'$ in parallel with $y$ and $z$ respectively, and freely adding the element~$t$ on the line spanned by $\{x,y,z\}$.
Then, similarly, $P_T(F_7'',S) \ba T$ has no $F_7$-detachable pairs.

A geometric illustration of a \spikelike\ is given in \cref{spikelikefig1}.

We will see three more \psep s, and how they can give rise to matroids without any $N$-detachable pairs, in \cref{sec-problematic}.

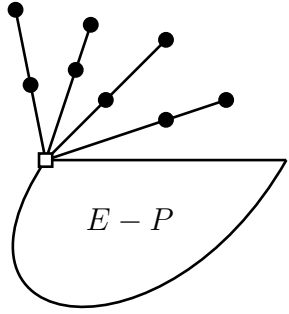
\begin{figure}
  \centering
  \begin{tikzpicture}[rotate=90,scale=0.8,line width=1pt]
    \tikzset{VertexStyle/.append style = {minimum height=5,minimum width=5}}
    \clip (-2.5,-6) rectangle (3.0,2);
    \node at (-1,-1.4) {\large$E-P$};
    \draw (0,0) .. controls (-3,2) and (-3.5,-2) .. (0,-4);
    \draw (0,0) -- (2.5,0.5);
    \draw (0,0) -- (2.25,-0.75);
    \draw (0,0) -- (2,-2);
    \draw (0,0) -- (1,-3);

    \SetVertexNoLabel
    \Vertex[x=1.25,y=0.25,LabelOut=true,L=$q_3$,Lpos=180]{c1}
    \Vertex[x=2.25,y=-0.75,LabelOut=true,L=$q_2$,Lpos=90]{c2}
    \Vertex[x=2.5,y=0.5,LabelOut=true,L=$q_1$,Lpos=180]{c3}
    \Vertex[x=1.5,y=-0.5,LabelOut=true,L=$q_4$,Lpos=135]{c4}

    \Vertex[x=1,y=-1,LabelOut=true,L=$q_1$,Lpos=180]{c5}
    \Vertex[x=2,y=-2,LabelOut=true,L=$q_1$,Lpos=180]{c6}

    \Vertex[x=1,y=-3,LabelOut=true,L=$q_4$,Lpos=135]{c7}
    \Vertex[x=0.67,y=-2,LabelOut=true,L=$q_4$,Lpos=135]{c8}

    \draw (0,0) -- (0,-4);

    \SetVertexNoLabel
    \tikzset{VertexStyle/.append style = {shape=rectangle,fill=white}}
    \Vertex[x=0,y=0]{a1}
  \end{tikzpicture}
  \caption{An example of a \spikelike\ in a matroid with rank $r(E-P)+3$}
  \label{spikelikefig1}
\end{figure}

\subsection*{More connectivity}
The next lemma is a consequence of the easily verified fact that the connectivity function is submodular.  We write ``by uncrossing'' to refer to an application of this lemma.

\begin{lemma}
  \label{onetrick}
  Let $M$ be a $3$-connected matroid, and let $X$ and $Y$ be $3$-separating subsets of $E(M)$.
  \begin{enumerate}
    \item If $|X \cap Y| \ge 2$, then $X \cup Y$ is $3$-separating.
    \item If $|E(M) - (X \cup Y)| \ge 2$, then $X \cap Y$ is $3$-separating.
  \end{enumerate}
\end{lemma}

The following lemmas are well known.

\begin{lemma}
  \label{swapSepSides}
  Let $e$ be an element of a matroid $M$, and let $X$ and $Y$ be disjoint sets whose union is $E(M) - \{e\}$.  Then $e \in \cl(X)$ if and only if $e \notin \cl^{*}(Y)$.
\end{lemma}

\begin{lemma}
  \label{extendSep}
  Let $X$ be an exactly $3$-separating set in a $3$-connected matroid $M$, 
  and suppose that $e \in E(M) - X$.  Then $X \cup e$ is $3$-separating if and only if $e \in \clstar(X)$.
\end{lemma}

\begin{lemma}
  \label{exactSeps}
  Let $(X, Y)$ be an exactly $3$-separating partition of a $3$-connected matroid $M$. Suppose $|X| \ge 3$ and $x \in X$. Then
  $x \in \clstar(X-x)$.
\end{lemma}

\begin{lemma}
  \label{gutses}
  Let $(X, Y)$ be an exactly $3$-separating partition of a $3$-connected matroid $M$, with $|X| \ge 3$ and $x \in X$. Then
  $(X-x, Y \cup x)$ is exactly $3$-separating if and only if $x$ is in 
  one of $\cl(X-x) \cap \cl(Y)$ and $\cocl(X-x) \cap \cocl(Y)$.
\end{lemma}

If $(X, Y)$ and $(X-x, x \cup Y)$ are exactly $3$-separating partitions in a $3$-connected matroid, then we say $x$ is a \emph{guts element} if $x \in \cl(X-x) \cap \cl(Y)$, and $x$ is a \emph{coguts element} if $x \in \cocl(X-x) \cap \cocl(Y)$.
We also say $x$ is \emph{in the guts of $(X,Y)$} or $x$ is \emph{in the coguts of $(X,Y)$}, respectively.

\begin{lemma}
  \label{gutsstayguts}
  Let $(X,Y)$ be a $3$-separation in a $3$-connected matroid. 
  Then $\cl(X) \cap \cocl(X) \cap Y = \emptyset$.
\end{lemma}

A $k$-separation $(X, E-X)$ of a matroid $M$ with ground set $E$ is \emph{vertical} if $r(X) \ge k$ and $r(E-X) \ge k$.
We also say a partition $(X, \{z\}, Y)$ of $E$ is a \emph{vertical $3$-separation} when $(X \cup \{z\}, Y)$ and $(X, Y \cup \{z\})$ are both vertical $3$-separations and $z \in \cl(X) \cap \cl(Y)$.
Note that, given a vertical $3$-separation $(X,Y)$ and some $z \in Y$, if $z \in \cl(X)$, then $(X,\{z\},Y)$ is a vertical $3$-separation, by \cref{extendSep,exactSeps}.

A vertical $k$-separation in $M^*$ is known as a cyclic $k$-separation in $M$.
More specifically, a $k$-separation $(X, E-X)$ of $M$ is \emph{cyclic} if $r^*(X) \ge k$ and $r^*(E-X) \ge k$; or, equivalently, if $X$ and $E-X$ contain circuits.
We also say that a partition $(X, \{z\}, Y)$ of $E$ is a \emph{cyclic $3$-separation} if $(X, \{z\}, Y)$ is a vertical $3$-separation in $M^*$.

We say that a partition $(X_1,X_2,\dotsc,X_m)$ of $E(M)$ is a \emph{path of $3$-separations} if $(X_1 \cup \dotsm \cup X_i, X_{i+1} \cup \dotsm \cup X_m)$ is a $3$-separation for each $i \in \seq{m-1}$.
Observe that a vertical, or cyclic, $3$-separation $(X, \{z\},Y)$ is an instance of a path of $3$-separations.

A proof of the following is in \cite{stabilizers}. We use this lemma, and its dual, freely without reference.

\begin{lemma}
\label{openVertSep}
Let $M$ be a $3$-connected matroid and let $z \in E(M)$.  
The following are equivalent:
\begin{enumerate}
    \item $M$ has a vertical $3$-separation $(X, \{z\}, Y)$.\label{vsi}
    \item $\si(M/z)$ is not $3$-connected.\label{vsii}
\end{enumerate}
\end{lemma}



A \emph{segment} in a matroid $M$ is a subset $S$ of $E(M)$ such that $M|S \cong U_{2,k}$ for some $k \ge 3$, while a \emph{cosegment} of $M$ is a segment of $M^*$.

\begin{lemma}
  \label{rank2Remove}
  Let $M$ be a $3$-connected matroid and let $S$ be a 
  segment
  with at least four elements.  If $s \in S$, then $M \ba s$ is $3$-connected.
\end{lemma}

The next two lemmas will be referred to by name. 

\begin{lemma}[Bixby's Lemma~\cite{bixby}]
  \label{bixbyL}
  Let $e$ be an element of a $3$-connected matroid $M$.
  Then either $M/e$ is $3$-connected up to parallel pairs, or $M \ba e$ is $3$-connected up to series pairs.
\end{lemma}

\begin{lemma}[Tutte's Triangle Lemma~\cite{tutte1966}]
  \label{ttL}
  Let $\{a,b,c\}$ be a triangle in a $3$-connected matroid $M$. If neither $M \ba a$ nor $M \ba b$ is $3$-connected, then $M$ has a triad which contains $a$ and exactly one element from $\{b,c\}$.
\end{lemma}




A proof of the following is in \cite{stabilizers}.  

\begin{lemma}
  \label{r3cocircsi}
  Let $C^*$ be a rank-$3$ cocircuit of a $3$-connected matroid $M$.
If $x \in C^*$ has the property that $\cl_M(C^*)-x$ contains a triangle of $M/x$, then $\si(M/x)$ is $3$-connected.
\end{lemma}

Proofs of the following two lemmas appear in \cite{bs2014}.

\begin{lemma}
  \label{r3cocirc}
  Let $M$ be a $3$-connected matroid with $r(M) \ge 4$.
  Suppose that $C^*$ is a rank-$3$ cocircuit of $M$.
  If there exists some $x \in C^*$ such that $x \in \cl(C^*-x)$, then $\co(M \ba x)$ is $3$-connected.
\end{lemma}

\begin{lemma}
  \label{presingle}
  Let $(X,Y)$ be a $3$-separation of a $3$-connected matroid $M$. If $X\cap \cl(Y)\neq \emptyset$ and $X\cap \cl^*(Y)\neq \emptyset$, then $|X\cap \cl(Y)|=1$ and $|X\cap \cl^*(Y)|=1$.
\end{lemma}

Suppose $M$ is a $3$-connected matroid, there is an element $d \in E(M)$ such that $M \ba d$ is $3$-connected, and $X \subseteq E(M\ba d)$ is exactly $3$-separating in $M \ba d$.
We say that $d$ \emph{blocks} $X$ if $X$ is not $3$-separating in $M$, and $d$ \emph{fully blocks} $X$ if neither $X$ nor $X \cup d$ is $3$-separating in $M$.
If $d$ blocks $X$, then $d \notin \cl(E(M\ba d)-X)$, so $d \in \cocl(X)$ by \cref{swapSepSides}.
It is easily shown that $d$ fully blocks $X$ if and only if $d \notin \cl(X) \cup \cl(E(M \ba d)-X)$.

\subsection*{Full closure}
A set $X$ in a matroid $M$ is {\em fully closed} if it is closed and coclosed; that is, $\cl(X)=X=\cl^*(X)$.
The {\em full closure} of a set $X$, denoted $\fcl(X)$, is the intersection of all fully closed sets that contain $X$.
It is easily seen that the full closure is a well-defined closure operator, and that one way of obtaining the full closure of a set $X$ is to take the closure of $X$, then the coclosure of the result, and repeat until neither the closure nor coclosure introduces new elements.
We frequently use the following straightforward lemma.

\begin{lemma}
  \label{fcllemma}
  Let $(X,Y)$ be a $2$-separation in a connected matroid $M$ where $M$ contains no series or parallel pairs.
  Then $(\fcl(X),Y-\fcl(X))$ is also a $2$-separation of $M$.
\end{lemma}

\subsection*{Fans}
Let $M$ be a $3$-connected matroid.
A subset $F$ of $E(M)$ having at least three elements is a \emph{fan} if there is an ordering $(f_1, f_2, \dotsc, f_k)$ of the elements of $F$ such that
\begin{enumerate}[label=\rm(\alph*)]
  \item $\{f_1,f_2,f_3\}$ is either a triangle or a triad, and\label{fani}
  \item for all $i \in \seq{k-3}$, if $\{f_i, f_{i+1}, f_{i+2}\}$ is a triangle, then $\{f_{i+1}, f_{i+2}, f_{i+3}\}$ is a triad, while if $\{f_i, f_{i+1}, f_{i+2}\}$ is a triad, then $\{f_{i+1}, f_{i+2}, f_{i+3}\}$ is a triangle.\label{fanii}
\end{enumerate}
An ordering of $F$ satisfying \ref{fani} and \ref{fanii} is a \emph{fan ordering} of $F$.
If $F$ has a fan ordering $(f_1, f_2, \dotsc, f_k)$ where $k \geq 4$, then $f_1$ and $f_k$ are the \emph{ends} of $F$, and $f_2, f_3, \dotsc, f_{k-1}$ are the \emph{internal elements} of $F$.
A fan ordering is unique, up to reversal, when $k \ge 5$. 

Let $F$ be a fan with ordering $(f_1, f_2, \dotsc, f_k)$ where $k \geq 4$,
and let $i \in \seq{k}$ if $k \geq 5$, or $i \in \{1,4\}$ if $k=4$.
An element $f_i$ is a \emph{spoke element} of $F$ if $\{f_1, f_2, f_3\}$ is a triangle and $i$ is odd, or if $\{f_1, f_2, f_3\}$ is a triad and $i$ is even; otherwise $f_i$ is a \emph{rim element}.

The next lemma follows easily from Bixby's Lemma.

\begin{lemma}
  \label{fanEnds}
  Let $M$ be a $3$-connected matroid that is not a wheel or a whirl. 
  Suppose $M$ has a fan~$F$ of at least four elements, and let $f$ be an end of $F$.
  \begin{enumerate}
    \item If $f$ is a spoke element, then $\co(M\ba f)$ is $3$-connected and $\si(M/f)$ is not $3$-connected.
    \item If $f$ is a rim element, then $\si(M/f)$ is $3$-connected and $\co(M \backslash f)$ is not $3$-connected.
  \end{enumerate}
\end{lemma}

A fan $F$ is \emph{maximal} if it is not properly contained in any other fan.  Oxley and Wu~\cite[Lemma~1.5]{ow2000} proved the following 
result concerning the ends of a maximal fan.

\begin{lemma}
  \label{fanEndsStrong}
  Let $M$ be a $3$-connected matroid 
  that is not a wheel or a whirl.
  Suppose $M$ has a maximal fan~$F$ of at least four elements, and let $f$ be an end of $F$.
  \begin{enumerate}
    \item If $f$ is a spoke element, then $M\ba f$ is $3$-connected.
    \item If $f$ is a rim element, then $M/f$ is $3$-connected.
  \end{enumerate}
\end{lemma}

\subsection*{Retaining an $N$-minor}

Let $M$ and $N$ be matroids.
Throughout, when we say that $M$ \emph{has an $N$-minor}, we mean that $M$ has an isomorphic copy of $N$ as a minor.
Let $X \subseteq E(M)$.  To simplify exposition, we say $M$ \emph{has an $N$-minor with} $|X \cap E(N)| \le 1$, for example, to mean that $M$ has an isomorphic copy $N'$ of $N$ as a minor such that $|X \cap E(N')| \le 1$.

For a matroid $M$ with a minor $N$, we say an element $e \in E(M)$ is \emph{$N$-contractible} if $M/e$ has an $N$-minor, and $e$ is \emph{$N$-deletable} if $M \ba e$ has an $N$-minor.
We also say a set $X \subseteq E(M)$ is \emph{$N$-contractible} if $M/X$ has an $N$-minor, and $X$ is \emph{$N$-deletable} if $M\ba X$ has an $N$-minor.
An element $e \in E(M)$ is \emph{doubly $N$-labelled} if both $M/e$ and $M \ba e$ have $N$-minors.

The next lemma has a straightforward proof.

\begin{lemma}
  \label{m2.7}
  Let $(X,Y)$ be a $2$-separation of a connected matroid $M$ and let $N$ be a $3$-connected minor of $M$.  Then $\{X, Y\}$ has a member $U$ such that $|U \cap E(N)| \leq 1$.  Moreover, if $u \in U$, then
  \begin{enumerate}
    \item $M/u$ has an $N$-minor if $M/u$ is connected, and
    \item $M \backslash u$ has an $N$-minor if $M \backslash u$ is connected.
  \end{enumerate}
\end{lemma}

The dual of the following is proved in \cite{bcosw2019,bs2014}.

%
\begin{lemma}
    \label{doublylabel}
Let $N$ be a $3$-connected minor of a $3$-connected matroid $M$. Let $(X, \{z\}, Y)$ be a cyclic $3$-separation of $M$ such that $M\ba z$ has an $N$-minor with $|X \cap E(N)| \le 1$. Let $X' = X-\cocl(Y)$
and $Y' = \cocl(Y) - z$.
Then
\begin{enumerate}
\item each element of $X'$ is $N$-deletable; and
\item at most one element of $\cocl(X)-z$ is not $N$-contractible, and if such an element~$x$ exists, then $x \in X' \cap \cl(Y')$ and $z \in \cocl(X' - x)$.
\end{enumerate}
\end{lemma}

Suppose $C$ and $D$ are disjoint subsets of $E(M)$ such that $M/C \ba D \cong N$.
We call the ordered pair $(C,D)$ an \emph{$N$-labelling of $M$}, and say that each $c \in C$ is \emph{$N$-labelled for contraction}, and each $d \in D$ is \emph{$N$-labelled for deletion}. 
We also say a set $C' \subseteq C$ is \emph{$N$-labelled for contraction}, and $D' \subseteq D$ is \emph{$N$-labelled for deletion}.
An element $e \in C \cup D$ or a set $X \subseteq C \cup D$
is \emph{$N$-labelled for removal}.

Let $(C,D)$ be an $N$-labelling of $M$, and let $c \in C$, $d \in D$, and $e \in E(M)-(C \cup D)$.
Then, we say that the ordered pair
$((C-c) \cup d,(D-d)\cup c)$ is 
obtained from $(C,D)$ by \emph{switching the $N$-labels on $c$ and $d$}.
Similarly, 
$((C-c) \cup e,D)$ (or $(C,(D-d) \cup e)$, respectively) is 
obtained from $(C,D)$ by \emph{switching the $N$-labels on $c$} (respectively, $d$) \emph{and $e$}.

The following straightforward lemma, 
which gives a sufficient condition for retaining a valid $N$-labelling after an $N$-label switch,
will be used freely.

\begin{lemma}
  \label{freeswitch}
  Let $M$ be a $3$-connected matroid, let $N$ be a $3$-connected minor of $M$ with $|E(N)| \ge 4$, and let $(C,D)$ be an $N$-labelling of $M$. 
  Suppose $\{d,e\}$ is a parallel pair in $M/c$, for some $c \in C$.
  Let $(C',D')$ be obtained from $(C,D)$ by switching the $N$-labels on $d$ and $e$; then $(C',D')$ is an $N$-labelling.
\end{lemma}

\subsection*{Delta-wye exchange}

Let $M$ be a matroid with a triangle $\Delta=\{a,b,c\}$.
Consider a copy of $M(K_4)$ having $\Delta$ as a triangle with $\{a',b',c'\}$ as the complementary triad labelled such that $\{a,b',c'\}$, $\{a',b,c'\}$ and $\{a',b',c\}$ are triangles.
Let $P_{\Delta}(M,M(K_4))$ denote the generalised parallel connection of $M$ with this copy of $M(K_4)$ along the triangle $\Delta$.
Let $M'$ be the matroid $P_{\Delta}(M,M(K_4))\backslash\Delta$ where the elements $a'$, $b'$ and $c'$ are relabelled as $a$, $b$ and $c$ respectively.
This matroid $M'$ 
is said to be obtained from $M$ by a \emph{\dY} on the triangle~$\Delta$.
Dually, a matroid $M''$ is obtained from $M$ by a \emph{\Yd} on the triad $\{a,b,c\}$
if $(M'')^*$ is obtained from $M^*$ by a \dY\ on $\{a,b,c\}$.

\section{Triangles and triads}
\label{sectris}

Let $M$ be a $3$-connected matroid and let $N$ be a $3$-connected minor of $M$.
If, for a triangle~$T$ and for all distinct $a,b \in T$, none of $M/a/b$, $M/a\ba b$, $M\ba a/b$, and $M\ba a\ba b$ have an $N$-minor, then $T$ is an \emph{\unfortunate\ triangle}.
Similarly, a triad $T^*$ of $M$ is an \emph{\unfortunate\ triad} if, for all distinct $a,b \in T^*$, none of $M/a/b$, $M/a\ba b$, $M\ba a/b$, and $M\ba a\ba b$ have an $N$-minor.
In this section, we show that if $M$ has a triangle or triad that is not $N$-grounded, then either $M$ or $M'$, which can be obtained from $M$ by performing a $\Delta$-$Y$ or $Y$-$\Delta$ exchange, has an $N$-detachable pair.

When $|E(N)| \ge 4$, no element of an \unfortunate\ triangle is $N$-contractible.  As we use this straightforward fact frequently, we state it as a lemma below.

\begin{lemma}
  \label{basicunfortunate}
  Let $M$ be a $3$-connected matroid with a $3$-connected minor $N$ where 
$|E(N)| \ge 4$.
If $T$ is an \unfortunate\ triangle of $M$ with $x \in T$, then $M/x$ does not have an $N$-minor.
\end{lemma}
\begin{proof}
  Let $T = \{x,y,z\}$.  Since $\{y,z\}$ is a parallel pair in $M/x$, and $|E(N)| \ge 4$, if $M/x$ has an $N$-minor, then $M /x \ba y$ has an $N$-minor.  Thus $T$ is not \unfortunate; a contradiction.
\end{proof}

We now prove the main result of this section.
Subject to this \lcnamecref{unfortunatetri}, we can then focus on the case where every triangle or triad of $M$ is \unfortunate.

\begin{theorem}
  \label{unfortunatetri}
  Let $M$ be a $3$-connected matroid, 
  and let $N$ be a $3$-connected minor of $M$ with $|E(N)| \ge 4$, where $|E(M)|-|E(N)| \ge 5$.
  Then either
  \begin{enumerate}
    \item $M$ has an $N$-detachable pair, or\label{ut1}
    \item there is a matroid $M'$ obtained by performing a single $\Delta$-$Y$ or $Y$-$\Delta$ exchange on $M$ such that $M'$ has 
      an $N$-detachable pair, or\label{ut2}
    \item each triangle or triad of $M$ is \unfortunate.\label{ut3}
  \end{enumerate}
\end{theorem}
\begin{proof}
  Suppose $M$ has a triangle or triad $T$ that is not \unfortunate.
  First, suppose that $M$ is a wheel or a whirl.
  By taking the dual, if necessary, we may assume that $T$ is a triangle.
  Let $T = \{x,y,z\}$ where $y$ is a rim element and $x$ and $z$ are spoke elements with respect to a fan ordering of $E(M)$.
  Since $T$ is not \unfortunate, it follows that either $M \ba x$ or $M \ba z$ has an $N$-minor.
  If $M$ is a wheel (respectively, a whirl), then $M/y \ba z$ is a wheel (respectively, a whirl) of rank $r(M)-1$.  In particular, $M/y \ba z$ is $3$-connected since $|E(M)| > 6$.
  Let $M'$ be the matroid obtained from $M$ by performing a \dY\ on $T$.  Then $M/y \ba z \cong M'/z/x$.
  As $M\ba x$ or $M \ba z$ has an $N$-minor, $N$ is a minor of a wheel or a whirl of rank $r(M)-1$, so $M'/z/x$ has an $N$-minor, and $\{x,z\}$ is an $N$-detachable pair of $M'$, satisfying \cref{ut2}.

  \smallskip

  Now, suppose $T$ is contained in a maximal fan~$F$ of size at least five.
  We start by proving the following claim:

  \begin{sublemma}
    \label{dcpair}
    Suppose there are distinct elements $c \in E(M)$ and $d \in F$ such that $M/c\ba d$ is $3$-connected and has an $N$-minor. 
    Then \cref{ut2} holds.
  \end{sublemma}
  \begin{slproof}
    Since $M\ba d$ is $3$-connected, \cref{fanEnds} implies that if $d$ is an end of $F$, it is a spoke element.
    Now $d$ is either an internal element or a spoke of $F$, so it is contained in a triangle $T_1$.  Let $M'$ be the matroid obtained from $M$ by performing a \dY\ on $T_1$.  Then $M \ba d$ is isomorphic to $M' / d$.  Hence $M' / d / c$ is $3$-connected and has an $N$-minor, as required.
  \end{slproof}

  By \cref{dcpair} and its dual, we can now look for a pair of elements, at least one of which is in $F$, whose removal in any way preserves $3$-connectivity and an $N$-minor.
  \Cref{fanEndsStrong} provides one candidate element for removal; to find the second, we require that the resulting matroid, after the element is removed, is not a wheel or a whirl.

  \begin{sublemma}
    \label{nowheel}
    The triangle or triad $T$ is contained in a maximal fan~$F'$ with ordering $(x_1,x_2,\dotsc,x_\ell)$, for $\ell \ge 5$, such that, up to duality, $\{x_1,x_2,x_3\}$ is a triangle, and $M \ba x_1$ is $3$-connected and not a wheel or a whirl.
  \end{sublemma}
  \begin{slproof}
    We have that $T$ is contained in a maximal fan~$F$ of size at least five.
    We may assume, by reversing the ordering if necessary, that $T \subseteq F-x_\ell$, and, by duality, that $x_1$ is a spoke element of $F$, so $\{x_1,x_2,x_3\}$ is a triangle.
    Then, by \cref{fanEndsStrong}, $M \ba x_1$ is $3$-connected.

    Towards a contradiction, suppose $M \ba x_1$ is a wheel or a whirl.
    Then $x_2$ is in a triangle of $M\ba x_1$ that meets $x_3$ or $x_4$, by orthogonality with the triad $\{x_2,x_3,x_4\}$ of $M \ba x_1$.
    If $\{x_2,x_3\}$ is contained in a triangle of $M \ba x_1$, then $\{x_1,x_2,x_3\}$ is contained in a $4$-element segment of $M$ that intersects the triad $\{x_2,x_3,x_4\}$ in two elements, which contradicts orthogonality.
    So $M \ba x_1$ has a triangle $\{x_2,x_4,q\}$, say, where $q \in E(M \ba x_1) - x_3$.

  Suppose $|F| > 6$.  Then $\{x_4,x_5,x_6\}$ is a triad, 
  and, by \cite[Lemma~3.4]{ow2000}, the only triangle of $M$ containing $x_4$ is $\{x_3,x_4,x_5\}$.
  Since $\{x_2,x_4,q\}$ is also a triangle of $M$, this is a contradiction.
  So $|F|=5$.

  Now $(x_1,x_3,x_2,x_4,q)$ is a fan ordering of $M$, and this fan contains $T$.
  It follows from orthogonality that $\{x_4,q\}$ is not contained in a triad, so this fan ordering extends to a maximal fan~$F'$ where $q$ is an end.
  As $M \ba q$ is $3$-connected by \cref{fanEndsStrong}, if $M \ba q$ is not a wheel or a whirl, then \cref{nowheel} holds for the fan~$F'$.

  So we may assume that $M \ba q$ is a wheel or a whirl.
  Now $(x_1,x_2,x_3,x_4,x_5)$ is a fan ordering in $M \ba q$ that extends to a fan ordering $(x_1,x_2,\dotsc,x_\ell)$ of $E(M \ba q)$.
  Observe that $\ell \ge 8$ and $\ell$ is even.
  In $M \ba x_1$, there is a fan with ordering $(q,x_2,x_4,x_3,x_5)$ that extends to a fan ordering of $E(M \ba x_1)$.
  So there is a triad containing $\{x_3,x_5\}$, and it meets $\{x_6,x_7\}$ by orthogonality, but if it contains $x_6$, then $\{x_3,x_4,x_5,x_6\}$ is a cosegment that intersects the triangle $\{x_5,x_6,x_7\}$ in two elements; a contradiction.
  So $\{x_3,x_5,x_7\}$ is a triad.
  If $\ell > 8$, then this triad intersects the triangle $\{x_7,x_8,x_9\}$ in a single element; a contradiction.
  So $|E(M)|=9$, and hence $r(M)=4$.
  It now follows that $q$ is in a triangle $\{q,x_6,x_8\}$.
  By circuit elimination, $\{x_2,x_4,x_6,x_8\}$ contains a circuit.
  As this set does not contain a triangle, $\{x_2,x_4,x_6,x_8\}$ is a circuit, so $M \ba q$ is a wheel.
  Since $\{x_2,x_4,q\}$ and $\{x_6,x_8,q\}$ are circuits of $M$, it follows that $M$ is binary.
  So $M$ has no $U_{2,4}$-minor, in which case $|E(N)| \ge 5$, and $|E(M)| \ge 10$; a contradiction.
  \end{slproof}

  Let $F_1$ be the fan $F'$ of \cref{nowheel} with ordering $(x_1,\dotsc,x_\ell)$.
  Now $M \ba x_1$ is $3$-connected, and is neither a wheel nor a whirl.

  \begin{sublemma}
    \label{justdone}
    There is an $N$-labelling such that $x_1$ is $N$-labelled for deletion, and either $x_2$ or $x_3$ is $N$-labelled for contraction.
  \end{sublemma}
  \begin{slproof}
    First, observe that if either $x_2$ or $x_3$ is $N$-labelled for contraction, then, since $\{x_1,x_2,x_3\}$ is a triangle and $|E(N)| \ge 4$, it follows that
    $x_1$ is $N$-labelled for deletion up to an $N$-label switch with $x_3$ or $x_2$ respectively, using \cref{freeswitch}.
    So it suffices to show that either $x_2$ or $x_3$ is $N$-labelled for contraction.

    Since $F_1$ contains the triangle or triad $T$ that is not \unfortunate, there is an internal element $x_j$ of $F_1$ that is $N$-labelled for removal.
    Suppose $x_2$ is $N$-labelled for deletion.
    Then $\{x_3,x_4\}$ is a series pair in $M \ba x_2$.
    It follows that, after possibly performing an $N$-label switch on $x_3$ and $x_4$, the element $x_3$ is $N$-labelled for contraction.

    Similarly, if $x_j$ is $N$-labelled for deletion for some $j \ge 3$, then, as $\{x_{j-1},x_j\}$ is contained in a triad, $x_{j-1}$ is $N$-labelled for contraction, up to a possible $N$-label switch.
    Likewise, if $x_j$ is $N$-labelled for contraction, for some $j > 3$, then, there is a triangle containing $\{x_{j-1},x_j\}$; after a possible $N$-label switch, $x_{j-1}$ is $N$-labelled for deletion.
    By repeating this process, we obtain an $N$-labelling where either $x_2$ or $x_3$ is $N$-labelled for contraction, as required.
    This proves the claim.
  \end{slproof}

  Consider the matroid $M \ba x_1$.  By \cref{justdone}, this matroid has an $N$-labelling where either $x_2$ or $x_3$ is $N$-labelled for contraction.
  The set $F_1 - x_1$ is a $4$-element fan 
  that is contained in a maximal fan $F_2$, with ordering $(y_1,y_2,\dotsc,y_t)$, for some $t \ge 4$.
  If $x_2$ is $N$-labelled for contraction and $x_2$ is an end of $F_2$, then, as $x_2$ is a rim, the matroid $M \ba x_1 / x_2$ is $3$-connected by \cref{fanEndsStrong}, and \cref{ut2} holds by \cref{dcpair}.

  So we may assume that either $x_3$ is $N$-labelled for contraction, or $x_2$ is not an end of $F_2$.  In either case, $F_2$ has an internal element that is $N$-labelled for contraction.
  By \cref{fanEndsStrong}, either $y_1$ is a spoke and $M \ba x_1 \ba y_1$ is $3$-connected, or $y_1$ is a rim and $M \ba x_1/y_1$ is $3$-connected.
  Using a similar argument as in \cref{justdone}, we can iteratively switch $N$-labels so that $y_1$ is $N$-labelled for deletion if it is a spoke, or $N$-labelled for contraction if it is a rim.
  It follows that \cref{ut2} holds.

  \medskip
  Now suppose $T$ is contained in a maximal $4$-element fan $F$.
  Let $(f_1,f_2,f_3,f_4)$ be a fan ordering of $F$ where $\{f_1,f_2,f_3\}$ is a triangle.
  Since $F$ contains $T$, which is not \unfortunate, at least one of $f_2$ and $f_3$, is $N$-labelled for removal.  Up to duality and switching labels on $f_2$ and $f_3$, we may assume that $f_2$ is $N$-labelled for deletion.
  Since $\{f_3,f_4\}$ is a series pair in $M \ba f_2$, we may also assume, up to an $N$-label switch, that $f_4$ is $N$-labelled for contraction.
  Now $M / f_4$ is $3$-connected, by \cref{fanEndsStrong}, and has an $N$-minor.
  Let $M'$ be the matroid obtained by \Yd\ on the triad $\{f_2,f_3,f_4\}$.  Then $M / f_4$ is isomorphic to $M' \ba f_4$.

  Now $\{f_1,f_2,f_3\}$ is a triangle of $M/f_4$ that does not meet a triad, so $M/f_4$ is not a wheel or a whirl.
  Hence, by the Splitter Theorem, there is an element $e \in E(M/f_4)$ such that either $M/f_4/e$ or $M/f_4 \ba e$ is $3$-connected and has an $N$-minor.
  In the latter case, $M/f_4 \ba e$ is isomorphic to $M' \ba f_4 \ba e$, so \cref{ut2} holds.

  \medskip
  Finally, we may assume that $T$ is a triangle that is not contained in a $4$-element fan.
  Let $T= \{a,b,c\}$.
  We claim that, up to relabelling, $M \ba a$ and $M \ba b$ have $N$-minors.
  Indeed, if $c$ is $N$-labelled for contraction, then, since $\{a,b\}$ is a parallel pair in $M/c$, both $M \ba a$ and $M \ba b$ have $N$-minors.
  On the other hand, if $T$ has no elements that are $N$-labelled for contraction, then, as $T$ is not \unfortunate, it has at most one element that is not $N$-labelled for removal, and, by labelling this element $c$, we have that $M\ba a$ and $M\ba b$ have $N$-minors.

  Since there is no triad meeting $T$, Tutte's Triangle Lemma implies that at least one of $M \ba a$ and $M \ba b$ is $3$-connected.
  Without loss of generality, let $M \ba a$ be $3$-connected.
  Now $M \ba a$ has a proper $N$-minor, so if $M \ba a$ is not a wheel or a whirl, then, by the Splitter Theorem, there is some element $x \in E(M \ba a)$ such that $M \ba a \ba x$ or $M \ba a / x$ is $3$-connected and has an $N$-minor.  In the first case, $M$ has an $N$-detachable pair as required, so assume the latter.
  Let $M'$ be the matroid obtained by a $\Delta$-$Y$ exchange on $T$.  Then $M \ba a$ is isomorphic to $M' / a$.
  In particular, $M' / a$ has an $N$-minor.
  Thus $\{a,x\}$ is an $N$-detachable pair in $M'$, satisfying \cref{ut2}.

  It remains to consider the case where $M\ba a$ is a wheel or a whirl.
  Since $M$ has no $4$-element fans, for every triad $T^*$ of $M \ba a$, 
  we have that $T^* \cup a$ is a cocircuit of $M$.
  By orthogonality, $T-a$ 
  has non-empty intersection with each such $T^*$. 
  If a wheel or whirl has rank more than four, then no two elements meet every triad.  
  So $r(M \ba a) \le 4$, and thus $|E(M \ba a)| \leq 8$.
  Thus, in the only remaining case $|E(M)| = 9$ and $|E(N)| = 4$, so $N \cong U_{2,4}$.
  Since $M \ba a$ has an $N$-minor, $M \ba a$ is the rank-$4$ whirl.

  Let $d$ be a spoke of $M \ba a$.  Then it is easily verified that $M \ba d$ is $3$-connected and has an $N$-minor.
  Moreover, $M \ba d$ is not a wheel or a whirl, and $d$ is in distinct triangles $T_1$ and $T_2$ of $M$.
  By the Splitter Theorem, there is some element $x \in E(M \ba d)$ such that $M \ba d \ba x$ or $M \ba d / x$ is $3$-connected and has an $N$-minor.  In the first case, $M$ has an $N$-detachable pair as required.
  In the latter case, observe that $x$ is not contained in either $T_1$ or $T_2$.
  Say $x \notin T_1$.
  Letting $M'$ be the matroid obtained by a $\Delta$-$Y$ exchange on $T_1$, we observe that $M \ba d / x \cong M' /d /x$ is $3$-connected and has an $N$-minor, so \cref{ut2} holds.
\end{proof}

\section{$5$-element planes}
\label{secplanes}

In this section, we show that when $M$ has a $U_{3,5}$ restriction, and there are certain elements whose removal preserves an $N$-minor, then $M$ has an $N$-detachable pair.
For $P \subseteq E(M)$, we say that $P$ is a \emph{$5$-element plane} if $M|P \cong U_{3,5}$.  We also say $P$ is a \emph{$5$-element coplane} if $M^*|P \cong U_{3,5}$.
%
The proofs of the first two lemmas are routine.

\begin{lemma}
  \label{6pointplane}
  Let $M$ be a $3$-connected matroid with $P \subseteq E(M)$ such that $M|P \cong U_{3,5}$.
  Then $M\ba p$ is $3$-connected for each $p \in \cl(P)-P$.
\end{lemma}

\begin{lemma}
  \label{basicplanelemma}
  Let $M$ be a $3$-connected matroid with a set $P$ such that $M|P \cong U_{3,5}$, and $|E(M)| \ge 6$.
  If $P$ contains a triad~$T^*$, then $M \ba p$ is $3$-connected for each $p \in P-T^*$.
\end{lemma}

\begin{lemma}
  \label{planelemma}
  Let $M$ be a $3$-connected matroid with $P \subseteq E(M)$ such that $M|P \cong U_{3,5}$.
  Suppose that $\cl(P)$ contains no triangles and $P$ contains no triads.
  If $M \ba p$ is not $3$-connected for some $p \in P$, then 
  there is a labelling $\{p_1,p_2,p_3,p_4\}$ of $P-p$ such that $M \ba p_i \ba p_j$ is $3$-connected for each $i \in \{1,2\}$ and $j \in \{3,4\}$.
\end{lemma}
  \begin{proof}
    Let $P = \{p,p_1,p_2,p_3,p_4\}$, and suppose $M \ba p$ is not $3$-connected.
    If $p$ is in a triad, then this triad is contained in $P$, by orthogonality; a contradiction.  So $M$ has a cyclic $3$-separation $(A,\{p\},B)$, where $(A,B)$ is a $2$-separation of $M \ba p$.
    Without loss of generality, we may assume that $\{p_1,p_2\} \subseteq A$.
    If $p_3 \in A$ or $p_4 \in A$, then $p \in \cl(A)$, so $(A \cup p, B)$ is $2$-separating in $M$; a contradiction.  So $\{p_3,p_4\} \subseteq B$.
    Let $A' = A - \{p_1,p_2\}$ and $B' = B - \{p_3,p_4\}$.
    Since $A$ and $B$ contain circuits and $\cl(P)$ contains no triangles, $|A'|,|B'|\ge 2$.
    Now,
    $(A', \{p_1,p_2\}, \{p\}, \{p_3,p_4\}, B')$ is a path of $3$-separations of $M$ where $p_1$ and $p_2$, and $p_3$ and $p_4$, are guts elements.
    Again using that $\cl(P)$ contains no triangles, it follows that $r(A'),r(B') \ge 3$.
    Furthermore, each $p_i$ is not in a triad, by orthogonality.
    Thus, by Bixby's Lemma, $M \ba p_i$ is $3$-connected for $i \in \{1,2,3,4\}$; and, moreover,
    $M \ba p_i \ba p_j$ is $3$-connected up to series pairs for $i \in \{1,2\}$ and $j \in \{3,4\}$.
    Suppose that $\{p_i,p_j\}$ is in a $4$-element cocircuit $C^*$ of $M$.  Then $E(M)-C^*$ is closed, so $C^*$ meets $A'$ and $B'$, and contains an element of $P - \{p_i,p_j\}$.  But this implies $|C^*| \ge 5$; a contradiction.
    This proves \cref{planelemma}.
  \end{proof}

  \begin{lemma}
    \label{6pointplane2}
    Let $M$ be a $3$-connected matroid with $P \subseteq E(M)$ such that $M|P \cong U_{3,6}$, and $X \subseteq P$ such that $|X|=4$. 
    Suppose that $\cl(P)$ contains no triangles.
    Then there are distinct elements $x_1,x_2 \in X$ such that $M\ba x_1\ba x_2$ is $3$-connected.
  \end{lemma}
  \begin{proof}
    Pick any distinct $x_1,x_2 \in X$.
    By \cref{6pointplane}, $M \ba x_1$ is $3$-connected, and $M|(P-x_1) \cong U_{3,5}$.
    If $M \ba x_1 \ba x_2$ is $3$-connected, then the \lcnamecref{6pointplane2} holds, so we may assume otherwise.
    Observe that $P$ contains no triads, by orthogonality.
    Now, by \cref{planelemma}, $M \ba x_1 \ba p \ba p'$ is $3$-connected for $p,p' \in P - \{x_1,x_2\}$, where we can choose $p$ and $p'$ such that $p \in X$.
    In particular, $M\ba x_1\ba p$ is $3$-connected for $\{x_1,p\} \subseteq X$, as required.
  \end{proof}

  The following lemma is useful for finding candidates for contraction in a $4$-element cocircuit, particularly in the case where the cocircuit is independent. 

  \begin{lemma}
      \label{r4cocirc}
      Let $M$ be a $3$-connected matroid and let $C^*$ be a $4$-element cocircuit of $M$. If there are distinct elements $c',c'' \in C^*$ such that neither $c'$ nor $c''$ is in a triangle, then
      $M/c$ is $3$-connected for some $c \in C^*$.
  \end{lemma}

\begin{proof}
Let $C^*=\{c_1,c_2,c_3,c_4\}$ and suppose that $c_1$ is one of two elements that is not contained in a triangle.
If $M/c_1$ is not $3$-connected, then $M$ has a vertical $3$-separation $(X,\{c_1\},Y)$.
We may assume that $c_2\in X$ and $c_3,c_4\in Y$.
Suppose that $c_2$ is not in a triangle.
If $X$ is a triad, then by the dual of \cref{r3cocirc}, $M/c_2$ is $3$-connected as required.
If $X$ is not a triad, then either $X$ is a cosegment with at least four elements, or $X$ contains a circuit.
In the first case, $M/c_2$ is $3$-connected by the dual of \cref{rank2Remove}.
In the second case, as $c_2 \in \cocl(Y\cup c_1)$, the circuit contained in $X$ does not contain $c_2$.
Now $(X-c_2,\{c_2\},Y\cup c_1)$ is a cyclic $3$-separation of $M$, so $M/c_2$ is once again $3$-connected, by Bixby's Lemma.
So we may assume that $c_2$ belongs to some triangle $T$.

As $C^*$ is a cocircuit, $T\cap(C^*-c_2)\neq\emptyset$ by orthogonality, so we may assume that $c_3 \in T$ and $c_4$ is the other element of $C^*$ that is not contained in a triangle.  
As $c_2\not\in\cl(Y)$, we have $|T\cap X|=2$, so $(Y- c_3,\{c_1\},X\cup c_3)$ is a vertical $3$-separation of $M$.
Note that $(Y-c_3) \cap C^* = \{c_4\}$.
Again, if $Y- c_3$ is not a triad or a cosegment, then $(Y-\{c_3,c_4\},\{c_4\},X\cup\{c_1,c_3\})$ is a cyclic $3$-separation of $M$, and $M/c_4$ is $3$-connected by Bixby's Lemma.  On the other hand, if $Y- c_3$ is a triad, then $M/c_4$ is $3$-connected by the dual of \cref{r3cocirc}; while if $Y-c_3$ is a cosegment with at least four elements, then $M/c_4$ is $3$-connected by the dual of \cref{rank2Remove}.
\end{proof}  

The next two results show the existence of $N$-detachable pairs when $M$ has a subset $P$ such that $M|P \cong U_{3,5}$.  The first handles the case where $P$ is $3$-separating, whereas the second handles the case where $P$ is not $3$-separating.

\begin{proposition}
  \label{basicplaneupgrade}
  Let $M$ be a $3$-connected matroid with $|E(M)| \ge 9$ and $r(M) \ge 5$, and let $N$ be a $3$-connected minor of $M$ where $|E(N)| \ge 4$ and every triangle or triad of $M$ is \unfortunate.
  Suppose there exists some exactly $3$-separating set $P \subseteq E(M)$ such that $M|P \cong U_{3,5}$, and there are distinct elements $d^*,p \in P$ such that
  \begin{enumerate}[label=\rm(\alph*)]
    \item either $P$ or $P-p$ is a cocircuit, and
    \item $M / d^* / p'$ has an $N$-minor for all $p' \in P-\{d^*,p\}$.
  \end{enumerate}
  Then $M$ has an $N$-detachable pair.
\end{proposition}

\begin{proof}
  First, observe that for any $p' \in P-\{d^*,p\}$, the set
  $P-\{d^*,p'\}$ is contained in a parallel class in $M/d^*/p'$.
  Since $|E(N)| \ge 4$,
  the matroid $M \ba q_1 \ba q_2$ has an $N$-minor for any distinct $q_1,q_2 \in P-\{d^*,p'\}$.
  By an appropriate choice of $p'$, it follows that $M \ba q_1 \ba q_2$ has an $N$-minor for all distinct $q_1,q_2 \in P-d^*$.

  Let $P= \{p_1,p_2,p_3,p_4,p_5\}$, where 
  $p_4=d^*$ and $p_5=p$.
  For each $i \in \seq{4}$, $M/p_i$ has an $N$-minor, so $P$ does not contain an \unfortunate\ triangle.
  Similarly, $M \ba p_i$ has an $N$-minor for each $i \in \{1,2,3,5\}$, so $P$ does not contain an \unfortunate\ triad.
  Suppose $\cl(P)$ contains an \unfortunate\ triangle~$T$.
  Then $T \subseteq \cl(P)-\{p_1,p_2,p_3,p_4\}$.
  Since $M/p_1/p_4$ has an $N$-minor and
  $\cl(P)-\{p_1,p_4\}$ is contained in a parallel class in $M/p_1/p_4$,
  there is an $N$-labelling $(C,D)$ such that $T \subseteq C$; a contradiction.
  So $\cl(P)$ does not contain any triangles.

  By \cref{planelemma}, we may assume that $M \ba p_i$ is $3$-connected for each 
  $i \in \seq{5}$.
  Since $P$ does not contain any triads,
  either $P$ is a cocircuit, or $P$ contains a $4$-element cocircuit.
  Towards a contradiction, we now assume that $M$ does not have an $N$-detachable pair.

  \begin{sublemma}
    \label{claim2}
    For each $i \in \seq{3}$, 
    there exists a cocircuit $\{p_i,p_i',p_5,z_i\}$ of $M$, where
    $p_i' \in P-\{p_i,p_5\}$ and $z_i \in E(M)-P$, and $M/z_i$ is $3$-connected. 
  \end{sublemma}
  \begin{slproof}
    We claim that $\co(M \ba p_5 \ba p_i)$ is $3$-connected for each $i \in \seq{4}$.
    First, suppose that $P-p_5$ is a cocircuit.
    Then, for $i \in \seq{4}$,
    $(P-\{p_i,p_5\},\{p_5\}, E(M)-P)$ is a vertical $3$-separation of $M \ba p_i$.
    Thus, by Bixby's Lemma, $\co(M \ba p_5\ba p_i)$ is $3$-connected.
    Now suppose that $P$ is a cocircuit.
    We will show that $\co(M \ba p_5 \ba p_i)$ is $3$-connected for $i=4$, but the argument is the same for $i \in \seq{3}$.
    Let $(X,Y)$ be a $2$-separation of $M \ba p_5 \ba p_4$.
    We may assume that $\{p_1,p_2\} \subseteq X$.
    Now $\{p_1,p_2,p_3\}$ is a triad of $M \ba p_5 \ba p_4$, so either
    $(X \cup p_3, Y-p_3)$ is a $2$-separation, or $Y$ is a series pair.
    But
    $p_5 \in \cl(X \cup p_3)$, so, in the former case, $(X \cup \{p_3,p_5\}, Y-p_3)$ is a $2$-separation of $M \ba p_4$; a contradiction.
    Thus $Y$ is a series pair, and it follows that
    $\co(M \ba p_5 \ba p_4)$ is $3$-connected.

    Let $i \in \seq{3}$, and
    recall that $M \ba p_i \ba p_5$ has an $N$-minor. 
    Since $\{p_i,p_5\}$ is not an $N$-detachable pair, it follows that $p_5$ is in a triad~$T^*$ of $M \ba p_i$.  
    By orthogonality, $T^*$ contains an element $p_i' \in P-\{p_i,p_5\}$, so let $T^* = \{p_5,p_i',z_i\}$.
    If $P$ is a cocircuit, then $T^* \nsubseteq P$, so $z_i \in E(M)-P$,
    whereas if $P-p_5$ is a cocircuit, then $p_5 \in \cl(E(M)-P)$, so, by orthogonality, $z_i \in E(M)-P$.
    Since $T^*$ is not a triad of $M$, $\{p_i,p_i',p_5,z_i\}$ is a cocircuit. 

    Suppose $M/z_i$ is not $3$-connected.
    If $z_i$ is in a triangle, then, by orthogonality with the cocircuit $\{p_i,p_i',p_5,z_i\}$, this triangle meets $P$; a contradiction.
    So $\si(M/z_i)$ is also not $3$-connected.  Let $(A,\{z_i\},B)$ be a vertical $3$-separation of $M$.
    Without loss of generality, $|A \cap P| \ge 3$, so $(A \cup P, \{z_i\}, B-P)$ is also a vertical $3$-separation, by uncrossing.
    But then $z_i \in \cocl(A \cup P) \cap \cl(B-P)$; a contradiction.
    So $M/z_i$ is $3$-connected.
  \end{slproof}
  
  \begin{sublemma}
    \label{claim3*}
    Suppose, up to relabelling $\{p_1,p_2,p_3\}$,
    that $M$ has a cocircuit $\{p_1,p_2,p_5,z\}$,
    for some $z \in E(M)-P$.
    Then $M$ has an $N$-detachable pair.
  \end{sublemma}
  \begin{slproof}
  Let $(C,D)$ be an $N$-labelling such that $\{p_3,p_4\} \subseteq C$; such an $N$-labelling exists since $M / p_3 / p_4$ has an $N$-minor.

    Since $\{p_1,p_2,p_5\}$ is contained in a parallel class in $M/p_3/p_4$, we may assume, up to switching the $N$-labels on $p_5$ and $p_1$ or $p_2$, that $p_1$ and $p_2$ are $N$-labelled for deletion.
    Moreover, as $\{z,p_5\}$ is a series pair in $M \ba p_1 \ba p_2$, we may also assume, by a possible $N$-label switch on $p_5$ and $z$, that $z$ is $N$-labelled for contraction.
    In particular, $\{z,p_3\}$ is an $N$-contractible pair.

    Since $P$ is exactly $3$-separating and $z \in \cocl(P)$, 
    \cref{gutsstayguts} implies that $z \notin \cl(P)$.
    So $P$ or $P-p_5$ is a rank-$3$ cocircuit in $M/z$.
    By \cref{r3cocircsi}, 
    $\si(M / z /p_3)$ is $3$-connected.
    Now either $M / z /p_3$ is $3$-connected, or $\{z,p_3\}$ is contained in a $4$-element circuit.
    In the former case, $M$ has an $N$-detachable pair. 
    So we may assume that $\{z,p_3\}$ is contained in a $4$-element circuit $C_z$.
  By orthogonality, $C_z$ meets $\{p_1,p_2,p_5\}$; moreover, since $z \notin \cl(P)$, we have $|C_z \cap P| = 2$.
  So $C_z = \{z,p_3,p'',f\}$ where $p'' \in \{p_1,p_2,p_5\}$ and $f \in E(M) - (P \cup z)$.
  
  Note that $p_3$ and $z$ are $N$-labelled for contraction.  Thus, after possibly switching the $N$-labels on $p''$ and $f$, the element $f$ is $N$-labelled for deletion.
  Let $p''' \in \{p_1,p_2\}-p''$,
  and note that $p'''$ is also $N$-labelled for deletion.
  As $(P \cup z, \{f\}, E(M) - (P \cup \{z,f\}))$ is a vertical $3$-separation 
  and
  $f$ is not in an \unfortunate\ triad, 
  the matroid $M \ba f$ is $3$-connected and has an $N$-minor.
  Note that $f \notin \cocl(P)$, so $P$ does not contain any triads in $M \ba f$.
  Thus, by \cref{planelemma}, $M \ba f \ba p'''$ is $3$-connected, so $\{f,p'''\}$ is an $N$-detachable pair. 
  \end{slproof}

  Now, by \cref{claim2,claim3*}, we may assume that
  $\{p_1,p_4,p_5,z_1\}$, $\{p_2,p_4,p_5,z_2\}$, and $\{p_3,p_4,p_5,z_3\}$ are cocircuits of $M$.

  Suppose $z_i = z_j$ for some distinct $i,j \in \seq{3}$.
  Then, by the cocircuit elimination axiom, $\{p_i,p_j,p_4,p_5\}$ contains a cocircuit; in fact, since
  $P$ does not contain any \unfortunate\ triads, this set is a $4$-element cocircuit.
  Since $P$ is not a cocircuit, $P-p_5$ is also a cocircuit, by hypothesis.
  But now $P-p_5$ is $3$-separating and
  $p_5 \in \cl(\{p_1,p_2,p_3,p_4\}) \cap \cocl(\{p_1,p_2,p_3,p_4\})$; a contradiction.
  So $z_i \neq z_j$ for all distinct $i,j \in \seq{3}$.

  For $j \in \{2,3\}$, the partition
  $(P,\{z_1\},\{z_j\},E(M)-(Z \cup \{z_1,z_j\}))$ is a path of $3$-separations where $z_1$ and $z_j$ are coguts elements.
  In particular, $z_j \in \cocl(E(M)-(Z \cup \{z_1,z_j\}))$, so $z_j \notin \cl(P \cup z_1)$.
  We now fix an $N$-labelling such that $p_1$ and $p_5$ are $N$-labelled for deletion and $p_2$ is $N$-labelled for contraction (such an $N$-labelling exists since $M/p_2/p_4$ has an $N$-minor and $\{p_1,p_3,p_5\}$ is contained in a parallel class in this matroid).
  We may also assume that $z_1$ is $N$-labelled for contraction, since
  $\{z_1,p_4\}$ is a series pair in $M \ba p_2 \ba p_4$.
  Recall that $M/z_1$ is $3$-connected. 
  By \cref{r3cocircsi}, $\si(M / z_1 /p_2)$ is $3$-connected. 
  Thus, either $\{z_1,p_2\}$ is an $N$-detachable pair, or
  $\{z_1,p_2\}$ is contained in a $4$-element circuit~$C_1$.
  By orthogonality, $C_1$ meets $\{p_1,p_4,p_5\}$ and $\{p_4,p_5,z_2\}$.
  Since $z_1 \notin \cl(P)$, we have $|C_1 \cap P|=2$.
  If $p_4 \in C_1$ or $p_5 \in C_1$, then $C_1=\{z_1,p_2,p_\ell,z_3\}$ for $\ell \in \{4,5\}$, so $z_3 \in \cl(P \cup z_1)$; a contradiction.
  On the other hand, if $\{p_4,p_5\} \cap C_1 = \emptyset$, then
  $\{p_1,z_2\} \subseteq C_1$, so $C_1=\{p_1,p_2,z_1,z_2\}$ and $z_2 \in \cl(P \cup z_1)$; a contradiction.
  This completes the proof.
\end{proof}

\begin{proposition}
  \label{planeupgrade}
  Let $M$ be a $3$-connected matroid with a $3$-connected matroid~$N$ as a minor, where $|E(N)| \ge 4$ and every triangle or triad of $M$ is \unfortunate.
  Suppose there exists $P \subseteq E(M)$ such that $M|P \cong U_{3,5}$ and $P$ is not $3$-separating,
  and there are elements $d^*,p \in P$ such that
  \begin{enumerate}[label=\rm(\alph*)]
    \item $M / d^*$ is $3$-connected,\label{planeupgradec1}
    \item $M / d^* / p'$ has an $N$-minor for all $p' \in P-\{d^*,p\}$,
      and\label{planeupgradec2}
    \item
      for any $p' \in P-d^*$ and distinct elements $u,v \in \cocl(P-d^*)-P$, either $M \ba p' \ba u$ or $M \ba p' \ba v$ has an $N$-minor.\label{planeupgradec4}
  \end{enumerate}
  Then $M$ has an $N$-detachable pair.
\end{proposition}
\begin{proof}
  Pick $p \in P$ such that $M /d^*/p'$ has an $N$-minor for each $p' \in P-\{d^*,p\}$.
  Since $P-\{d^*,p'\}$ is contained in a parallel class in $M/d^*/p'$ and $|E(N)| \ge 4$,
  the matroid $M \ba q_1 \ba q_2$ has an $N$-minor for any distinct $q_1,q_2 \in P-\{d^*,p'\}$.
  As $p'$ is chosen arbitrarily among $P-\{d^*,p\}$, it follows that $M \ba q_1 \ba q_2$ has an $N$-minor for all distinct $q_1,q_2 \in P-d^*$.

  As $M \ba p'$ has an $N$-minor for each $p' \in P-d^*$, $P$ does not contain an \unfortunate\ triad.
  Suppose $\cl(P)$ contains an \unfortunate\ triangle~$T$.  Then $T$ does not meet $P-\{d^*,p\}$, since $M / p'$ has an $N$-minor for each $p' \in P-\{d^*,p\}$.
  There exist distinct $p', p'' \in P-\{d^*,p\}$ such that $M/p'/p''$ has an $N$-minor, and $T$ is contained in a parallel class in this matroid.  But this contradicts the fact that $T$ is \unfortunate, so $\cl(P)$ does not contain any triangles.

  \begin{sublemma}
    \label{p0}
    If there are distinct elements $q,q',q'' \in P-d^*$ such that neither $\{q,q'\}$ nor $\{q,q''\}$ is contained in a $4$-element cocircuit of $M$, then $M$ has an $N$-detachable pair.
  \end{sublemma}
  \begin{slproof}
    Recall that $\cl(P)$ does not contain an \unfortunate\ triangle and $P$ does not contain an \unfortunate\ triad. 
    Thus, by \cref{planelemma}, either $M$ has an $N$-detachable pair, or $M \ba q$ is $3$-connected.
    By the dual of \cref{r4cocirc}, 
    if there are distinct elements $q'$ and $q''$ in $P-\{d^*,q\}$ that are not contained in a triad of $M \ba q$, then either 
    $\{q,q'\}$ or $\{q,q''\}$ is an $N$-detachable pair.
  \end{slproof}

  \begin{sublemma}
    \label{p1}
    There is a labelling $\{p_1,p_2,p_3,p_4\}$ of $P-d^*$ such that one of the following holds:
    \begin{enumerate}
      \item $\{p_1,p_2,p_3,u\}$ and $\{p_2,p_3,p_4,v\}$ are cocircuits of $M$, with $u,v \in E(M)-P$, or\label{p1a}
      \item $\{p_1,p_2,p_3,u\}$, $\{d^*,p_2,p_4,u_2\}$ and $\{d^*,p_3,p_4,u_3\}$ are cocircuits of $M$, with $u, u_2, u_3 \in E(M)-P$, or\label{p1b}
      \item each of $\{d^*,p_1,p_3\}$, $\{d^*,p_1,p_4\}$, $\{d^*,p_2,p_3\}$, $\{d^*,p_2,p_4\}$, and $\{d^*,p_3,p_4\}$ is contained in a $4$-element cocircuit of $M$.\label{p1c}
    \end{enumerate}
  \end{sublemma}
  \begin{slproof}
    By orthogonality, a $4$-element cocircuit that intersects $P$ must contain at least three elements of $P$; in fact, since $P$ is not $3$-separating, such a cocircuit contains exactly three elements of $P$.

    If there are no cocircuits containing a $3$-element subset of $\{p_1,p_2,p_3,p_4\}$, then by repeated applications of \cref{p0}, it follows that \cref{p1c} holds.
    On the other hand, if there are two cocircuits of $M$ containing distinct $3$-element subsets of $\{p_1,p_2,p_3,p_4\}$, then \cref{p1a} holds.
    So assume that $\{p_1,p_2,p_3,u\}$ is a cocircuit of $M$ for $u \in E(M)-P$, and every other $4$-element cocircuit meeting $P$ contains $d^*$.
    If neither $\{p_2,p_4\}$ nor $\{p_3,p_4\}$ is contained in a $4$-element cocircuit, then $M$ has an $N$-detachable pair by \cref{p0}; so we may assume that $\{p_3,p_4,v\}$ is a $4$-element cocircuit for some $v \in E(M)-P$.
    But by repeating this argument with $\{p_1,p_4\}$ and $\{p_2,p_4\}$, we deduce that 
    $\{p_2,p_4,v'\}$ is a cocircuit for some $v' \in E(M)-P$.
    Since \cref{p1b} holds in this case, this completes the proof.
  \end{slproof}

  Let $u$ and $v$ be elements in $E(M)-P$ contained in distinct $4$-element cocircuits that intersect $P$ in three elements.
  If $u = v$, then $P$ contains a cocircuit by the cocircuit elimination axiom, contradicting the fact that $P$ is not $3$-separating.
  So $u \neq v$.

  \begin{sublemma}
    \label{p3}
    Let $u,v \in E(M)-P$ be distinct elements in $4$-element cocircuits $C^*_u$ and $C^*_v$, respectively, where $C^*_u \subseteq P \cup u$ and $C^*_v \subseteq P \cup v$.
    Then $\si(M/u/v)$ is $3$-connected.
  \end{sublemma}
  \begin{slproof}
    Suppose $(X,Y)$ is a $2$-separation of $M/u/v$ where neither $X$ nor $Y$ is contained in a parallel class.
    We may assume that $|X \cap P| \ge 3$ and that $X$ is closed.
    Thus $P \subseteq X$.
    But $\{u,v\} \subseteq \cocl(P) \subseteq \cocl(X)$, so
    $(X \cup \{u,v\}, Y)$ is a $2$-separation of $M$; a contradiction.
  \end{slproof}

  \begin{sublemma}
    \label{p5}
    Let $C^*_u$ be a $4$-element cocircuit with $u \in C^*_u \subseteq P \cup u$, for $u \in E(M)-P$,
    and let $p' \in P-C^*_u$. Then
    $\si(M / p' / u)$ is $3$-connected.
  \end{sublemma}
  \begin{slproof}
    Suppose $\si(M/p'/u)$ is not $3$-connected, and let $(X,Y)$ be a $2$-separation in $M/p'/u$ where neither $X$ nor $Y$ is a parallel pair.
    We may assume that $|X \cap (P-p')| \ge 2$ and that $X$ is closed.
    Since $r_{M/p'}(P-p')=2$, we have $P-p' \subseteq X$.
    Since $p' \notin C^*_u$, we have $u \in \cocl_{M/p'}(P-p')$, and $(X\cup u,Y)$ is a $2$-separation of $M/p'$; a contradiction.
  \end{slproof}

  \begin{sublemma}
    \label{p6}
    If \cref{p1}\cref{p1a} holds, then $M$ has an $N$-detachable pair.
  \end{sublemma}
  \begin{slproof}
    Let $u$ and $v$ be elements in $E(M)-P$ such that $\{p_1,p_2,p_3,u\}$ and $\{p_2,p_3,p_4,v\}$ are cocircuits of $M$.
    Recall that $M \ba p_2\ba p_3$ has an $N$-minor.
    Let $(C,D)$ be an $N$-labelling such that $\{p_2,p_3\} \subseteq D$.
    Since $\{p_1,u\}$ and $\{p_4,v\}$ are series pairs in $M\ba p_2 \ba p_3$, we may assume that $\{u,v\} \subseteq C$.

    If $M/u/v$ is $3$-connected, then $\{u,v\}$ is an $N$-detachable pair. 
    By \cref{p3}, $\si(M/u/v)$ is $3$-connected.
    Since $u$ and $v$ are $N$-labelled for contraction, each is not in an \unfortunate\ triangle.
    So we may assume there is a $4$-element circuit~$C_{uv}$ of $M$ containing $\{u,v\}$.
    By orthogonality, $C_{uv}$ contains at least one element in $P$.
    Let $C_{uv} = \{u,v,p',z\}$ for some $p' \in P$
    and $z \in E(M)-\{u,v,p'\}$.

    We claim that $z \notin P$.
    Let $Z = E(M) - (P \cup \{u,v\})$.
    Since $\lambda(P)=3$ and $u,v \in \cocl(P)$, we have $r(Z) = r(M)-2$.
    Suppose $z \in P$.
    Then $r(P \cup \{u,v\}) \le 4$, so $(Z, P\cup \{u,v\})$ is a $3$-separation.
    Next we show that $(Z, \{d^*\}, (P -d^*) \cup \{u,v\})$ is a vertical $3$-separation.  Clearly $d^* \in \cl(P-d^*)$.
    If $d^*$ is in a cocircuit containing $v$ and elements of $P-d^*$, then cocircuit elimination with $\{v,p_2,p_3,p_4\}$ implies that $P$ contains a cocircuit; a contradiction.
    So $d^* \notin \cocl((P-d^*) \cup v)$; thus $d^* \in \cl(Z \cup u)$.
    But if $d^* \notin \cl(Z)$, then $u \in \cl(Z \cup d^*)$ by the Mac~Lane-Steinitz exchange property, contradicting that $u \in \cocl(\{p_1,p_2,p_3\})$.
    So $d^* \in \cl(Z)$, and
    $(Z, \{d^*\}, (P -d^*) \cup \{u,v\})$ is a vertical $3$-separation implying that $M/d^*$ is not $3$-connected, contradicting \ref{planeupgradec1}.

    Now $C_{uv} \cap P = \{p'\}$, so $p' \in \{p_2,p_3\}$, by orthogonality.
    Since $\{p',z\}$ is a parallel pair in $M/u/v$, by switching the $N$-labels on $p'$ and $z$, we have that $z$ is $N$-labelled for deletion.

   In $M/u$, $\{v,p',z\}$ is a triangle, and, since $M \ba z$ has an $N$-minor, $z$ is not in an \unfortunate\ triad.
   Thus Tutte's Triangle Lemma implies that $M/u \ba z$ or $M/u \ba v$ is $3$-connected.
   Since $\{u,z\}$ and $\{u,v\}$ are not contained in triads, either $M \ba z$ or $M \ba v$ is $3$-connected.
   Moreover, the same argument applies with the roles of $u$ and $v$ swapped, implying that either $M \ba z$ or $M \ba v$ is $3$-connected.

   Thus, if $M \ba z$ is not $3$-connected, then both $M \ba u$ and $M \ba v$ are $3$-connected.
   Then,
   since $(M \ba u)|P \cong (M \ba v)|P \cong U_{3,5}$, it follows from \cref{basicplanelemma} that $M \ba u \ba p_1$ and $M \ba v \ba p_1$ are $3$-connected.
   Thus either $\{u, p_1\}$ or $\{v, p_1\}$ is an $N$-detachable pair, by \cref{planeupgradec4}.

   Now we may assume that $M \ba z$ is $3$-connected.
   As $(M \ba z)|P \cong U_{3,5}$, 
   if $P$ does not contain a triad of $M \ba z$, then, by \cref{planelemma}, $M$ has an $N$-detachable pair.
   So suppose that $z$ is in a $4$-element cocircuit $C^*_z$ with elements in $P$.
   Let $Q=P \cup \{u,v,z\}$.
   Observe that $Q$ is $3$-separating, as $r(Q) \le 5$ due to the circuit $\{u,v,p',z\}$, and $r(E(M)-Q) = r(M)-3$, as $r(E(M)-P) = r(M)$ and $\{u,v,z\} \subseteq \cocl(P)$.
   If $d^* \notin C^*_z$, then $$\lambda(Q-d^*) = r(Q-d^*) + r^*(Q-d^*) - |Q-d^*| \le 5+4-7 = 2.$$
   It follows, by \cref{gutses}, that $d^*$ is a guts element in the path of $3$-separations $(Q-d^*, \{d^*\}, E(M)-Q)$.
   But then $M/d^*$ is not $3$-separating; a contradiction.
   So $d^* \in C^*_z$.
   Now $T^* = C^*_z-z$ is a triad in $M \ba z$ with $d^* \in T^*$.
   Let $p'' \in P-(T^* \cup p')$.
   By \cref{basicplanelemma}, $M \ba z \ba p''$ is $3$-connected. 
   Since $p'' \in P-\{d^*,p'\}$,
   $M \ba z \ba p''$ has an $N$-minor, so $\{z, p''\}$ is an $N$-detachable pair.
  \end{slproof}

    \begin{sublemma}
      \label{p7}
      If \cref{p1}\cref{p1b} holds, then $M$ has an $N$-detachable pair.
    \end{sublemma}
    \begin{slproof}
      If $M \ba p_1 \ba p_4$ is $3$-connected, then $M$ has an $N$-detachable pair, so assume otherwise.
      Suppose $\{p_1,p_4\}$ is not contained in a $4$-element cocircuit.
      Then $\co(M \ba p_1 \ba p_4)$ is also not $3$-connected.
      Now $M \ba p_1 \ba p_4$ has a 
      $2$-separation $(X,Y)$ where
      $|X \cap \{p_2,p_3,d^*\}| \ge 2$ and $X$ is fully closed.
      But it follows that $\{p_2,p_3,d^*\} \subseteq X$, and hence $(X \cup \{p_1,p_4\}, Y)$ is a $2$-separation of $M$; a contradiction.
      So $\{p_1,p_4\}$ is contained in a $4$-element cocircuit.
      If this cocircuit also contains either $p_2$ or $p_3$, then, up to relabelling $\{p_1,p_2,p_3,p_4\}$, we are in case \cref{p1}\cref{p1a}.  By \cref{p6}, we may assume otherwise.
      So $\{d^*,p_i,p_4,u_i\}$ is a cocircuit for all $i \in \seq{3}$.

      Let $i \in \seq{3}$ and $j \in \seq{3}-i$ such that $p_j \neq p$.
      Then $M / d^*/p_j$ has an $N$-minor, and, as $P-\{d^*,p_j\}$ is contained in a parallel class in this matroid, $M /p_j \ba p_i \ba p_4$ also has an $N$-minor.
      Since $\{d^*,u_i\}$ is a series pair in $M /p_j \ba p_i \ba p_4$, it follows that $\{p_j,u_i\}$ is $N$-contractible in $M$.
      Now, by \cref{p5}, either $\{p_j,u_i\}$ is an $N$-detachable pair, or $\{p_j,u_i\}$ is contained in a $4$-element circuit $C_{i,j}$.

      By orthogonality, $C_{i,j}$ meets $\{d^*,p_i,p_4\}$.
      If $C_{i,j} \subseteq P \cup u_i$, then $u_i \in \cl(P)$.
      Then $M|(P \cup u_i) \cong U_{3,6}$, and $M$ has an $N$-detachable pair by \cref{6pointplane,basicplanelemma}.
      So let $C_{i,j} = \{p_j,u_i,q_{i,j},v_{i,j}\}$ where $q_{i,j} \in \{d^*,p_i,p_4\}$ and $v_{i,j} \in E(M)-(P \cup u_i)$ (for ease of notation, $q_{i,j} = q_{j,i}$ and $v_{i,j} = v_{j,i}$).
      If $v_{i,j} = u$, then by letting $j' \in \seq{3}-\{i,j\}$, 
      orthogonality between $C_{i,j}$ and the cocircuit $\{d^*,p_{j'},p_4,u_{j'}\}$ implies that $q_{i,j} = p_i$, whereas orthogonality between $C_{i,j}$ and the cocircuit $\{d^*,p_j,p_4,u_j\}$ implies that $q_{i,j} \neq p_i$. 
      So $v_{i,j} \neq u$.
      Observe that $p_j$ is a member of the cocircuit $\{p_1,p_2,p_3,u\}$, and recall that $u \neq u_i$.
      Then, by orthogonality, $q_{i,j} = p_i$, so
      $\{p_i,p_j,u_i,v_{i,j}\}$ is a circuit.
      If $v_{i,j} \neq u_j$, then
      $\{p_j,p_4,d^*,u_j\}$ is a cocircuit that intersects this circuit in one element; 
      a contradiction.

      Without loss of generality we may now assume that $\{p_1,p_2,u_1,u_2\}$ is a circuit.
      It follows that $(P \cup \{u_1, u_2\}, E(M)-(P \cup \{u_1,u_2\}))$ is $3$-separating in $M$.
      Since $M / d^* \ba p_1 \ba p_2$ has an $N$-minor, and $\{p_3,u\}$ is a series pair in this matroid, $M/d^*/u$ has an $N$-minor up to an $N$-label switch.
      Suppose $M/d^*/u$ is not $3$-connected.
      If $\{d^*,u\}$ is contained in a $4$-element cocircuit, then, by orthogonality, this cocircuit is contained in $P \cup u$.  It follows, by cocircuit elimination with $\{p_1,p_2,p_3,u\}$, that $P$ contains a cocircuit; a contradiction.
      So $M/d^*/u$ has a $2$-separation $(U,V)$ for which we may assume $|P \cap U| \ge 2$ and $U$ is fully closed.
      It follows that $P-d^* \subseteq U$, and hence $(U \cup u,V)$ is a $2$-separation of $M/d^*$.  But $M/d^*$ is $3$-connected, so this is contradictory.
      Hence $\{d^*,u\}$ is an $N$-detachable pair.
    \end{slproof}

    \begin{sublemma}
      \label{p8}
      If \cref{p1}\cref{p1c} holds, then $M$ has an $N$-detachable pair.
    \end{sublemma}
    \begin{slproof}
      Consider $M \ba p_1 \ba p_2$.
      If this matroid is $3$-connected, then $M$ has an $N$-detachable pair, so assume otherwise.
      If $\co(M \ba p_1 \ba p_2)$ is not $3$-connected, then there is a $2$-separation $(X,Y)$ of $M \ba p_1 \ba p_2$ for which $|X \cap \{p_3,p_4,d^*\}| \ge 2$ and $X$ is fully closed.
      It follows that
      $(X \cup \{p_1,p_2\},Y)$ is a $2$-separation of $M$; a contradiction.
      So $\{p_1,p_2\}$ is contained in a $4$-element cocircuit of $M$.
      If this cocircuit also contains $p_3$ or $p_4$, then, up to relabelling, we are in case \cref{p1}\cref{p1b}.  Hence, by \cref{p7}, we may assume $\{d^*,p_1,p_2\}$ is contained in a $4$-element cocircuit of $M$.

      Now, for all distinct $i,j \in \seq{4}$, we have that $\{d^*,p_i,p_j,u_{i,j}\}$ is a cocircuit for some $u_{i,j} \in E(M)-P$,
      where $u_{i,j} \neq u_{i',j'}$ if $i \neq i'$ or $j \neq j'$.
      (For ease of notation, we let $u_{i,j} = u_{j,i}$.)
      We may assume that $p = p_4$.
      Then, for all $i \in \seq{3}$ 
      we have that $M/d^*/p_i$ has an $N$-minor.
      For all distinct $j,j' \in \seq{4}-i$,
      since $P-\{d^*,p_i\}$ is contained in a parallel class in $M/d^*/p_i$, 
      the matroid $M/p_i\ba p_j \ba p_{j'}$ has an $N$-minor,
      and it follows that $M/p_i/u_{j,j'}$ has an $N$-minor.
      By \cref{p5}, if $\{p_i,u_{j,j'}\}$ is not contained in a $4$-element circuit, 
      then $M$ has an $N$-detachable pair.
      So we may assume that $\{p_i,u_{j,j'}\}$ is contained in a $4$-element circuit, 
      for all $i \in \seq{3}$ and distinct $j,j' \in \seq{4}-i$.

      Let $\{i,j,j',\ell\} = \seq{4}$ with $i \in \seq{3}$.
      Consider the $4$-element circuit containing $\{p_i,u_{j,j'}\}$. 
      By orthogonality, this circuit meets $\{p_s,u_{i,s},d^*\}$ for each $s \in \seq{4}-i$.
      Hence,
      as the 
      $p_s$'s and $u_{i,s}$'s are distinct for $s \in \seq{4}-i$, 
      the circuit contains $d^*$.
      That is, $\{d^*,p_i,u_{j,j'},v_{i,\ell}\}$ is a circuit for some $v_{i,\ell}$.
      Since $M/d^*/p_i$ has an $N$-minor, 
      it follows that $\{p_t,u_{j,j'}\}$ is $N$-deletable for $t \in \seq{4}-i$.
%
      If $v_{i,\ell} \in P$, then $u_{j,j'} \in \cl(P)$, so $M|(P \cup u_{j,j'}) \cong U_{3,6}$, and it follows from \cref{6pointplane} that $M$ has an $N$-detachable pair. 
      So we may assume that $v_{i,\ell} \in E(M)-P$.
      Now, if $M \ba u_{j,j'}$ is $3$-connected, then, as $(M \ba u_{j,j'})|P \cong U_{3,5}$, and $\{d^*,p_j,p_{j'}\}$ is a triad in $M \ba u_{j,j'}$, it follows, by \cref{basicplanelemma}, that
      $\{p_\ell, u_{j,j'}\}$ is an $N$-detachable pair. 
      So we may assume that $M \ba u_{j,j'}$ is not $3$-connected.

      Again, we let $\{i,j,j',\ell\} = \seq{4}$ with $i \in \seq{3}$.
      Consider $M/d^*$.  Recall that this matroid is $3$-connected, and
      observe that $\{p_i,u_{j,j'},v_{i,\ell}\}$ is a triangle, where $v_{i,\ell} \in E(M)-P$.
      Suppose that $\{p_i,u_{j,j'},v_{i,\ell}\}$ is part of a $4$-element fan.  Then there is a triad of $M$ that contains two elements of $\{p_i,u_{j,j'},v_{i,\ell}\}$.
      But as $p_i$ and $u_{j,j'}$ are $N$-deletable, neither is contained in an \unfortunate\ triad, so this is contradictory.

      Now, by Tutte's Triangle Lemma, either $M/d^* \ba u_{j,j'}$ or $M/d^* \ba v_{i,\ell}$ is $3$-connected.
      If $M/d^* \ba u_{j,j'}$ is $3$-connected, then, as $d^*$ is not in a triangle since it is $N$-contractible, $M \ba u_{j,j'}$ is $3$-connected; a contradiction.
      So $M / d^* \ba v_{i,\ell}$ is $3$-connected, and hence $M \ba v_{i,\ell}$ is $3$-connected.

      Observe that, for $i \in \seq{3}$ and $s \in \seq{4}-i$, the matroid $M \ba p_s \ba v_{i,\ell}$ has an $N$-minor, since $M/d^*/p_i\ba p_s$ has an $N$-minor and $\{u_{j,j'},v_{i,\ell}\}$ is a parallel pair in this matroid.
      As $(M \ba v_{i,\ell})|P \cong U_{3,5}$,
      if $P$ does not contain a triad of $M \ba v_{i,\ell}$, then, by \cref{planelemma}, $M$ has an $N$-detachable pair.
      So suppose $v_{i,\ell}$ is in a $4$-element cocircuit $C^*$ of $M$, where $C^* \subseteq P \cup v_{i,\ell}$.
      Then, by orthogonality, $C^*$ contains $d^*$ or $p_i$.
      Thus, there exists some $s \in \seq{4}-i$ such that $p_s \notin C^*$.
      By \cref{basicplanelemma}, $M \ba v_{i,\ell} \ba p_s$ is $3$-connected,
      so $\{v_{i,\ell},p_s\}$ is an $N$-detachable pair.
      This completes the proof of \cref{p8}.
    \end{slproof}
    The \lcnamecref{planeupgrade} now follows from \cref{p1} and \cref{p6,p7,p8}.
\end{proof}

\section{\Psep s}
\label{sec-problematic}

Throughout this series of papers, we will build up a collection of particular $3$-separators.  Any such \psep~$P$ will have the property that it can appear in a $3$-connected matroid~$M$, with a $3$-connected minor $N$, and with no $N$-detachable pairs, where 
$E(M) - E(N) \subseteq P$.
Recall that the first example we have seen is a \spikelike\ (see \cref{def-spike-like}).
In this section, we define three more \psep s, 
and, for each, describe the construction of a matroid containing the \psep, and with no $N$-detachable pairs.
These \psep s are illustrated in \cref{sfpspid,psfig2}.

Note that this is not a complete list of all such $3$-separators that can give rise to a matroid without an $N$-detachable pair.
Here, we first just consider those \psep s that are either a single-element extension, or the dual of a single-element coextension, of a structure known as a flan, which we consider in the next section.

Let $M$ be a $3$-connected matroid with ground set~$E$.

\begin{definition}
  \label{def-pspider}
  Let $P \subseteq E$ be a $6$-element exactly $3$-separating set such that $P = Q \cup \{p_1,p_2\}$ and $Q$ is a quad.  If there exists a labelling $\{q_1,q_2,q_3,q_4\}$ of $Q$ such that
  \begin{enumerate}[label=\rm(\alph*)]
    \item $\{p_1,p_2,q_1,q_2\}$ and $\{p_1,p_2,q_3,q_4\}$ are the circuits of $M$ contained in $P$, and
    \item $\{p_1,p_2,q_1,q_3\}$ and $\{p_1,p_2,q_2,q_4\}$ are the cocircuits of $M$ contained in $P$, 
  \end{enumerate}
  then $P$ is an \emph{\pspider} of $M$ with {\em associated partition} $(Q,\{p_1,p_2\})$.
\end{definition}

\begin{definition}
  \label{def-twisted}
  Let $P \subseteq E$ be a $6$-element exactly 3-separating set of $M$. If there exists a labelling $\{s_1,s_2,t_1,t_2,u_1,u_2\}$ of $P$ such that
  \begin{enumerate}[label=\rm(\alph*)]
    \item $\{s_1,s_2,t_2,u_1\}$, $\{s_1,t_1,t_2,u_2\}$, and $\{s_2,t_1,u_1,u_2\}$ are the circuits of $M$ contained in $P$; and
    \item $\{s_1,s_2,t_1,t_2\}$, $\{s_1,s_2,u_1,u_2\}$, and $\{t_1,t_2,u_1,u_2\}$ are the cocircuits of $M$ contained in $P$; 
  \end{enumerate}
  then $P$ is a \emph{\twisted} of $M$.
\end{definition}


\begin{figure}[hbtp]
  \begin{subfigure}{0.45\textwidth}
    \centering
    \begin{tikzpicture}[rotate=90,scale=0.875,line width=1pt]
      \tikzset{VertexStyle/.append style = {minimum height=5,minimum width=5}}
      \clip (-2.5,2) rectangle (3.0,-6);
      \node at (-1,-1.4) {$E-P$};
      \draw (0,0) .. controls (-3,2) and (-3.5,-2) .. (0,-4);
      \draw (0,0) -- (2,-2) -- (0,-4);
      \draw (0,0) -- (2.5,0.5) -- (2,-2);
      \draw (0,0) -- (2.25,-0.75);
      \draw (2,-2) -- (1.25,0.25);

      \Vertex[x=1.25,y=0.25,LabelOut=true,L=$q_3$,Lpos=180]{c1}
      \Vertex[x=2.25,y=-0.75,LabelOut=true,L=$q_2$,Lpos=90]{c2}
      \Vertex[x=2.5,y=0.5,LabelOut=true,L=$q_1$,Lpos=180]{c3}
      \Vertex[x=1.5,y=-0.5,LabelOut=true,L=$q_4$,Lpos=135]{c4}
      \Vertex[x=1.33,y=-2.67,LabelOut=true,L=$p_1$,Lpos=45]{c5}
      \Vertex[x=0.67,y=-3.33,LabelOut=true,L=$p_2$,Lpos=45]{c6}

      \draw (0,0) -- (0,-4);

      \SetVertexNoLabel
      \tikzset{VertexStyle/.append style = {shape=rectangle,fill=white}}
    \end{tikzpicture}
    \caption{An \pspider.}
    \label{sfpspida}
  \end{subfigure}
  \begin{subfigure}{0.45\textwidth}
    \centering
    \begin{tikzpicture}[rotate=90,scale=0.7,line width=1pt]
      \tikzset{VertexStyle/.append style = {minimum height=5,minimum width=5}}
      \clip (-2.5,-6) rectangle (4.4,2);
      \node at (-1,-1.4) {$E-P$};
      \draw (0,0) .. controls (-3,2) and (-3.5,-2) .. (0,-4);
      \draw (0,0) -- (4.0,0.9);
      \draw (0,-2) -- (2.5,-2.2); 
      \draw (0,-4) -- (3.8,-4.9); 

      \Vertex[x=3.0,y=0.67,LabelOut=true,Lpos=180,L=$s_1$]{a2}
      \Vertex[x=2.0,y=0.45,LabelOut=true,Lpos=180,L=$s_2$]{a3}

      \Vertex[x=2.5,y=-2.2,LabelOut=true,Lpos=90,L=$t_2$]{b1}
      \Vertex[x=0.64,y=-2.056,LabelOut=true,Lpos=-45,L=$t_1$]{b2}

      \Vertex[x=3.8,y=-4.9,LabelOut=true,L=$u_1$]{c1}
      \Vertex[x=2.8,y=-4.67,LabelOut=true,L=$u_2$]{c2}

      \draw[dashed] (3.8,-4.9) .. controls (2.0,-2) .. (4.0,0.9);
      \draw[dashed] (2.8,-4.67) .. controls (1.0,-2) .. (3.0,0.67);
      \draw[dashed] (1.8,-4.45) .. controls (0.25,-2) .. (2.0,0.45);

      \draw (0,0) -- (0,-4);

      \SetVertexNoLabel
      \tikzset{VertexStyle/.append style = {shape=rectangle,fill=white}}
      \Vertex[x=4.0,y=0.9]{a1}
      \Vertex[x=1.5,y=-2.12]{b3}
      \Vertex[x=1.8,y=-4.45]{c3}

    \end{tikzpicture}
    \caption{A \twisted.}
  \end{subfigure}
  \caption{Two \psep s.  Each is a single-element extension of a $5$-element flan, and is in a matroid with rank $r(E-P)+2$. }
  \label{sfpspid}
\end{figure}
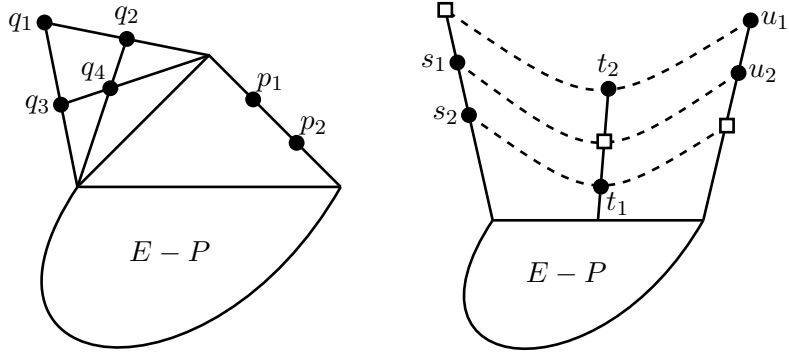

\begin{definition}
  Let $P \subseteq E$ be an exactly $3$-separating set with $P=\{p_1,p_2,q_1,q_2,s_1,s_2\}$.
  Suppose that
  \begin{enumerate}[label=\rm(\alph*)]
    \item $\{p_1,p_2,s_1,s_2\}$, $\{q_1,q_2,s_1,s_2\}$, and $\{p_1,p_2,q_1,q_2\}$ are the circuits of $M$ contained in $P$; and
    \item $\{p_1,q_1,s_1,s_2\}$, $\{p_2,q_2,s_1,s_2\}$, 
      $\{p_1,p_2,q_1,q_2,s_1\}$ and $\{p_1,p_2,q_1,q_2,s_2\}$ 
      are the cocircuits of $M$ contained in $P$. 
  \end{enumerate}
  Then $P$ is a \emph{\tvamoslike} of $M$. 
\end{definition}

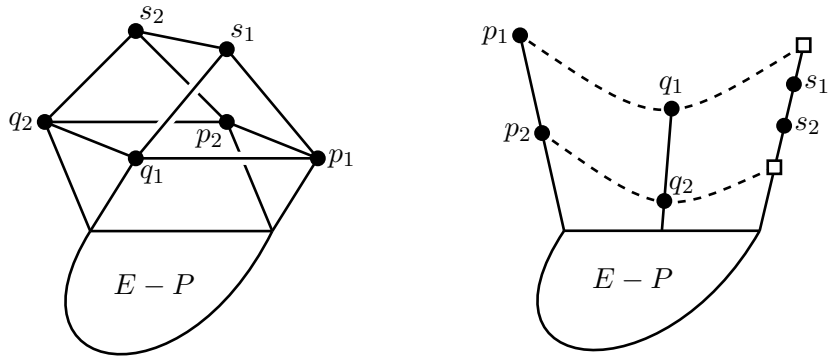
\begin{figure}[htbp]
  \begin{subfigure}{0.49\textwidth}
    \centering
    \begin{tikzpicture}[rotate=90,xscale=1.21,yscale=0.605,line width=1pt]
      \tikzset{VertexStyle/.append style = {minimum height=5,minimum width=5}}
      \clip (-1.5,2) rectangle (2.7,-6);
      \node at (-0.6,-1.4) {$E-P$};
      \draw (0,0) .. controls (-1.6,2) and (-2,-2) .. (0,-4);

      \draw (1.2,1) -- (0.8,-1);
      \draw (0,-4) -- (1.2,-3);

      \draw (1.2,1) -- (2.2,-1) -- (1.2,-3);
      \draw (2.2,-1) -- (2.0,-3);

      \draw (1.2,1) -- (1.2,-3);


      \draw[white,line width=5pt] (0.8,-1) -- (2.0,-3) -- (0.8,-5);
      \draw[white,line width=5pt] (0.8,-1) -- (0.8,-5);
      \draw[white,line width=5pt] (0.8,-1) -- (0,0);
      \draw (0.8,-1) -- (2.0,-3) -- (0.8,-5);
      \draw (0.8,-1) -- (0.8,-5);
      \draw (1.2,-3) -- (0.8,-5) -- (0,-4);
      \draw (0.8,-1) -- (0,0);
      \draw (0,0) -- (0,-4);
      \draw (0,0) -- (1.2,1);


      \Vertex[x=1.2,y=1,LabelOut=true,L=$q_2$,Lpos=180]{c1}
      \Vertex[x=1.2,y=-3,LabelOut=true,L=$p_2$,Lpos=225]{c4}
      \Vertex[x=0.8,y=-1,LabelOut=true,L=$q_1$,Lpos=-45]{c5}
      \Vertex[x=0.8,y=-5,LabelOut=true,L=$p_1$,Lpos=0]{c6}
      \Vertex[x=2.0,y=-3,LabelOut=true,L=$s_1$,Lpos=45]{c5}
      \Vertex[x=2.2,y=-1,LabelOut=true,L=$s_2$,Lpos=45]{c6}

    \end{tikzpicture}
    \caption{A \tvamoslike\ of $M$.}
  \end{subfigure}
  \begin{subfigure}{0.50\textwidth}
    \centering
    \begin{tikzpicture}[rotate=90,scale=0.65,line width=1pt]
      \tikzset{VertexStyle/.append style = {minimum height=5,minimum width=5}}
      \clip (-2.78,-6) rectangle (5.0,2);
      \node at (-1,-1.4) {$E-P$};
      \draw (0,0) .. controls (-3,2) and (-3.5,-2) .. (0,-4);
      \draw (0,0) -- (4.0,0.9);
      \draw (0,-2) -- (2.5,-2.2); 
      \draw (0,-4) -- (3.8,-4.9); 

      \Vertex[x=2.0,y=0.45,LabelOut=true,Lpos=180,L=$p_2$]{a3}

      \Vertex[x=2.5,y=-2.2,LabelOut=true,Lpos=90,L=$q_1$]{b1}
      \Vertex[x=0.61,y=-2.05,LabelOut=true,Lpos=35,L=$q_2$]{b2}

      \Vertex[x=3.00,y=-4.7,LabelOut=true,L=$s_1$]{c1}
      \Vertex[x=2.15,y=-4.5,LabelOut=true,L=$s_2$]{c2}

      \Vertex[x=4.0,y=0.9,LabelOut=true,Lpos=180,L=$p_1$]{a1}

      \draw[dashed] (3.8,-4.9) .. controls (2.0,-2) .. (4.0,0.9);
      \draw[dashed] (1.3,-4.3) .. controls (0.25,-2) .. (2.0,0.45);

      \draw (0,0) -- (0,-4);

      \SetVertexNoLabel
      \tikzset{VertexStyle/.append style = {shape=rectangle,fill=white}}
      \Vertex[x=3.8,y=-4.9]{c1}
      \Vertex[x=1.3,y=-4.3]{c3}

    \end{tikzpicture}
    \caption{A \tvamoslike\ of $M^*$.}
  \end{subfigure}
  \caption{Geometric representations of a \tvamoslike\ 
  of $M$ and $M^*$.}
  \label{psfig2}
\end{figure}

For a $3$-connected matroid $M$, we say that a pair $\{x_1,x_2\} \subseteq E(M)$ is \emph{detachable} if either $M/x_1/x_2$ or $M \ba x_1 \ba x_2$ is $3$-connected.
In the former case, we will say that $\{x_1,x_2\}$ is a \emph{contraction pair}; in the latter case, $\{x_1,x_2\}$ is a \emph{deletion pair}.

Each \psep~$P$ that we have seen in this section can be used to construct a $3$-connected matroid~$M$ with a $3$-connected matroid~$N$ as a minor, such that $M$ has no $N$-detachable pairs, and $E(M)-E(N) \subseteq P$. 
For the \pspider\ and \twisted, this follows from the fact that for such a $3$-separator $P$, there is no detachable pair contained in $P$.
On the other hand, the \tvamoslike\ can contain a detachable pair, but appear in a matroid with no $N$-detachable pairs.

We first consider the \pspider.
Let $M$ be a $3$-connected matroid with a $3$-separation $(P,S)$ such that $P$ is an \pspider, the matroid $N = M \ba P$ is $3$-connected, $\cl(P)$ does not contain any triangles, and $\cocl(P)$ does not contain any triads.
Provided $N$ is sufficiently structured to ensure that $M/s$ and $M \ba s$ have no $N$-minor for any $s \in S$, the matroid $M$ has no $N$-detachable pairs, even after first performing a $\Delta$-$Y$ or $Y$-$\Delta$ exchange.
Note that in this example $|E(M)|-|E(N)|=6$.

It is also possible that $|E(M)|-|E(N)| = 5$.
In this case, up to duality, the $3$-connected $N$-minor can be obtained by extending $M$ by an element $e$ in the guts of $(P,S)$, then restricting to $S \cup e$.
Different cases arise depending on where in the guts needs to be ``filled in'' in order to obtain $N$.
If $P$ is labelled as in \cref{sfpspida},
then $e$ is in either 
\begin{itemize}
  \item $\cl(\{q_1,q_3\}) \cap \cl(\{q_2,q_4\}) \cap \cl(S)$,
  \item $\cl(\{p_1,p_2\}) \cap \cl(S)$, or
  \item $\cl(S) - (\cl(\{q_1,q_3\}) \cup \cl(\{q_2,q_4\}) \cup \cl(\{p_1,p_2\}))$.
\end{itemize}
%

Here we have just focussed on cases where $|E(M)| - |E(N)| \ge 5$, though it is also possible that $|E(M)| - |E(N)| \in \{3,4\}$.
%

Similarly, a \twisted\ can appear in a $3$-connected matroid $M$ with a $3$-connected minor $N$ such that $E(M)-E(N) \subseteq P$ and $M$ has no $N$-detachable pairs, where $|E(M)|-|E(N)| \in \{3,4,5,6\}$.

  \begin{figure}[htb]
    \centering
    \begin{tikzpicture}[rotate=90,xscale=1.331,yscale=0.6655,line width=1pt]
      \tikzset{VertexStyle/.append style = {minimum height=5,minimum width=5}}
      \clip (-1.5,2) rectangle (2.7,-6);
      \draw (0,0) -- (-1.3,-2) -- (0,-4);
      \draw (-0.65,-1) -- (0,-4);
      \draw (0,0) -- (-0.65,-3);
      \draw (0,-2) -- (-1.3,-2);

      \draw (1.2,1) -- (0.8,-1);
      \draw (0,-4) -- (1.2,-3);

      \draw (1.2,1) -- (2.2,-1) -- (1.2,-3);
      \draw (2.2,-1) -- (2.0,-3);

      \draw (1.2,1) -- (1.2,-3);

      \draw[dotted] (1.2,1) .. controls (-0.4,-0.3333) and (-0.4,-4.3333) .. (1.2,-3);

      \draw[white,line width=5pt] (0.8,-1) -- (2.0,-3) -- (0.8,-5);
      \draw[white,line width=5pt] (0.8,-1) -- (0.8,-5);
      \draw[white,line width=5pt] (0.8,-1) -- (0,0);
      \draw[white,line width=5pt] (0.8,-1) .. controls (-0.2666,0.2222) and (-0.2666,-3.7777) .. (0.8,-5);
      \draw (0.8,-1) -- (2.0,-3) -- (0.8,-5);
      \draw (0.8,-1) -- (0.8,-5);
      \draw (1.2,-3) -- (0.8,-5) -- (0,-4);
      \draw (0.8,-1) -- (0,0);
      \draw (0,0) -- (0,-4);
      \draw (0,0) -- (1.2,1);

      \draw[dashed] (0.8,-1) .. controls (-0.2666,0.2222) and (-0.2666,-3.7777) .. (0.8,-5);

      \Vertex[x=1.2,y=1,LabelOut=true,L=$q_2$,Lpos=180]{c1}
      \Vertex[x=1.2,y=-3,LabelOut=true,L=$p_2$,Lpos=225]{c4}
      \Vertex[x=0.8,y=-1,LabelOut=true,L=$q_1$,Lpos=-45]{c5}
      \Vertex[x=0.8,y=-5,LabelOut=true,L=$p_1$,Lpos=0]{c6}
      \Vertex[x=2.0,y=-3,LabelOut=true,L=$s_1$,Lpos=45]{c2}
      \Vertex[x=2.2,y=-1,LabelOut=true,L=$s_2$,Lpos=45]{c3}

      \Vertex[x=0,y=0,LabelOut=true,L=$t_1$,Lpos=180]{c7}
      \Vertex[x=0,y=-4,LabelOut=true,L=$t_2$,Lpos=0]{c8}
      \SetVertexNoLabel
      \Vertex[x=-1.3,y=-2,LabelOut=true,L=$s_2$,Lpos=45]{c9}
      \Vertex[x=-.65,y=-1,LabelOut=true,L=$s_2$,Lpos=45]{c10}
      \Vertex[x=-.65,y=-3,LabelOut=true,L=$s_2$,Lpos=45]{c11}
      \Vertex[x=-.45,y=-2,LabelOut=true,L=$s_2$,Lpos=45]{c12}

      \SetVertexNoLabel
      \tikzset{VertexStyle/.append style = {fill=white}}
      \tikzset{VertexStyle/.append style = {shape=rectangle}}
      \Vertex[x=0,y=-2]{v2}
    \end{tikzpicture}
    \caption{A matroid $M$ with a \tvamoslike\ and no $F_7^-$-detachable pairs.}
    \label{tvamoseg}
  \end{figure}
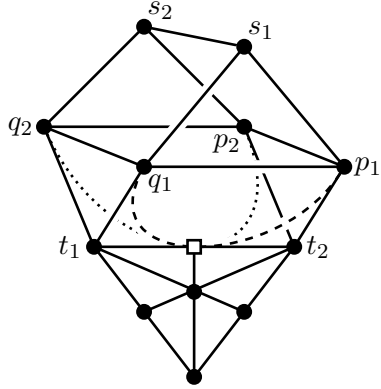

Finally, we consider the \tvamoslike.
We describe a matroid $M$ with a \tvamoslike\ and with no $F_7^-$-detachable pairs; this matroid is illustrated in \cref{tvamoseg}.
  Let $U_8$ be the paving matroid on ground set $\{p_1,p_2,q_1,q_2,s_1,s_2,t_1,t_2\}$ whose non-spanning circuits are $\{t_1,t_2,p_1,q_1\}$, $\{t_1,t_2,p_2,q_2\}$, $\{p_1,p_2,q_1,q_2\}$, $\{p_1,p_2,s_1,s_2\}$, and $\{q_1,q_2,s_1,s_2\}$. 
  Let $U_8^+$ be the 
  single-element extension of $U_8$ by the element~$z$ such that $z$ is in the span of the lines $\{t_1,t_2\}$, $\{q_1,p_1\}$, and $\{q_2,p_2\}$ and $z$ is not a loop.
  Label the triangle $T=\{t_1,t_2,z\}$.
  Let $F_7^-$ be a copy of the non-Fano matroid with $E(F_7^-) \cap E(U_8^+) = T$ such that $T$ is a triangle of $F_7^-$.
  Now let $M=P_{T}(U_8^+,F_7^-) \ba z$, and observe that $M$ is $3$-connected and has an $F_7^-$-minor. 
  In particular, $F_7^- \cong M/p_1 \ba \{s_1,s_2,p_2,q_2\}$, for example, so $|E(M)|-|E(F_7^-)| = 5$.
  Let $X = \{p_1,p_2,q_1,q_2,s_1,s_2\}$; the set $X$ is a \tvamoslike\ of $M$.
  
  The matroid $M$ has no $F_7^-$-detachable pairs.
  To see this, first observe that neither $M\ba y$ nor $M/y$ has an $F_7^-$-minor for any $y \in E(M)-X$. 
  Due to the $4$-element circuits and cocircuits contained in $X$, the only detachable pairs of $M$ contained in $X$ are the deletion pairs
  $\{p,q\} \in \{\{p_1,q_2\}, \{p_2,q_1\} \}$.
  But for any such $\{p,q\}$, the matroid $M \ba p \ba q$ has no $F_7^-$-minor.

  Note that although here we have used $N=F_7^-$ as the minor, other choices of $N$ would work provided $N$ is sufficiently structured and has a triangle $T=\{t_1,t_2,z\}$.

\section{Flans}
\label{secflans}

Let $F$ be a set of elements in a $3$-connected matroid $M$, with $t=|F| \ge 4$.
Suppose $F$ has an ordering $(f_1,f_2,\dotsc,f_t)$ such that
\begin{enumerate}[label=\rm(\alph*)]
\item if $i \in \seq{t-2}$ is odd, then $\{f_i,f_{i+1},f_{i+2}\}$ is a triad; and
\item if $i \in \{4,5,\dotsc,t\}$ is even, then $f_i\in\cl(\{f_1,f_2,\ldots,f_{i-1}\})$.
\end{enumerate}
Then $F$ is a {\em flan} of $M$, 
and $(f_1,f_2,\dotsc,f_t)$ is a \emph{flan ordering} (or just an \emph{ordering}) of $F$.
%
A flan $F$ is {\em maximal} if 
it is not properly contained in another flan.
Note that $\{f_1,f_2,\dotsc,f_i\}$ is $3$-separating for any $i \in \seq{t}$.

In this section we consider the case where there is an $N$-deletable element $d \in E(M)$ such that $M \ba d$ is $3$-connected, but $M \ba d$ has a flan~$F$ with at least five elements.
In this case, we show that either $M$ has an $N$-detachable pair, or 
$F \cup d$ is one of the $3$-separators defined in \cref{sec-problematic}.
We focus on the case where the flan has at least five elements, because
if $M \ba d$ has a maximal $4$-element flan with ordering $(f_1,f_2,f_3,f_4)$, then $M \ba d \ba f_4$ is $3$-connected.

Note that a flan generalises the notion of a fan.
Note also that the definition of a flan used here is more restrictive than that often appearing in the literature (see \cite{Hall2005,Hall2007}, for example).
A geometric illustration of an example of a flan is given in \cref{figflan}.

\begin{figure}[htbp]
    \centering
    \begin{tikzpicture}[scale=0.8]
        \draw[line width=1pt] (0,0) -- (2,3.464) -- (0,4);
        \draw[line width=1pt] (0,0) .. controls (4,2) and (6.5,5.464) .. (2,3.464);
        \SetVertexNoLabel
        \tikzset{VertexStyle/.append style = {shape=rectangle,draw,minimum height=5,minimum width=5}}
        \tikzset{VertexStyle/.append style = {fill=white}}
        \Vertex[x=0,y=0,L=$c$]{c}
        \tikzset{VertexStyle/.append style = {shape=circle,fill=black}}
        \node at (3.05,3.25) {\small$E(M)-F$};
        \SetVertexLabel
        \Vertex[x=-2,y=3.464,LabelOut=true,Lpos=90,L=$f_4$]{y1}
        \Vertex[x=1,y=1.732,LabelOut=true,Lpos=180,L=$f_8$]{y1p}
        \Vertex[x=0,y=4,LabelOut=true,Lpos=90,L=$f_6$]{g2}
        \Vertex[x=-3.464,y=2,LabelOut=true,Lpos=180,L=$f_1$]{x2}
        \Vertex[x=-2.598,y=1.5,LabelOut=true,Lpos=-90,L=$f_2$]{b0}

        \Vertex[x=-.75,y=2.8,LabelOut=true,Lpos=90,L=$f_5$]{yg}
        \Vertex[x=-2.05,y=2.05,LabelOut=true,Lpos=90,L=$f_3$]{xy}
        \Vertex[x=.75,y=2.8,LabelOut=true,Lpos=90,L=$f_7$]{ygp}
        \Edge(x2)(y1)
        \Edge(g2)(y1)
        \Edge(c)(g2)
        \Edge(c)(b0)
        \Edge(x2)(b0)
        \Edge(c)(y1)
        \Edge(c)(y1p)
    \end{tikzpicture}
    \caption{An example of a flan $F$ with ordering $(f_1,f_2,\dotsc,f_8)$.  Note that $r(M) = r(E(M)-F) + 3$.}
    \label{figflan}
\end{figure}
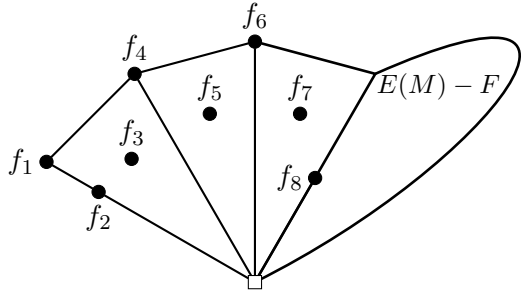

We start with a lemma that demonstrates that certain elements in a flan are candidates for contraction.

\begin{lemma}
  \label{flanend}
  Let $F$ be a maximal flan in a $3$-connected matroid~$M$, with $|F| \ge 5$ and $F \neq E(M)$.
  Let $i \in \seq{3}$, 
  let $j$ be an odd integer such that $5 \le j \le |F|$, and
  suppose $F$ has an ordering $(f_1,f_2,f_3,\dotsc,f_{|F|})$
  such that $f_i$ and $f_j$ are not contained in triangles.
  Then 
  \begin{enumerate}
    \item $M / f_i$, $M / f_j$ and $\si(M/f_i/f_j)$ are $3$-connected; 
    \item if $j \ge 7$, then $M / f_i / f_j$ is $3$-connected; and 
    \item if $|F|=5$, then $M / f_i / f_j$ is $3$-connected. 
  \end{enumerate}
\end{lemma}
\begin{proof}
  Let $F' = \{f_1,f_2,\dotsc,f_j\}$.
  Suppose that $(F'-f_j, \{f_j\}, E(M)-F')$ is a cyclic $3$-separation of $M$.
  Then $\si(M/f_j)$ is $3$-connected by Bixby's Lemma.
  Since
  $f_j$ is not contained in a triangle, 
  $M/f_j$ is $3$-connected.
  On the other hand, if $(F'-f_j, \{f_j\}, E(M)-F')$ is not a cyclic $3$-separation of $M$, 
  then $(E(M)-F')\cup f_j$ is independent.
  Then $j = |F|$ and $F' = F$, otherwise $f_{j+1}$ is in a circuit contained in $E(M)-F'$.
  If $|E(M)-F| < 3$, then, as $M$ is $3$-connected, $F$ spans $M$, contradicting the maximality of $F$.
  It follows that $(E(M)-F) \cup f_j$ is a cosegment consisting of at least four elements, so $M/f_j$ is $3$-connected by the dual of \cref{rank2Remove}.

  Now, by the dual of \cref{r3cocirc}, the matroids $\si(M /f_i)$ and $\si(M /f_i/f_j)$ are $3$-connected. 
  But $f_i$ is not in a triangle, 
  so $M / f_i$ is $3$-connected. 
  This proves (i).

  Now suppose $\{f_i,f_j\}$ is contained in a $4$-element circuit $C$.
  By orthogonality, $C$ must contain an element $f_i' \in \{f_1,f_2,f_3\}-f_i$ and an element $f_j' \in \{f_{j-2},f_{j-1}\}$.  If the elements $\{f_i,f_i',f_j,f_j'\}$ are distinct, then $f_j \in \cl(F'-f_j)$; a contradiction.
  It follows that if $j > 5$, then
$\{f_i,f_j\}$ is not contained in a $4$-element circuit, and thus $M/f_i/f_j$ is $3$-connected.
  This proves (ii).

  Now we may assume that $j=5$, 
  and $C \cap \{f_1,f_2,f_3,f_4,f_5\} = \{f_\ell,f_3,f_5\}$ for some $\ell \in \{1,2\}$.
  Let $C - \{f_\ell,f_3,f_5\} = \{x\}$.
  Then $\{f_1,f_2,f_3,f_4,f_5,x\}$ is a flan.  Thus, if $|F|=5$, 
  then
  $\{f_i,f_j\}$ is not contained in a $4$-element circuit, and thus $M/f_i/f_j$ is $3$-connected.  This proves (iii).
\end{proof}

The next \lcnamecref{flannew} deals with the case where the flan 
has at least six elements.

\begin{proposition}
  \label{flannew}
    Let $M$ be a $3$-connected matroid, and let $N$ be a $3$-connected minor of $M$, where $|E(N)| \ge 4$ and every triangle or triad of $M$ is \unfortunate.
    Suppose that $M \ba d$ is $3$-connected 
    and has a maximal flan~$F$ with $|F| \ge 6$ and ordering $(f_1,f_2,f_3,\dotsc,f_{|F|})$, 
    where $M \ba d \ba f_5$ has an $N$-minor
    with $|\{f_1,\dotsc,f_4\} \cap E(N)| \le 1$.
    Then $M$ has an $N$-detachable pair.
\end{proposition}
\begin{proof}
  Let $t = |F|$.
  Observe that $(\{f_1,f_2,f_3,f_4\}, E(M \ba d) - \{f_1,\dotsc,f_5\})$ is a $2$-separation of $M \ba d \ba f_5$.

  \begin{sublemma}
    \label{flannewnsl}
    $M \ba d/ f_i / f_j$ has an $N$-minor for $i \in \{1,2\}$ and 
    $j \in \{5,7\} \cap \seq{t}$.
  \end{sublemma}
  \begin{slproof}
    First consider $j=5$.
    Since $\{f_3,f_4\}$ is a series pair in $M \ba d \ba f_5$, we have that $M \ba d \ba f_5 / f_4$ is connected.
    Thus
    \cref{m2.7} implies that
    $M \ba d\ba f_5 /f_4$, and hence $M\ba d/f_4$
    has an $N$-minor.
    By further applications of \cref{m2.7}, we obtain that $M \ba d / f_4 \ba f_3$ has an $N$-minor, and that $M \ba d / f_4 \ba f_3 / f_i$ has an $N$-minor for $i \in \{1,2\}$.
    Since $M \ba d \ba f_3 /f_i$ has an $N$-minor and $\{f_4,f_5\}$ is a series pair in this matroid, $M \ba d \ba  f_3/f_i /f_5$, and in particular $M \ba d/f_i / f_5$, has an $N$-minor, as required.

    Now suppose $t \ge 7$, and consider $j=7$.
    As $M \ba d \ba f_5$ has an $N$-minor and
    $M \ba d \ba f_5 / f_i$ is connected, $M \ba d \ba f_5 / f_i$ has an $N$-minor by \cref{m2.7}.  But $\{f_6,f_7\}$ is a series pair in this matroid, so we deduce that $M \ba d/f_i /f_7$ has an $N$-minor.
  \end{slproof}

  \begin{sublemma}
    \label{flansl7}
    If $t \ge 7$, then $M$ has an $N$-detachable pair.
  \end{sublemma}
  \begin{slproof}
    Let $i \in \{1,2\}$.
    By \cref{flanend}(ii), $M \ba d/f_i/f_7$ is $3$-connected, and, by \cref{flannewnsl}, $M \ba d/f_i /f_7$ has an $N$-minor.
    So either $\{f_i,f_7\}$ is an $N$-detachable pair, or $d$ is in a parallel pair in $M/f_i/f_7$.
    Since neither $f_i$ nor $f_7$ is in an \unfortunate\ triangle, $M$ has a $4$-element circuit $\{d,f_i,f_7,t_i\}$ 
    where, by orthogonality, $t_i \in \{f_3,f_4,f_5\}$.
    By circuit elimination, $\{f_1,f_2,t_1,t_2,f_7\}$ contains a circuit.  But $f_7 \notin \cl(\{f_1,f_2,f_3,f_4,f_5\})$,
    so this circuit is $\{f_1,f_2,t_1,t_2\}$,
    and it follows that $\{t_1,t_2\} = \{f_3,f_4\}$.

    We now work towards showing that either $\{f_5,f_7\}$ is an $N$-detachable pair, or $\{f_5,f_7\}$ is contained in a $4$-element circuit of $M\ba d$.
    We have that $\{d,f_4\} \subseteq \cl_{M/f_7}(\{f_1,f_2,f_3\})$,
    but $f_5 \notin \cl_{M/f_7}(\{f_1,f_2,f_3\}) = \cl_{M/f_7}(\{f_1,f_2,f_3,f_4,d\})$.
    Consider a triangle containing $f_5$ in $M/f_7$.  It can contain at most one element in $\{f_1,f_2,f_3,f_4,d\}$.  Thus, it cannot contain $d$, as $d \notin \cl_{M/f_7}(E(M/f_7) - \{f_1,f_2,f_3\})$ since $d$ blocks the triad $\{f_1,f_2,f_3\}$.

    We claim that $M\ba d/f_5/f_7$ has an $N$-minor.
    Since $M\ba d \ba f_5$ has an $N$-minor and $M \ba d \ba f_5 / f_2 / f_3$ is connected, \cref{m2.7} implies that $M \ba d \ba f_5 / f_2 / f_3$ has an $N$-minor.
    Fix an $N$-labelling $(C,D)$ with $\{f_2,f_3\} \subseteq C$ and $\{d,f_5\} \subseteq D$.
    Since $\{f_6,f_7\}$ is a series pair in 
    $M\ba d \ba f_5$, we may assume that $f_7$ is $N$-labelled for contraction,
    up to an $N$-label switch on $f_6$ and $f_7$.
    Since $\{f_1,f_4\}$ is a parallel pair in $M / f_2 / f_3$, we may also assume, up to an $N$-label switch on $f_1$ and $f_4$, that $f_4$ is $N$-labelled for deletion.
    In particular, observe that $M \ba d \ba f_4 / f_7$ has an $N$-minor.
    Finally, $\{f_3,f_5\}$ is a series pair in $M \ba d \ba f_4$, so, after switching the $N$-labels on $f_3$ and $f_5$, the element $f_5$ is $N$-labelled for contraction, and $f_3$ is $N$-labelled for deletion.
    To summarise, $\{f_2,f_5,f_7\} \subseteq C$ and $\{d,f_3,f_4\} \subseteq D$.
    So $M\ba d/f_5/f_7$ has an $N$-minor, as claimed.

    Since $f_7$ is $N$-contractible, $f_7$ is not in a triangle of $M \ba d$.
    Thus, by \cref{flanend}, $M \ba d /f_7$ is $3$-connected, and $(f_1,f_2,\dotsc,f_6)$ is a flan ordering in this matroid.
    Hence, either $f_5$ is in a triangle of $M \ba d / f_7$, or, by another application of \cref{flanend}, $M \ba d /f_5/f_7$ is $3$-connected.
    In the latter case, as $d$ is not in a triangle with $f_5$ in $M/f_7$,
    the matroid $M/f_5/f_7$ is $3$-connected, so
    $\{f_5,f_7\}$ is an $N$-detachable pair.
    So we may assume that $\{f_5,f_7\}$ is contained in a $4$-element circuit of $M \ba d$.

    By orthogonality, this circuit meets $\{f_3,f_4\}$.
    But if this circuit is contained in $\{f_1,f_2,\dotsc,f_7\}$, then $f_7 \in \cl(\{f_1,f_2,\dotsc,f_6\})$; a contradiction.
    It follows, by orthogonality, that 
    $\{f_4,f_5,f_7,f_8\}$ is a circuit, where
    $f_8 \in E(M\ba d)-\{f_1,f_2,\dotsc,f_7\}$.
    Note that $\{f_1,f_2,\dotsc,f_8\}$ is a flan.
    Recall the $N$-labelling $(C,D)$ of $M$ with $d \in D$ and $\{f_5,f_7\} \subseteq C$.
    Since 
    $\{f_4,f_8\}$ is a parallel pair in $M/f_5/f_7$,
    the element $f_8$ is $N$-labelled for deletion
    after switching the $N$-labels on $f_4$ and $f_8$.
    In particular, $M \ba d \ba f_8$ has an $N$-minor.

    Let $Z=E(M\ba d)-\{f_1,\dotsc,f_8\}$,
    and observe that
    $(\{f_1,f_2,\dotsc,f_7\}, \{f_8\}, Z)$ is a 
    path of $3$-separations.
    Suppose $|Z| = 1$. Then $E(M \ba d)$ is a $9$-element flan, 
    $f_2$ and $f_7$ are $N$-labelled for contraction with respect to the $N$-labelling $(C,D)$,
    and it is easily verified that $M/f_2/f_7$ is $3$-connected. So $\{f_2,f_7\}$ is an $N$-detachable pair.
    We now may assume that $|Z| \ge 2$.

    We claim that $\co(M \ba d \ba f_8)$ is $3$-connected.
    If $r(Z) \ge 3$, then $(\{f_1,f_2,\dotsc,f_7\}, \{f_8\}, Z)$ is a vertical $3$-separation,
    and the claim follows from Bixby's Lemma.
    On the other hand, if $r(Z) \le 2$ and $|Z| \ge 3$, then $(M \ba d)|(Z \cup f_8) \cong U_{2,4}$, and $M \ba d \ba f_8$ is $3$-connected by \cref{rank2Remove}.
    Finally, if $|Z| = 2$, then 
    $Z \cup \{f_7,f_8\}$ is a rank-$3$ cocircuit, and $\co(M\ba d \ba f_8)$ is $3$-connected by \cref{r3cocirc}, thus proving the claim.
    So either $\{d,f_8\}$ is an $N$-detachable pair, or $f_8$ is in a triad of $M \ba d$.

    We may now assume that $f_8$ is in a triad~$T^*$ of $M \ba d$.
    Since $f_8 \notin \cocl(\{f_1,f_2,\dotsc,f_7\})$, the triad $T^*$ contains an element $q \in E(M)-\{f_1,f_2,\dotsc,f_8\}$.
    By orthogonality, $T^*$ meets $\{f_4,f_5,f_7\}$.
    But if $T^* = \{f_4,f_8,q\}$, then $T^*$ intersects the circuit $\{f_1,f_2,f_3,f_4\}$ in one element; a contradiction.
    Similarly, if $T^* = \{f_5,f_8,q\}$,
    then $\{d,f_5,f_8,q\}$ is a cocircuit of $M$ that intersects the circuit $\{d,f_1,f_7,t_1\}$ (where $t_1 \in \{f_3,f_4\}$) in one element; a contradiction.
    We deduce that $T^* = \{f_7,f_8,q\}$.
    After relabelling $q$ as $f_9$, we observe that $(f_1,f_2,\dotsc,f_9)$ is a flan ordering.

    Next we claim that $M/f_i/f_9$ has an $N$-minor for $i \in \{1,2\}$.
    Again we recall the $N$-labelling $(C,D)$ from earlier, which has $\{f_2,f_7\} \subseteq C$ and $\{d,f_8\} \subseteq D$.
    Since
    $\{f_7,f_9\}$ is a series pair in $M \ba d \ba f_8$, 
    the element $f_9$ is $N$-labelled for contraction after switching the $N$-labels on $f_7$ and $f_9$.
    So $M/f_2/f_9$ has an $N$-minor.
    Using a similar argument, but starting with an $N$-labelling $(C',D')$ that has $\{f_1,f_3\} \subseteq C'$ and $\{d,f_5\} \subseteq D'$, one can show that $M/f_1/f_9$ has an $N$-minor.

    Let $i \in \{1,2\}$.
    Since each of $f_i$ and $f_9$ is not contained in an \unfortunate\ triangle,
    $M \ba d / f_i / f_9$ is $3$-connected,
    by \cref{flanend}(ii).
    Now either $\{f_i,f_9\}$ is an $N$-detachable pair, or $d$ is in a parallel pair in $M/f_i/f_9$.
    Hence, we may assume that
    $M$ has a $4$-element circuit containing $\{d,f_i,f_9\}$. 
    By orthogonality,
    this circuit meets $\{f_3,f_4,f_5\}$ and $\{f_5,f_6,f_7\}$, so
    $\{d,f_i,f_5,f_9\}$ is a circuit.
    As $i \in \{1,2\}$ was chosen arbitrarily, we may now assume that
    $\{d,f_1,f_5,f_9\}$ and $\{d,f_2,f_5,f_9\}$ are circuits.
    By circuit elimination, $\{f_1,f_2,f_5,f_9\}$ contains a circuit.
    But this set does not contain a triangle, and $f_9 \notin \cl(\{f_1,f_2,f_5\})$, so this is contradictory.
  \end{slproof}

  \begin{sublemma}
    \label{flansl6}
    If $t=6$, then $M$ has an $N$-detachable pair.
  \end{sublemma}
  \begin{slproof}
    For each $i \in \{1,2\}$,
    the matroid $\si(M \ba d / f_i / f_5)$ is $3$-connected, by \cref{flanend}(i), and $M \ba d /f_i /f_5$ has an $N$-minor, by \cref{flannewnsl}.
    First, suppose that $\{d,f_i,f_5\}$ is not contained in a $4$-element circuit, for some $i \in \{1,2\}$.
    It follows that if $M \ba d/f_i/f_5$ is $3$-connected, then $M /f_i/f_5$ is $3$-connected, and $\{f_i,f_5\}$ is an $N$-detachable pair.
    So 
    assume that $\{f_i,f_5\}$ is contained in a $4$-element circuit in $M \ba d$. 
    By orthogonality, this circuit must also contain an element of $\{f_1,f_2,f_3\}-f_i$ and an element of $\{f_3,f_4\}$.  Thus, if it does not contain $f_3$, then $f_5 \in \cl(\{f_1,f_2,f_3,f_4\})$; a contradiction.
    It follows that this circuit
    is $\{f_i,f_3,f_5,q\}$ for some $q \in E(M\ba d)-\{f_1,f_2,f_3,f_4,f_5\}$.
    
    Since $M \ba d / f_i / f_5$ has an $N$-minor, and $\{f_3,q\}$ is a parallel pair in this matroid, $\{d,q\}$ is $N$-deletable.
    Let $F' = \{f_1,f_2,f_3,f_4,f_5\}$.
    Now $F'$ and $F' \cup q$ are $3$-separating in $M \ba d$, by \cref{extendSep}.
    As $|F' \cap E(N)| \le 1$ and $|E(N)| \ge 4$, we have $|E(M\ba d)-F'| \ge 3$.
    We claim that $\co(M \ba d \ba q)$ is $3$-connected.
    If $r(E(M\ba d)-F') \le 2$, then $(E(M\ba d)-F') \cup f_5$ is a rank-$3$ cocircuit, and it follows, by \cref{r3cocirc}, that $\co(M \ba d \ba q)$ is $3$-connected.
    On the other hand, if $r(E(M\ba d)-F') \ge 3$,
    then
    $(F', \{q\},E(M\ba d)-(F' \cup q))$ is a vertical $3$-separation of $M \ba d$, so $\si(M \ba d/q)$ is not $3$-connected, and hence $\co(M \ba d\ba q)$ is $3$-connected by Bixby's Lemma.
    So $\{d,q\}$ is an $N$-detachable pair unless $q$ is in a triad $T^*$ of $M \ba d$.
    Note also that, by the foregoing, $q \in \cl(E(M \ba d)-(F' \cup q))$.

    As $q \notin \cocl_{M \ba d}(F')$, the triad $T^*$ contains at most one element of $F'$.
    By orthogonality, $T^*$ meets $\{f_i,f_3,f_5\}$.
    Now if $f_5 \notin T^*$, then, as $\{f_1,f_2,f_3,f_4\}$ is a circuit, $T^*$ intersects $\{f_1,f_2,f_3,f_4\}$ in two elements; a contradiction.
    It follows that $T^* = \{f_5,q,q'\}$, for $q' \in E(M\ba d)-F'$. 
    Now, as $f_6 \in \cl(F')$ but $f_6 \notin \cl(F'-f_5)$, there is a circuit containing $\{f_5,f_6\}$ that is contained in $F$.
    Again by orthogonality, we deduce that $f_6 \in T^*$, so let $q=f_6$.
    Now $F \cup q'$ is a flan, contradicting that $F$ is a maximal $6$-element flan.  

    \medskip

    It remains to consider the case where $M$ has circuits $\{d,f_1,f_5,g_1\}$ and $\{d,f_2,f_5,g_2\}$ for some $g_1 \in E(M)-\{d,f_1,f_5\}$ and $g_2 \in E(M)-\{d,f_2,f_5\}$.
    First, suppose that $g_1 = g_2 = f_3$.
    Then $\{d,f_1,f_3,f_5\}$ and $\{d,f_2,f_3,f_5\}$ are circuits, so $\{f_1,f_2,f_3,f_5\}$ contains a circuit by circuit elimination.  But $\{f_1,f_2,f_3,f_5\}$ does not contain a triangle, and $f_5 \notin \cl(\{f_1,f_2,f_3\})$, since $f_5 \in \cocl_{M \ba d}(\{f_1,f_2,f_3,f_4\})$, so this is contradictory.

    Let $i \in \{1,2\}$ such that $g_i \neq f_3$.
    We will show that either $\{f_3,g_i\}$ is an $N$-detachable pair, or
    $g_i \in E(M \ba d)-F$ and
    there is a $4$-element cocircuit~$C^*_i$ with $\{f_3,g_i\} \subseteq C^*_i \subseteq F \cup g_i$.

    Using a similar argument as in the proof of \cref{flannewnsl}, it follows from \cref{m2.7} that $M \ba f_3 / f_i / f_5$ has an $N$-minor for $i \in \{1,2\}$.
    Since $\{d,g_i\}$ is a parallel pair in $M/f_i/f_5$, 
    the pair $\{f_3,g_i\}$ is $N$-labelled for deletion,
    up to swapping the $N$-labels on $d$ and $g_i$.

    Suppose that $\co(M \ba f_3 \ba g_i)$ is not $3$-connected.
    Then $M \ba f_3 \ba g_i$ has a $2$-separation $(X,Y)$ for which $(\fcl(X),Y-\fcl(X))$ and $(X-\fcl(Y),\fcl(Y))$ are also $2$-separations.
    Thus, we may assume that
    $\{f_1,f_2,d\} \subseteq X$ and $X$ is fully closed.
    Pick $j$ such that $\{i,j\} = \{1,2\}$.
    If $f_5 \in X$, then
    $g_j \in X$ due to the circuit $\{f_j,f_5,d,g_j\}$, and
    $f_4 \in X$ due to the cocircuit $\{f_4,f_5,d\}$ of $M \ba f_3 \ba g_i$.
    Similarly, if $f_4 \in X$, then $\{f_5,g_j\} \subseteq X$.
    But then 
    $\{f_3,g_i\} \subseteq \cl(X)$, so $(X \cup \{f_3,g_i\},Y)$ is a $2$-separation of $M$; a contradiction.
    So $\{f_4,f_5\} \subseteq Y$.
    Now, if $g_j \in X$, then $f_5 \in \cl(X)$, so $X$ is not fully closed; a contradiction.
    So $g_j \in Y$.
    Consider $\fcl(Y)$.  As $d \in \fcl(Y)$, we have $f_j \in \fcl(Y)$, so $f_i \in \fcl(Y)$, and $(X-\fcl(Y), \fcl(Y) \cup \{f_3,g_i\})$ is a $2$-separation of $M$; a contradiction.
    So $\co(M \ba f_3 \ba g_i)$ is $3$-connected. 

    So we may assume that $M$ has a $4$-element cocircuit $C^*_i$ containing $\{g_i,f_3\}$,
    otherwise $M$ has an $N$-detachable pair.
    By orthogonality, $C^*_i$ meets $\{f_1,f_2,f_4\}$ and $\{d,f_i,f_5\}$.
    Suppose $d \in C^*_i$.  If $f_4 \in C^*_i$, then $\{f_3,f_4,f_5,g_i\}$ is a cosegment of $M \ba d$.  If $g_i \in \{f_1,f_2\}$, then $r^*_{M \ba d}(F-f_6) = 2$, so $\lambda_{M \ba d}(F-f_6) = 4+2-5=1$; a contradiction.  But on the other hand if $g_i \notin \{f_1,f_2\}$, then
    this contradicts orthogonality with the circuit $\{f_1,f_2,f_3,f_4\}$. 
    So $\{f_1,f_2\}$ meets $C^*_i$, and thus $\{f_1,f_2,f_3,g_i\}$ is a cosegment of $M \ba d$.
    As before, $g_i \notin \{f_4,f_5\}$, otherwise $\lambda_{M \ba d}(F)=1$; and $g_i \neq f_6$, since $f_6 \notin \cocl_{M \ba d}(F-f_6)$.
    Observe that 
    there is a circuit contained in
    $\{f_1,f_3,f_4,f_5,f_6\}$ with at least four elements.
    But this contradicts orthogonality.  So $d \notin C^*_i$.

    Suppose $|C^*_i \cap F| \le 2$.
    Then $C^*_i \cap F = \{f_i,f_3\}$.
    Again, pick $j$ such that
    $\{i,j\} = \{1,2\}$, and
    observe that there is a circuit contained in $\{f_j,f_3,f_4,f_5,f_6\}$ that contains $f_3$. But this contradicts orthogonality, so $C^*_i \subseteq F \cup g_i$.
    Now suppose $C^*_i \subseteq F$.
    Since $F-f_6$ is exactly $3$-separating in $M \ba d$ and $f_6 \in \cl(F-f_6)$, it follows that $f_6 \notin \cocl_{M \ba d}(F-f_6)$.
    So $C^*_i \subseteq F-f_6$.
    But, as $f_3 \in C^*_i$, and $r^*_{M \ba d}(\{f_1,f_2,f_3\}) = r^*_{M \ba d}(\{f_3,f_4,f_5\}) = 2$, the set $C^*_i$ is a $4$-element cosegment in $M \ba d$.  Thus $r^*_{M \ba d}(F-f_6)=2$, and $\lambda_{M \ba d}(F-f_6)=1$; a contradiction. 
    We deduce that $g_i \notin F$.

    Finally, we recall that
    $\{d,f_1,f_5,g_1\}$ and $\{d,f_2,f_5,g_2\}$ are circuits, so, by circuit elimination, there are circuits contained in the sets $\{d,f_1,f_2,g_1,g_2\}$ and $\{f_1,f_2,f_5,g_1,g_2\}$.
    By the foregoing, if $f_3 \notin \{g_1,g_2\}$, then $\{g_1,g_2\} \cap F = \emptyset$, and $\{g_1,g_2\} \subseteq \cocl_{M \ba d}(F)$.
    Since $\{d,f_3,f_4,f_5\}$ is a cocircuit, we deduce, by orthogonality, that $\{f_1,f_2,g_1,g_2\}$ is a circuit.  Thus $g_1 \neq g_2$.  Since $\{g_1,g_2\} \subseteq \cocl_{M \ba d}(F)$, we have
    \begin{align*}
    \lambda_{M \ba d}(F \cup \{g_1,g_2\}) &= r_{M \ba d}(F \cup \{g_1,g_2\}) + r^*_{M \ba d}(F) - 8 \\
    &\le (r(F) + 1) +4 - 8 = 1;
  \end{align*}
  a contradiction.
%
    On the other hand, if $g_2 = f_3$, say, then $g_1 \neq f_3$, and there is a cocircuit $C_1^* \subseteq F \cup g_1$ of $M \ba d$ with $g_1 \in E(M \ba d)-F$.
    Since $\{f_1,f_2,f_5,g_1,g_2\}=\{f_1,f_2,f_3,f_5,g_1\}$ contains a circuit and $f_5 \notin \cl(\{f_1,f_2,f_3\})$, this circuit must contain $g_1$.  But then $g_1 \in \cl_{M \ba d}(F) \cap \cocl_{M \ba d}(F)$; a contradiction.
  \end{slproof}
  The \lcnamecref{flannew} now follows from \cref{flansl6,flansl7}.
  \end{proof}

  Next we address the case where $M \ba d$ has a maximal $5$-element flan~$F$.  We can break this into two cases depending on whether or not $d$ fully blocks $F$.  The next \lcnamecref{5flan} primarily deals with the case where $d$ fully blocks $F$.  However, we use a more general hypothesis than this, as the same argument also applies in one situation that arises when $d$ does not fully block $F$.
  More specifically, suppose $F$ has the ordering $(f_1,f_2,f_3,f_4,f_5)$.
  The following \lcnamecref{5flan} applies when
    any $4$-element circuit containing $\{f_i,f_5,d\}$ for $i \in \{1,2\}$ is not contained in $F \cup d$.
  In particular, observe that this is the case when $d$ fully blocks $F$.

\begin{proposition}
    \label{5flan}
    Let $M$ be a $3$-connected matroid with a $3$-connected matroid~$N$ as a minor, where $|E(N)| \ge 4$ and every triangle or triad of $M$ is \unfortunate.
    Suppose that $M \ba d$ is $3$-connected and has a maximal $5$-element flan~$F$ with ordering $(f_1,f_2,f_3,f_4,f_5)$, where
    any $4$-element circuit containing $\{f_i,f_5,d\}$ for $i \in \{1,2\}$ is not contained in $F \cup d$,
    and $M \ba d \ba f_3$ has an $N$-minor.
    Then $M$ has an $N$-detachable pair.
\end{proposition}

\begin{proof}
  We start by showing that either there is an $N$-detachable pair in $F$, or there are certain $4$-element circuits in $M$ containing $d$ and intersecting $F$.

  \begin{sublemma}
    \label{5flansb1}
    For each $i \in \{1,2\}$, the matroids $M\ba d/f_i/f_5$ and $M\ba f_3/f_i/f_5$ have $N$-minors.
  \end{sublemma}
  \begin{slproof}
    The matroid $M \ba d \ba f_3$ has an $N$-minor, and $\{f_1,f_2\}$ and $\{f_4,f_5\}$ are series pairs in this matroid, so $M \ba d \ba f_3 /f_i/f_5$ has an $N$-minor for $i \in \{1,2\}$. 
  \end{slproof}

  Now it follows from \cref{5flansb1} that none of $f_1$, $f_2$, and $f_5$ is contained in an \unfortunate\ triangle.
  Hence, we can apply \cref{flanend} to deduce that $M \ba d/f_i/f_5$ is $3$-connected for each $i \in \{1,2\}$.
  Then $\{f_i,f_5\}$ is an $N$-detachable pair for $i \in \{1,2\}$ unless $d$ is in a parallel pair in $M/f_i/f_5$.
  Since neither $f_i$ nor $f_5$ is contained in an \unfortunate\ triangle, we have that $\{d,f_i,f_5,g_i\}$ is a circuit in $M$ 
  for some $g_i \in E(M) - \{d,f_i,f_5\}$.
  By hypothesis, $g_i \notin F$ for $i \in \{1,2\}$.
  Moreover, if $g_1 = g_2$, then there is a circuit contained in $\{d,f_1,f_2,f_5\}$; a contradiction.
  So $g_1$ and $g_2$ are distinct elements of $E(M\ba d)-F$.

  \begin{sublemma}
    \label{predelcands}
    $\{f_1,f_2,g_1,g_2\}$ is a circuit of $M$.
  \end{sublemma}
  \begin{slproof}
    As $r(\{d,f_5,f_1,f_2,g_1,g_2\}) \le 4$ and $d \notin \cl(E(M)-\{f_3,f_4,f_5\})$, we have $r(\{f_1,f_2,g_1,g_2\}) \le 3$.
    Since neither $f_1$ nor $f_2$ is in an \unfortunate\ triangle, we deduce that $\{f_1,f_2,g_1,g_2\}$ is a circuit.
  \end{slproof}

  We now work towards showing that, for each $i \in \{1,2\}$, either $\{f_3,g_i\}$ is an $N$-detachable pair, or $\{f_3,g_i\}$ is contained in a $4$-element cocircuit.
  To this end, we start by showing that $M \ba f_3 \ba g_i$ has an $N$-minor for $i \in \{1,2\}$.
  Let $i \in \{1,2\}$.
  By \cref{5flansb1}, $M \ba f_3 / f_i / f_5$ has an $N$-minor.
  Since $\{d,g_i\}$ is a parallel pair in this matroid and $|E(N)| \ge 4$, the matroid $M \ba f_3 \ba g_i$ has an $N$-minor.

  \begin{sublemma}
    \label{delcands2}
    Let $i \in \{1,2\}$.
    If $\{f_3,g_i\}$ is not contained in a $4$-element cocircuit of $M$, then $\{f_3,g_i\}$ is an $N$-detachable pair.
  \end{sublemma}
  \begin{slproof}
    We give the proof for $i=2$; the argument is almost identical when $i=1$.
    Observe that neither $f_3$ nor $g_2$ is in a triad of $M$, since $M \ba f_3$ and $M \ba g_2$ have $N$-minors.
    Let $(P,Q)$ be a $2$-separation of $M \ba f_3 \ba g_2$.
    Since $\{f_3,g_2\}$ is not contained in a $4$-element cocircuit of $M$,
    we have that $(\fcl(P),Q-\fcl(P))$ and $(P-\fcl(Q),\fcl(Q))$ are $2$-separations in $M \ba f_3 \ba g_2$.
    So we may assume that $P$ is fully closed.
    As $\{f_1,f_2,d\}$ is a triad in $M \ba f_3 \ba g_2$, 
    without loss of generality $\{f_1,f_2,d\} \subseteq P$.
    If $\{f_5,g_1\}$ meets $P$, then $\{f_5,g_1\} \subseteq P$, due to the circuit $\{f_1,f_5,d,g_1\}$, and $f_4 \in P$ as well, due to the triad $\{f_4,f_5,d\}$.
    But then $(P \cup \{f_3,g_2\},Q)$ is a $2$-separation of $M$; a contradiction.
    So $\{f_5,g_1\} \subseteq Q$, and, similarly, $f_4 \in Q$.
    Now consider $\fcl(Q)$.  Due to the triad $\{f_4,f_5,d\}$, we have $d \in \fcl(Q)$, and thus $\{f_1,f_2\} \subseteq \fcl(Q)$ due to the circuits $\{f_1,f_5,d,g_1\}$ and $\{f_2,f_5,d,g_2\}$.
    Thus $(P-\fcl(Q),\fcl(Q) \cup \{f_3,g_2\})$ is a $2$-separation of $M$; a contradiction.
    So $M \ba f_3 \ba g_2$ is $3$-connected.
    Since $\{f_3,g_2\}$ is $N$-deletable, this completes the proof of \cref{delcands2}.
  \end{slproof}

  It remains to consider the case where $\{f_3,g_i\}$ is contained in a $4$-element cocircuit for each $i \in \{1,2\}$. 
  We next prove two claims regarding the elements that appear in such a cocircuit.

  \begin{sublemma}
    \label{delcands}
    Let $i \in \{1,2\}$.
    If $\{f_3,g_i\}$ is in a $4$-element cocircuit of $M$, then
    either this cocircuit contains $f_i$, or $g_i \in \cocl(F \cup d)$.
  \end{sublemma}
  \begin{slproof}
    Let $C^*$ be a $4$-element cocircuit of $M$ containing $\{f_3,g_i\}$.
    Pick $i'$ such that $\{i,i'\} = \{1,2\}$.
    By orthogonality, $C^*$ meets $\{f_1,f_2,f_4\}$ and $\{f_i,f_5,d\}$.
    Thus, either $f_i \in C^*$, or $C^*$ meets $\{f_{i'},f_4\}$ and $\{f_5,d\}$ in which case $g_i \in \cocl(F \cup d)$.
  \end{slproof}

  \begin{sublemma}
    \label{5flansc}
    Suppose $\{f_i,f_3,g_i,h_i\}$ is a cocircuit of $M$, for some $i \in \{1,2\}$ and $h_i \in E(M) - (F \cup \{d,g_i\})$.
    Then, either $M$ has an $N$-detachable pair, or $h_i \in \cl(F \cup \{d,g_i\})-\{g_1,g_2\}$.
  \end{sublemma}
  \begin{slproof}
    First, we will show that if $\{f_5,h_i\}$ is not contained in a $4$-element circuit, then $\{f_5, h_i\}$ is an $N$-detachable pair.
    Pick $i'$ such that $\{i,i'\} = \{1,2\}$.
    Observe that $\{f_i,f_3,g_1,g_2\}$ is not a cocircuit, by orthogonality with the circuit $\{f_{i'},f_5,d,g_{i'}\}$.  So $h_i \neq g_{i'}$.

    Let $(P,Q)$ be a $2$-separation in $M / f_5 / h_i$, where neither $P$ nor $Q$ is contained in a parallel class.  So $(\fcl(P), Q- \fcl(P))$ is also a $2$-separation.
    Without loss of generality,
    $\{f_i,d,g_i\} \subseteq P$.  If $P$ meets $\{f_{i'},f_3\}$, then
    $\{f_{i'},f_3\} \subseteq P$ due to the cocircuit $\{f_1,f_2,f_3,d\}$, and $f_4 \in P$ due to the circuit $\{f_1,f_2,f_3,f_4\}$.
    But then $(P \cup \{f_5,h_i\},Q)$ is a $2$-separation of $M$; a contradiction.
    So $\{f_{i'},f_3\} \subseteq Q$.  
    Since $\{f_1,f_2,g_1,g_2\}$ is a circuit, by \cref{predelcands}, we have $g_{i'} \in Q$, otherwise $f_{i'} \in \fcl(P)=P$.
    Now consider the $2$-separation $(P',Q') = (P-\fcl(Q),\fcl(Q))$.
    We have $d \in Q'$, due to the triangle $\{f_{i'},d,g_{i'}\}$, and it follows that 
    $f_i \in Q'$, due to the cocircuit $\{f_1,f_2,f_3,d\}$; 
    $f_4 \in Q'$, due to the circuit $\{f_1,f_2,f_3,f_4\}$; and
    $g_i \in Q'$, due to the triangle $\{f_i,d,g_i\}$.
    So $(P', Q' \cup \{f_5,h_i\})$ is a $2$-separation of $M$; a contradiction.
    So $M / f_5 / h_i$ is $3$-connected up to parallel pairs.

    We claim that $M /f_5 / h_i$ has an $N$-minor.
    Since $M /f_i/f_5 \ba f_3$ has an $N$-minor, there is an $N$-labelling $(C,D)$ with $\{f_i,f_5\} \subseteq C$ and $f_3 \in D$.
    As $\{g_i,d\}$ is a parallel pair in $M/f_i/f_5$, we may assume, up to switching the $N$-labels on $g_i$ and $d$, that $g_i \in D$.
    Now, in $M\ba f_3 \ba g_i$, we have that $\{f_i,h_i\}$ is a series pair, so, after switching the $N$-labels on $f_i$ and $h_i$, we obtain an $N$-labelling $(C',D')$ with $\{h_i,f_5\} \subseteq C'$. So $M / h_i / f_5$ has an $N$-minor, as claimed.

    We may now assume that $\{f_5,h_i\}$ is contained in a $4$-element circuit,
    otherwise $M$ has an $N$-detachable pair.
    By orthogonality, this circuit meets $\{f_3,f_4,d\}$ and $\{f_i,f_3,g_i\}$.
    If it does not contain $f_3$, then $h_i \in \cl(\{f_i,f_4,f_5,d,g_i\})$, as required.
    So suppose it contains $f_3$.  Then, again by orthogonality, it also meets $\{f_1,f_2,d\}$, in which case
    $h_i \in \cl(F \cup d)$.
  \end{slproof}

  Observe that $\{g_1,g_2\} \nsubseteq \cocl(F \cup d)$.
  Indeed, if $\{g_1,g_2\} \subseteq \cocl(F \cup d)$, then
  \begin{align*}
    \lambda(F \cup \{d,g_1,g_2\}) 
    &= r(F \cup d) + r^*(F \cup d) - 8 \\
    &\le 5 + 4 - 8 = 1;
  \end{align*}
  a contradiction.
  So there exists some $\ell \in \{1,2\}$ such that $g_\ell \notin \cocl(F \cup d)$.
  Then, by \cref{delcands}, $M$ has a cocircuit $\{f_\ell,f_3,g_\ell, h_\ell\}$ for some $h_\ell \in E(M) - \{f_\ell,f_3,g_\ell\}$. 
  In fact, $h_\ell \notin F \cup d$, since $g_\ell \notin \cocl(F \cup d)$.
  Thus, by \cref{5flansc}, we may assume that $h_\ell \in \cl(F \cup \{d,g_\ell\})$ and $h_\ell \notin \{g_1,g_2\}$.

  \begin{sublemma}
    \label{5flanscomb}
    For each $i \in \{1,2\}$, we have $g_i \notin \cocl((F \cup \{d,g_1,g_2\})-g_i)$. Moreover,
    there are distinct elements $h_1,h_2 \in E(M) -(F \cup \{d,g_1,g_2\})$ such that $\{f_1,f_3,g_1,h_1\}$ and $\{f_2,f_3,g_2,h_2\}$ are cocircuits of $M$.
  \end{sublemma}
  \begin{slproof}
    Consider $\lambda(F \cup \{d,g_1,g_2,h_\ell\})$.
    Observe that $\{g_1,g_2,h_\ell\} \subseteq \cl(F \cup d)$ and 
    $h_\ell \in \cocl(F \cup g_\ell)$.
    Thus,
    \begin{align*}
      \lambda(F \cup \{d,g_1,g_2,h_\ell\}) &= r(F \cup d) + r^*(F \cup \{d,g_1,g_2\}) - 9 \\
      &\le r^*(F \cup \{d,g_1,g_2\}) -4.
    \end{align*}
    Now if either $g_1 \in \cocl(F \cup \{d,g_2\})$ or $g_2 \in \cocl(F \cup \{d,g_1\})$, then $$\lambda(F \cup \{d,g_1,g_2,h_\ell\}) \le (r^*(F \cup d)+1) -4 = 1;$$ a contradiction.

    By \cref{delcands}, $\{f_1,f_3,g_1\}$ and $\{f_2,f_3,g_2\}$ are each contained in a $4$-element cocircuit of $M$.  Let these cocircuits be $\{f_1,f_3,g_1,h_1\}$ and $\{f_2,f_3,g_2,h_2\}$ respectively.
    Observe that, for each $i \in \{1,2\}$, we have $h_i \in E(M)- (F \cup \{d,g_1,g_2\})$, since $g_i \notin \cocl((F \cup \{d,g_1,g_2\})-g_i)$.

    Suppose that $h_1 = h_2$.
    Then $\{f_1,f_2,f_3,g_1,g_2\}$ contains a cocircuit, by cocircuit elimination.
    Since $g_1 \notin \cocl(F \cup g_2)$ and $g_2 \notin \cocl(F \cup g_1)$, it follows that $\{f_1,f_2,f_3\}$ is a cocircuit of $M$; a contradiction.
  \end{slproof}

  Now, by \cref{5flanscomb}, there are distinct elements $h_1,h_2 \in E(M) - (F \cup \{d,g_1,g_2\})$ such that
  $\{h_1,h_2\} \subseteq \cl(F \cup \{d,g_1,g_2\}) = \cl(F \cup d)$.
%
  Note that 
  $\{h_1,h_2\} \subseteq \cocl(F \cup \{g_1,g_2\})$.
  Thus,
  \begin{align*}
    \lambda(F \cup \{d,g_1,g_2,h_1,h_2\}) &= r(F \cup d) + r^*(F \cup \{d,g_1,g_2\}) - 10 \\
    &\le 5 + (r^*(F \cup d)+2) -10 \\
    &= 1;
  \end{align*}
  a contradiction.
  This completes the proof.
\end{proof}

Next we handle the case where $M \ba d$ has a maximal $5$-element flan and $d$ does not fully block $F$.  Since $d$ blocks the triads of $M \ba d$ contained in $F$, we have that $d \in \cl_M(F)$. 

\begin{proposition}
  \label{flan5structure}
  Let $M$ be a $3$-connected matroid with a $3$-connected matroid~$N$ as a minor, where $|E(N)| \ge 4$, and every triangle or triad of $M$ is \unfortunate.
  Let $d$ be an element of $M$ such that $M\backslash d$ is $3$-connected 
  and has a  maximal $5$-element flan~$F$ with ordering $(f_1,f_2,f_3,f_4,f_5)$, where $d \in\cl_M(F)$. 
  Suppose that either
  \begin{enumerate}[label=\rm(\alph*)]
    \item $M \ba d \ba f_5$ has an $N$-minor with $|\{f_1,\dotsc,f_4\} \cap E(N)| \le 1$, or
    \item $M /f_i /f_{i'}$ has an $N$-minor for all distinct $i,i' \in \seq{3}$.
  \end{enumerate}
  Then one of the following holds:
  \begin{enumerate}
    \item $M$ has an $N$-detachable pair, 
    \item $F\cup d$ is a \twisted\ of $M$, 
    \item $F\cup d$ is an \pspider\ of $M$, or
    \item $F\cup d$ is a \tvamoslike\ of $M^*$.
  \end{enumerate}
\end{proposition}

\begin{proof}
  First, we observe that each element in $F-f_5$ is $N$-deletable in $M \ba d$.
  Indeed, if (a) holds, then since $(F-f_5, \{f_5\}, E(M \ba d)-F)$ is a cyclic $3$-separation of $M \ba d$, and $F-f_5$ is a circuit, this follows from \cref{doublylabel}(i). 
  On the other hand, if (b) holds, then since $M / f_i / f_i'$ has an $N$-minor for all distinct $i,i' \in \seq{3}$, and $\{f_1,f_2,f_3,f_4\}$ is a circuit, it follows that each element in $F-f_5$ is $N$-deletable up to an $N$-label switch.

  Now each triad of $M \ba d$ contained in $F$ is not an \unfortunate\ triad,
  so $\{f_1,f_2,f_3,d\}$ and $\{f_3,f_4,f_5,d\}$ are cocircuits of $M$.
  Moreover, as $M \ba d \ba f_3$ has an $N$-minor, and $\{f_1,f_2\}$ and $\{f_4,f_5\}$ are parallel pairs in this matroid,
    $M \ba d/f_i / f_5$ has an $N$-minor for $i \in \{1,2\}$.
  By \cref{flanend}(iii), $M \ba d / f_i / f_5$ is $3$-connected, for $i \in \{1,2\}$.  Thus, 
  assuming (i) does not hold,
  we deduce the existence of $4$-element circuits $\{f_1,f_5,d,g_1\}$ and $\{f_2,f_5,d,g_2\}$. 

  We claim that $\{g_1,g_2\} \subseteq F$ or $\{g_1,g_2\} \subseteq \cl(F \cup d)-(F \cup d)$.
  Suppose $g_1 \notin F$.
  Since $F$ is a maximal flan, $g_1 \notin \cl(F)$.
  By circuit elimination, $\{f_1,f_2,f_5,g_1,g_2\}$ contains a circuit.
  If this circuit contains $g_1$, then $g_1 \in \cl(F \cup g_2)$, so $g_2 \notin F$, and $\{g_1,g_2\} \subseteq \cl(F \cup d)-F$ as required.
  So suppose $\{f_1,f_2,f_5,g_2\}$ is a circuit.  Then $g_2 \in F$, since $F$ is a maximal flan, so
  $g_2 \in \{f_3,f_4\}$.  It follows that $F \subseteq \cl(\{f_1,f_2,g_2\})$; a contradiction.

  If $g_1,g_2 \notin F$, then we can apply \cref{5flan}, so (i) holds.
  So we may assume that $\{g_1, g_2\} \subseteq F$.
  Observe that 
  since $\{f_1,f_2,f_3,f_4\}$ and $\{f_1,f_5,d,g_1\}$ 
  are circuits,
  every element of $F \cup d$ is in a circuit contained in $F \cup d$.

\begin{sublemma}
\label{anothercovering}
If $C_1$ and $C_2$ are distinct circuits of $M$ contained in $F \cup d$,
then $F \cup d = C_1\cup C_2$. Similarly, if 
$C_1^*$ and $C_2^*$ are distinct cocircuits of $M$ contained in $F \cup d$,
then $F \cup d = C_1^*\cup C_2^*$.
\end{sublemma}

\begin{slproof}
The set $F \cup d$ is exactly $3$-separating in $M$, and $r(F \cup d)=4$, so
$r(E(M)-(F \cup d))=r(M)-2$.
Suppose that $C_1^*\subseteq F \cup d$ and $C_2^*\subseteq F \cup d$ are distinct cocircuits of $M$.
Then $E(M)-(C_1^*\cup C_2^*)$ is a flat of rank $r(M)-2$. Thus, if $x\in(F\cup d)-(C_1^*\cup C_2^*)$, then $x\in\cl(E(M)-(F \cup d))$.
But this contradicts the fact that every element of $F \cup d$ is contained in some cocircuit that is itself contained in $F \cup d$.
The proof of the dual result follows in the same manner due to the fact that $r^*_M(F \cup d)=4$ and every element of $F \cup d$ is contained in a circuit that is itself contained in $F \cup d$.
\end{slproof}

  Now we assume that (i) does not hold, and show that either (ii), (iii), or (iv) holds.
  As $\{f_1,f_5,d,g_1\}$ and $\{f_2,f_5,d,g_2\}$ are circuits of $M$ contained in $F \cup d$, either $g_1 = f_2$ and $g_2=f_1$ so that these circuits coincide, or, by \cref{anothercovering}, $\{g_1,g_2\} = \{f_3,f_4\}$.
  We will show that in the former case (iii) or (iv) holds, whereas in the latter case (ii) holds.

  \begin{sublemma}
    \label{flan6sub1}
    $M/f_1 \ba f_2 \ba f_5$ and $M/f_2 \ba f_1 \ba f_5$ have $N$-minors.
  \end{sublemma}
  \begin{slproof}
  First suppose that (a) holds.
  Since $M \ba d \ba f_5$ has an $N$-minor and $M \ba d \ba f_5 / f_1 /f_3$ is connected, \cref{m2.7} implies that 
  $M \ba f_5 / f_1 /f_3$ has an $N$-minor.  Since $\{f_2,f_4\}$ is a parallel pair in this matroid, $M /f_1 \ba f_2 \ba f_5$ has an $N$-minor, up to an $N$-label switch.
  Similarly, $M \ba f_5 / f_2/f_3$ has an $N$-minor, and, up to an $N$-label switch, 
  $M /f_2 \ba f_1 \ba f_5$ has an $N$-minor.

  Now suppose (b) holds.
  Recall that either $\{f_i,f_4,f_5,d\}$ and $\{f_{i'},f_3,f_5,d\}$ are circuits for some $\{i,i'\} = \{1,2\}$, or $\{f_1,f_2,f_5,d\}$ is a circuit.  Assume the former.
  Since $M/f_{i'}/f_3$ has an $N$-minor, and $\{f_i,f_4\}$ and $\{f_5,d\}$ are parallel pairs in this matroid, $M/f_{i'} \ba f_i \ba f_5$ has an $N$-minor.
  Moreover, $M \ba d \ba f_3$ has an $N$-minor, where $\{f_i,f_{i'}\}$ and $\{f_4,f_5\}$ are series pairs in this matroid, so $M \ba f_3 / f_i /f_4$ has an $N$-minor.  But $\{f_5,d\}$ is a parallel pair in this matroid, so $M \ba f_5 / f_i /f_4$ has an $N$-minor.  Now $\{f_{i'},f_3\}$ is a parallel pair in this matroid, so $M / f_i \ba f_{i'} \ba f_5$ has an $N$-minor as required.

  Now we assume that $\{f_1,f_2,f_5,d\}$ is a circuit.
  Since, for any $\{i,i'\} = \{1,2\}$,
  the matroid $M/f_i/f_{i'}$ has an $N$-minor,
  and $\{f_3,f_4\}$ and $\{f_5,d\}$ are parallel pairs in this matroid, $M/f_{i'} \ba f_4 \ba f_5$ has an $N$-minor.
  Since $\{f_3,d\}$ is a series pair in this matroid, $M/f_{i'}/f_3 \ba f_5$ has an $N$-minor.  Now $\{f_i,f_4\}$ is a parallel pair in this matroid, so $M/f_{i'} \ba f_i \ba f_5$ has an $N$-minor as required.
  \end{slproof}

\begin{sublemma}
\label{flan6sub2}
Either $M/f_1\backslash f_2\ba f_5$ or $M/f_2\backslash f_1\ba f_5$ is $3$-connected.
\end{sublemma}

\begin{slproof}
Let $\{i,j\} = \{1,2\}$.
We start by showing that either $M/f_i\backslash f_j\ba f_5$ is $3$-connected, or there is a $4$-element cocircuit $\{f_j, f_j', f_5, h_j\}$ where $f_j' \in \{f_3,f_4\}$ and $h_j \in E(M)-(F \cup d)$.
Consider the $3$-connected matroid $M/f_i$.
Observe that $r_{M/f_i}((F - f_i) \cup d) = 3$.
Since $r_M(F \cup d) = 4$, it follows that $\{f_i,f_3,f_4,d\}$ is independent in $M$. 
So $\{f_3,f_4,f_5,d\}$ is a rank-$3$ cocircuit in $M/f_i$, with $f_5 \in \cl_{M/f_i}(\{f_3,f_4,d\})$.
Thus, by \cref{r3cocirc}, $\co(M/f_i\backslash f_5)$, and indeed $M/f_i\backslash f_5$, is $3$-connected.
Now $(\{f_3,f_4,d\},\{f_j\},E(M)-(F \cup d))$ is a vertical $3$-separation in $M/f_i \ba f_5$.
By Bixby's Lemma, 
$\co(M/f_i \ba f_j \ba f_5)$ is $3$-connected.  
So
$M/f_i\backslash f_j\ba f_5$ is $3$-connected unless $f_j$ is in a triad of $M/f_i\backslash f_5$
that meets both $\{f_3,f_4,d\}$ and $E(M)-(F \cup d)$.
By orthogonality, this triad does not contain $d$.
So $\{f_j,f_j',f_5,h_j\}$ is a cocircuit of $M$ where $f_j' \in \{f_3,f_4\}$ and $h_j \in E(M)-(F \cup d)$, as claimed.

Suppose neither $M/f_2\ba f_1\ba f_5$ nor $M/f_1\ba f_2\ba f_5$ is $3$-connected.  Then $\{f_1,f_1',f_5,h_1\}$ and $\{f_2,f_2',f_5,h_2\}$ are cocircuits, where $f_1',f_2' \in \{f_3,f_4\}$.

Recall that $M /f_i \ba f_j \ba f_5$ has an $N$-minor
when $\{i,j\} = \{1,2\}$.
Since $\{f_j',h_j\}$ is a series pair in this matroid, it follows that $M/f_i/h_j$ has an $N$-minor.

Next, we claim that either $M\ba d/f_1/h_2$ or $M\ba d/f_2/h_1$ is $3$-connected. 
Suppose not, 
so $M\ba d/f_i/h_j$ is not $3$-connected for $\{i,j\} = \{1,2\}$.
Observe that $(F-f_i,\{h_j\},E(M)-(F \cup \{d,h_j\}))$ is a cyclic $3$-separation of $M\ba d/f_i$,
so $\si(M\ba d/f_i/h_j)$ is $3$-connected, by Bixby's Lemma.
Thus $M\ba d/f_i/h_j$ contains a parallel pair, implying that $\{f_i,h_j\}$ is contained in a $4$-element circuit in $M \ba d$ that, by orthogonality, intersects $\{f_1,f_2,f_3\}$ in two elements.
But if $f_3$ is in this circuit, then it also meets $\{f_4,f_5\}$, by orthogonality, in which case $h_j \in \cl(F)$; a contradiction.  We deduce that $\{f_1,f_2,h_j,q_j\}$ is a circuit for some $q_j \in E(M)-(F \cup d)$.

Now $\{f_1,f_2,h_1,q_1\}$ and $\{f_1,f_2,h_2,q_2\}$ are both circuits,
so $r(\{f_1,h_1,h_2,q_1,q_2\}) \le 4$.
Since $f_1 \in \cocl(\{f_2,f_3,d\})$, the set $\{h_1,h_2,q_1,q_2\}$ contains a circuit.
But such a circuit intersects one of the cocircuits $\{f_1,f_5,f_1',h_1\}$ or $\{f_2,f_5,f_2',h_2\}$ in a single element, contradicting orthogonality.


Up to labels, we may now assume that $M\ba d/f_1/h_2$ is $3$-connected.
So either $\{f_1,h_2\}$ is an $N$-detachable pair, contradictory to the assumption that (i) does not hold, or there is a $4$-element circuit of $M$ containing $\{d,f_1,h_2\}$.  By orthogonality, such a circuit meets $\{f_3,f_4,f_5\}$.  So $\{d,f_1,f_\ell,h_2\}$ is a circuit, for $\ell \in \{3,4,5\}$.
But then $h_2 \in \cl(F \cup d) \cap \cocl(F \cup d)$ where $F \cup d$ is exactly $3$-separating; a contradiction.

Thus $M/f_1\backslash f_2\ba f_5$ or $M/f_2\backslash f_1\ba f_5$ is $3$-connected as required.
\end{slproof}

Now \cref{flan6sub1,flan6sub2}, 
together with the assumption that $M$ has no $N$-detachable pairs, implies that
$M$ has a $4$-element cocircuit $\{f_1,f_2,f_5,z\}$.

\begin{sublemma}
\label{flan6sub3}
  If $z \notin F$, then $\{f_3,z\}$ is an $N$-detachable pair.
\end{sublemma}
\begin{slproof}
  First we show that $M/f_3/z$ has an $N$-minor.
  Suppose (a) holds.
  Since $M \ba d \ba f_5$ has an $N$-minor and $M \ba d \ba f_5 / f_2 /f_3$ is connected, \cref{m2.7} implies that $M \ba f_5 / f_2 /f_3$ has an $N$-minor.
  Since $\{f_1,f_4\}$ is a parallel pair in this matroid, $M /f_3 \ba f_1 \ba f_5$ has an $N$-minor, up to an $N$-label switch.
  Now suppose (b) holds.
  Since $M/f_1/f_2$ has an $N$-minor,
  and $\{f_3,f_4\}$ and $\{f_5,d\}$ are parallel pairs in this matroid, $M/f_2 \ba f_4 \ba f_5$ has an $N$-minor.
  Since $\{f_3,d\}$ is a series pair in this matroid, $M/f_2/f_3 \ba f_5$ has an $N$-minor.  Now $\{f_1,f_4\}$ is a parallel pair in this matroid, so $M/f_3 \ba f_1 \ba f_5$ has an $N$-minor.
  So in either case $M /f_3 \ba f_1 \ba f_5$ has an $N$-minor.
  Now $\{f_1,f_2,f_5,z\}$ is a cocircuit of $M$, so $\{f_2,z\}$ is a series pair in $M /f_3 \ba f_1 \ba f_5$.  It follows that $M /f_3 /z$ has an $N$-minor as required.

  Next we show that $\si(M/f_3/z)$ is $3$-connected.
  Evidently, $z \in \cocl(F \cup d)$,
  where $F \cup d$ is exactly $3$-separating, so $z \notin \cl(F \cup d)$, by \cref{gutsstayguts},
  implying $z \in \cocl(E(M)-(F \cup \{d,z\}))$, by \cref{swapSepSides}.
  Note that $M/f_3$ is $3$-connected by \cref{flanend}(i). Now
  $((F-f_3) \cup d, \{z\}, E(M)-(F \cup \{d,z\}))$ is a cyclic $3$-separation in $M/f_3$.
  It follows that $\co(M/f_3 \ba z)$ is not $3$-connected, so
  $\si(M/f_3 /z)$ is $3$-connected by Bixby's Lemma, as required.

  Now, if $M/f_3/z$ is not $3$-connected, then $M$ has a $4$-element circuit containing $\{f_3,z\}$.  By orthogonality, such a circuit~$C$ intersects the cocircuits $\{f_1,f_2,f_5,z\}$, $\{f_1,f_2,f_3,d\}$, and $\{f_3,f_4,f_5,d\}$ in at least two elements.  So $C \subseteq F \cup \{d,z\}$.  But then $z \in \cl(F \cup d)$; a contradiction. 
  We deduce that $M/f_3/z$ is $3$-connected.
\end{slproof}

By \cref{flan6sub3}, we may now assume that $z \in \{f_3,f_4\}$.
Since $\{f_1,f_2,f_3,d\}$ is a cocircuit of $M$, it follows from \ref{anothercovering} that $z = f_4$ so that $\{f_1,f_2,f_4,f_5\}$ is a cocircuit.
Now we examine the potential configurations of the $4$-element circuits $\{f_1,f_5,d,g_1\}$ and $\{f_2,f_5,d,g_2\}$, each of which is contained in $F \cup d$. If $g_1=f_3$, then $g_2=f_4$ due to \ref{anothercovering}.
In this situation, it is easily checked that $F \cup d$ is a \twisted\ of $M$, so that (ii) holds, as illustrated in \cref{tw2}. Similarly, if $g_1=f_4$, we obtain a \twisted\ as shown in \cref{tw1}.

\begin{figure}
  \begin{subfigure}{0.45\textwidth}
    \begin{tikzpicture}[rotate=90,scale=0.7,line width=1pt]
      \tikzset{VertexStyle/.append style = {minimum height=5,minimum width=5}}
      \clip (-2.5,-6) rectangle (4.4,2);
      \draw (0,0) .. controls (-3,2) and (-3.5,-2) .. (0,-4);
      \draw (0,0) -- (4.0,0.9);
      \draw (0,-2) -- (2.5,-2.2); 
      \draw (0,-4) -- (3.8,-4.9); 

      \Vertex[x=3.0,y=0.67,LabelOut=true,Lpos=180,L=$f_1$]{a2}
      \Vertex[x=2.0,y=0.45,LabelOut=true,Lpos=180,L=$f_2$]{a3}

      \Vertex[x=2.5,y=-2.2,LabelOut=true,Lpos=90,L=$f_3$]{b1}
      \Vertex[x=0.64,y=-2.056,LabelOut=true,Lpos=-45,L=$d$]{b2}

      \Vertex[x=3.8,y=-4.9,LabelOut=true,L=$f_4$]{c1}
      \Vertex[x=2.8,y=-4.67,LabelOut=true,L=$f_5$]{c2}

      \draw[dashed] (3.8,-4.9) .. controls (2.0,-2) .. (4.0,0.9);
      \draw[dashed] (2.8,-4.67) .. controls (1.0,-2) .. (3.0,0.67);
      \draw[dashed] (1.8,-4.45) .. controls (0.25,-2) .. (2.0,0.45);

      \draw (0,0) -- (0,-4);

      \SetVertexNoLabel
      \tikzset{VertexStyle/.append style = {shape=rectangle,fill=white}}
      \Vertex[x=4.0,y=0.9]{a1}
      \Vertex[x=1.5,y=-2.12]{b3}
      \Vertex[x=1.8,y=-4.45]{c3}

    \end{tikzpicture}
    \caption{When $g_1 = f_3$.} 
    \label{tw2}
  \end{subfigure}
  \begin{subfigure}{0.45\textwidth}
    \begin{tikzpicture}[rotate=90,scale=0.7,line width=1pt]
      \tikzset{VertexStyle/.append style = {minimum height=5,minimum width=5}}
      \clip (-2.5,-6) rectangle (4.4,2);
      \draw (0,0) .. controls (-3,2) and (-3.5,-2) .. (0,-4);
      \draw (0,0) -- (4.0,0.9);
      \draw (0,-2) -- (2.5,-2.2); 
      \draw (0,-4) -- (3.8,-4.9); 

      \Vertex[x=3.0,y=0.67,LabelOut=true,Lpos=180,L=$f_1$]{a2}
      \Vertex[x=2.0,y=0.45,LabelOut=true,Lpos=180,L=$f_2$]{a3}

      \Vertex[x=2.5,y=-2.2,LabelOut=true,Lpos=90,L=$f_4$]{b1}
      \Vertex[x=0.64,y=-2.056,LabelOut=true,Lpos=-45,L=$f_5$]{b2}

      \Vertex[x=3.8,y=-4.9,LabelOut=true,L=$f_3$]{c1}
      \Vertex[x=2.8,y=-4.67,LabelOut=true,L=$d$]{c2}

      \draw[dashed] (3.8,-4.9) .. controls (2.0,-2) .. (4.0,0.9);
      \draw[dashed] (2.8,-4.67) .. controls (1.0,-2) .. (3.0,0.67);
      \draw[dashed] (1.8,-4.45) .. controls (0.25,-2) .. (2.0,0.45);

      \draw (0,0) -- (0,-4);

      \SetVertexNoLabel
      \tikzset{VertexStyle/.append style = {shape=rectangle,fill=white}}
      \Vertex[x=4.0,y=0.9]{a1}
      \Vertex[x=1.5,y=-2.12]{b3}
      \Vertex[x=1.8,y=-4.45]{c3}

    \end{tikzpicture}
    \caption{When $g_1 = f_4$.} 
    \label{tw1}
  \end{subfigure}
  \caption{The two possible labellings of the \twisted\ when \cref{flan5structure}(ii) holds.}
  \label{tws}
\end{figure}
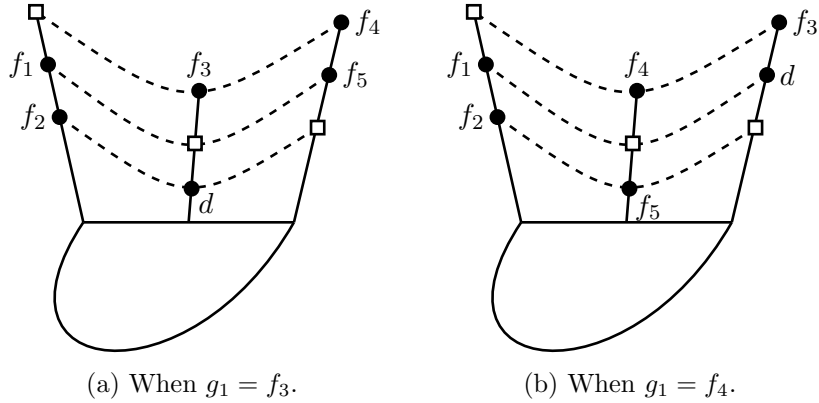

The final possibilities arise when $g_1=f_2$. In this case, \ref{anothercovering} forces $g_2=f_1$. 
First, suppose that $\{f_3,f_4,f_5,d\}$ is a circuit.
Then $F \cup d$ is an \pspider\ with associated partition $(\{f_3,f_4,f_5,d\},\{f_1,f_2\})$, as illustrated in \cref{pss}; so (iii) holds.
We may now assume $\{f_3,f_4,f_5,d\}$ is independent.
Then, since $r(F \cup d) = 4$, the element $f_1$ (respectively, $f_2$) is in a circuit contained in $\{f_1,f_3,f_4,f_5,d\}$ (respectively, $\{f_2,f_3,f_4,f_5,d\}$).
Since $\{f_1,f_2,f_5,d\}$ is also a circuit contained in $F \cup d$, \ref{anothercovering} implies that these circuits contain $\{f_3,f_4\}$.
Similarly, due to the circuit $\{f_1,f_2,f_3,f_4\}$, \ref{anothercovering} implies that these circuits contain $\{f_5,d\}$.
So $\{f_1,f_3,f_4,f_5,d\}$ and $\{f_2,f_3,f_4,f_5,d\}$ are circuits.
It follows that $F\cup d$ is a \tvamoslike\ in $M^*$, so (iv) holds.  The labelling of the \tvamoslike\ in the dual is illustrated in \cref{tw3}.
\end{proof}

\begin{figure}
    \begin{tikzpicture}[rotate=90,scale=0.8,line width=1pt]
      \tikzset{VertexStyle/.append style = {minimum height=5,minimum width=5}}
      \clip (-2.5,2) rectangle (3.0,-6);
      \draw (0,0) .. controls (-3,2) and (-3.5,-2) .. (0,-4);
      \draw (0,0) -- (2,-2) -- (0,-4);
      \draw (0,0) -- (2.5,0.5) -- (2,-2);
      \draw (0,0) -- (2.25,-0.75);
      \draw (2,-2) -- (1.25,0.25);

      \Vertex[x=1.25,y=0.25,LabelOut=true,L=$f_5$,Lpos=180]{c1}
      \Vertex[x=2.25,y=-0.75,LabelOut=true,L=$f_3$,Lpos=90]{c2}
      \Vertex[x=2.5,y=0.5,LabelOut=true,L=$f_4$,Lpos=180]{c3}
      \Vertex[x=1.5,y=-0.5,LabelOut=true,L=$d$,Lpos=135]{c4}
      \Vertex[x=1.33,y=-2.67,LabelOut=true,L=$f_2$,Lpos=45]{c5}
      \Vertex[x=0.67,y=-3.33,LabelOut=true,L=$f_1$,Lpos=45]{c6}

      \draw (0,0) -- (0,-4);

      \SetVertexNoLabel
      \tikzset{VertexStyle/.append style = {shape=rectangle,fill=white}}
    \end{tikzpicture}
  \caption{The labelling of the \pspider\ when \cref{flan5structure}(iii) holds.}
  \label{pss}
\end{figure}
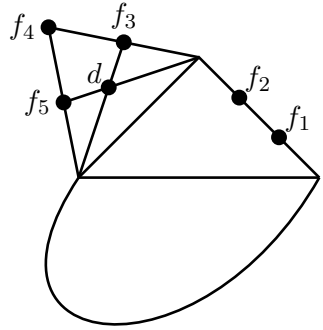

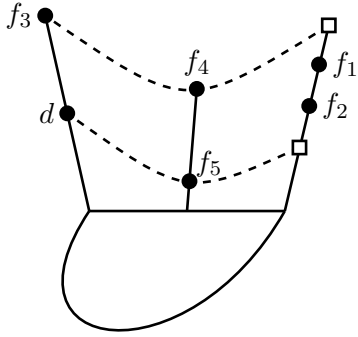
\begin{figure}
  \centering
  \begin{tikzpicture}[rotate=90,scale=0.65,line width=1pt]
    \tikzset{VertexStyle/.append style = {minimum height=5,minimum width=5}}
    \clip (-2.78,-6) rectangle (5.0,2);
    \draw (0,0) .. controls (-3,2) and (-3.5,-2) .. (0,-4);
    \draw (0,0) -- (4.0,0.9);
    \draw (0,-2) -- (2.5,-2.2); 
    \draw (0,-4) -- (3.8,-4.9); 

    \Vertex[x=2.0,y=0.45,LabelOut=true,Lpos=180,L=$d$]{a3}

    \Vertex[x=2.5,y=-2.2,LabelOut=true,Lpos=90,L=$f_4$]{b1}
    \Vertex[x=0.61,y=-2.05,LabelOut=true,Lpos=35,L=$f_5$]{b2}

    \Vertex[x=3.00,y=-4.7,LabelOut=true,L=$f_1$]{c1}
    \Vertex[x=2.15,y=-4.5,LabelOut=true,L=$f_2$]{c2}

    \Vertex[x=4.0,y=0.9,LabelOut=true,Lpos=180,L=$f_3$]{a1}

    \draw[dashed] (3.8,-4.9) .. controls (2.0,-2) .. (4.0,0.9);
    \draw[dashed] (1.3,-4.3) .. controls (0.25,-2) .. (2.0,0.45);

    \draw (0,0) -- (0,-4);

    \SetVertexNoLabel
    \tikzset{VertexStyle/.append style = {shape=rectangle,fill=white}}
    \Vertex[x=3.8,y=-4.9]{c1}
    \Vertex[x=1.3,y=-4.3]{c3}

  \end{tikzpicture}
  \caption{The labelling of the \tvamoslike\ of $M^*$
    when \cref{flan5structure}(iv) holds.}
  \label{tw3}
\end{figure}

By combining \cref{flannew,5flan,flan5structure}, we obtain the following:

\begin{corollary}
  \label{flancorollary}
  Let $M$ be a $3$-connected matroid with a $3$-connected matroid~$N$ as a minor, where $|E(N)| \ge 4$, and every triangle or triad of $M$ is \unfortunate.
  Let $d$ be an element of $M$ such that $M\backslash d$ is $3$-connected 
  and has a flan~$F$ with ordering $(f_1,f_2,\dotsc,f_t)$ where $t \ge 5$. 
  Suppose that $M \ba d \ba f_5$ has an $N$-minor with $|\{f_1,\dotsc,f_4\} \cap E(N)| \le 1$. 
  Then either
  \begin{enumerate}
    \item $M$ has an $N$-detachable pair, or
    \item $F\cup d$ is either a \twisted\ of $M$, an \pspider\ of $M$, or a \tvamoslike\ of $M^*$.
  \end{enumerate}
\end{corollary}
\begin{proof}
  If $t \ge 6$, then (i) holds by \cref{flannew}.
  So suppose $t=5$ and $F$ is a maximal flan.
  First, suppose that $d$ fully blocks $F$.  Towards an application of \cref{5flan}, we claim that $f_3$ is $N$-deletable in $M \ba d$.
  Observe that $(F-f_5, \{f_5\}, E(M \ba d) - F)$ is a cyclic $3$-separation of $M \ba d$.  Since $F-f_5$ is a circuit, \cref{doublylabel}(i) implies that 
  $M \ba d \ba f_3$ has an $N$-minor, as claimed.  Now, by \cref{5flan}, we may assume that $d$ does not fully block $F$.
  Then $d \in \cl(F)$, and so, by \cref{flan5structure}, the \lcnamecref{flancorollary} follows.
\end{proof}

\section{Unveiling the $3$-separating set $X$}
\label{secunveil}

In this section, we prove our main result, \cref{foundation}.
For the entirety of the section, we work under the following hypotheses.
  Let $M$ be a $3$-connected matroid and let $N$ be a $3$-connected minor of $M$ where
  $|E(N)| \ge 4$, and
  every triangle or triad of $M$ is \unfortunate.
  Suppose, for some $d \in E(M)$, that $M \ba d$ is $3$-connected and has a cyclic $3$-separation $(Y, \{d'\}, Z)$ with $|Y| \ge 4$, where $M \ba d \ba d'$ has an $N$-minor with $|Y \cap E(N)| \le 1$.

First, we handle the cases where $Y$ contains either a $4$-element cosegment, or a particular configuration of two triads.

\begin{lemma}
  \label{coseg}
  If $Y$ contains a $4$-element cosegment of $M \ba d$, then $M$ has an $N$-detachable pair.
\end{lemma}
\begin{proof}
  Suppose $X$ is a $4$-element cosegment of $M \ba d$ contained in $Y$.
  If $X \subseteq \cocl(Z)$, then $X \cup d'$ is a cosegment in $M\ba d$. 
  Since $M \ba d \ba d'$ has an $N$-minor, neither $d$ nor $d'$ is in a triad of $M$, and any pair of elements in $\cocl_{M \ba d \ba d'}(X)$ is $N$-contractible.
  In particular, there are no triads of $M$ contained in $\cocl(X \cup \{d,d'\})$, and so
  $M^*|(X \cup \{d,d'\}) \cong U_{3,6}$.
  Now, by \cref{6pointplane2}, $M$ has an $N$-detachable pair.

        So we may assume that $|X \cap \cocl(Z)| \le 1$.
        Let $x \in X$, where $x \in \cocl(Z)$ if such an element exists.
        Since $x' \in \cocl(X -\{x,x'\})$ for each $x' \in X-x$, we have $x' \notin \cl(\cocl(Z))$.
        Thus, by \cref{doublylabel}, each $x' \in X-x$ is 
        doubly $N$-labelled in $M \ba d$. 
        As $d$ blocks every triad of $M \ba d$ contained in $X$, the set $X \cup d$ is a $5$-element coplane in $M$.
        If $d$ does not fully block $X$, then $d \in \cl(X)$, in which case $X \cup d$ is $3$-separating in $M$, and
        $M$ has an $N$-detachable pair by the dual of \cref{basicplaneupgrade}.
        So we may assume that $d$ fully blocks $X$.

        Let $p' \in X-x$.
        Towards an application of \cref{planeupgrade}, we claim that 
        for distinct elements $u,v \in \cl(X)-X$, either $M/p' /u$ or $M/p' /v$ has an $N$-minor.
        Recall that $M \ba d/p'$ has an $N$-minor.
        By the dual of \cref{rank2Remove}, $M \ba d /p'$ is $3$-connected.
        Now $(Y-p', \{d'\}, Z)$ is a path of $3$-separations in $M \ba d/p'$.
        Let $Z' = \cocl_{M\ba d}(Z)-d'$ and $Y' = Y-Z'$.
        Then $(Y'-p', \{d'\}, Z')$ is a path of $3$-separations of $M \ba d/p'$ where $Z' \cup d'$ is coclosed, and $X-x \subseteq Y'$.
        Note that $Y'-p'$ contains a circuit in $M\ba d/p'$, since $Y'$ contains a circuit in $M \ba d$.
        In order to show that $(Y'-p', \{d'\}, Z')$ is a cyclic $3$-separation of $M\ba d/p'$, it remains only to observe that $d' \in \cocl_{M \ba d}(Y'-p')$,
        which follows from the fact that $p' \in \cocl_{M \ba d}(X-\{x,p'\})$.

        We may assume there are distinct elements $u,v \in \cl(X)-X$, otherwise the claim holds trivially.
        Then
        $\{u,v\} \subseteq \cl(Y)$.
        If $\{u,v\} \subseteq Y-p'$, then either $M/p'/u$ or $M/p'/v$ is $3$-connected by \cref{doublylabel}(ii).
        Moreover, if $\{u,v\} \cap Z \neq \emptyset$, then, since $d' \in \cocl_{M\ba d}(Y) \cap Z$, \cref{presingle} implies that $\{u,v\} -Z \neq \emptyset$.
        So suppose, without loss of generality, that $v \in Z$ and $u \in Y$.  By \cref{doublylabel}(ii) again, the claim holds unless $u \in \cl_{M/p'}(Z')$.
        But then it follows that $Y - \{p',u\}$ is exactly $3$-separating in $M\ba d/p'$, with $\{u,v\} \subseteq \cl_{M/p'}(Y-\{p',u\}) \cap Z$ and $d' \in \cocl(Y-\{p',u\}) \cap Z$, contradicting \cref{presingle}.

        Now $M$ has an $N$-detachable pair by the dual of \cref{planeupgrade}.
\end{proof}

\begin{lemma}
  \label{prespecifictriads}
  Suppose that $Y$ contains a set $X = \{s_1, s_2, t_1, t_2, t_3\}$ such that the following hold:
  \begin{enumerate}[label=\rm(\alph*)]
    \item each $x \in X$ is not in a triangle or triad of $M$;
    \item $\{s_1, s_2, t_3\}$ and $\{t_1, t_2, t_3\}$ are triads of $M \ba d$;
    \item for each $i \in \seq{3}$ there are elements $v_i, w_i \in \cl(X \cup d) - (X \cup d)$ such that $\{s_1, t_i, d, v_i\}$ and $\{s_2, t_i, d, w_i\}$ are circuits; and
    \item $P = \{v_1,v_2,v_3,w_1,w_2,w_3\}$ is a $6$-element rank-$3$ set, and if $P$ contains a triangle~$T$, then $T$ is either $\{v_i,v_j,w_k\}$ or $\{v_i,w_j,w_k\}$ for some $\{i,j,k\}=\{1,2,3\}$.
  \end{enumerate}
  Then $M$ has an $N$-detachable pair.
\end{lemma}
\begin{proof}
  Let $i \in [3]$.
  Since $\{s_1,t_i,d,v_i\}$ and $\{s_2,t_i,d,w_i\}$ are circuits, $\{s_1,s_2,t_i,v_i,w_i\}$ contains a circuit, by circuit elimination.
But $t_i \in \cocl_{M \ba d}(\{t_1,t_2,t_3\}-t_i)$, so \cref{swapSepSides} and (a) imply that $\{s_1,s_2,v_i,w_i\}$ is a circuit.
  \begin{sublemma}
    \label{goodminors}
    Either $M$ has an $N$-detachable pair, or $M \ba d \ba v_i$ and $M \ba d \ba w_i$ have $N$-minors for each $i \in \seq{3}$.
  \end{sublemma}
  \begin{slproof}
Note that if $d' \in \cocl_{M \ba d}(\{t_1,t_2,t_3\})$, then $\{d',t_1,t_2,t_3\}$ is a cosegment in $M\ba d$ whose triads are blocked by $d$, so $\{d,d',t_1,t_2,t_3\}$ is a $5$-element plane in $M^*$.  But then, by the dual of \cref{basicplaneupgrade}, $M$ has an $N$-detachable pair.
So we may assume that $d' \notin \cocl_{M \ba d}(\{t_1,t_2,t_3\})$.
Now $M \ba d/s_1$ is $3$-connected by the dual of \cref{r3cocirc}, and $M \ba d \ba d' / s_1$ has an $N$-minor by \cref{m2.7}.
Applying \cref{doublylabel}(ii), 
we deduce that $M \ba d / s_1 / s_2$ has an $N$-minor, since $d' \notin \cocl_{M \ba d/s_1}(X-\{s_1,s_2\})$.
As $\{v_i,w_i\}$ is a parallel pair in $M \ba d /s_1 /s_2$, for each $i \in \seq{3}$, the matroids $M \ba d \ba v_i$ and $M \ba d \ba w_i$ have $N$-minors.
\end{slproof}

If $\{v_i,v_j,w_k\}$ and $\{v_i,w_j,w_k\}$ are independent for all $\{i,j,k\} = \{1,2,3\}$, then $M|P \cong U_{3,6}$, and $M$ has an $N$-detachable pair by \cref{6pointplane,goodminors}.
So, without loss of generality, we may assume that $\{v_i,v_j,w_3\}$ or $\{v_i,w_j,w_3\}$ is a triangle $T$ for some $\{i,j\} = \{1,2\}$.
We claim that $w_3$ is not in a triad of $M \ba d$.
Towards a contradiction, suppose that $T^*$ is a triad of $M \ba d$ containing $w_3$.
By (d), $\{v_1,v_2,w_1,w_2\}$ is a circuit $C$.
By orthogonality between $T^*$ and $T$, and between $T^*$ and $C$, we deduce that $T^* -w_3\subseteq \{v_1,v_2,w_1,w_2\}$.
But then $T^*$ intersects the circuit $\{s_1,s_2,v_3,w_3\}$ in a single element; a contradiction.

Now, by Tutte's Triangle Lemma, either $M \ba d\ba w_3$ or $M \ba d\ba v_i$ is $3$-connected.
By \cref{goodminors}, it follows that $M$ has an $N$-detachable pair, thus completing the proof.
\end{proof}

\begin{lemma}
  \label{specifictriads}
  Suppose that $Y$ contains a set $X = \{s_1, s_2, t_1, t_2, u\}$ such that the following hold:
  \begin{enumerate}[label=\rm(\alph*)]
    \item $\{s_1, s_2, u\}$ and $\{t_1, t_2, u\}$ are triads of $M \ba d$,\label{st3}
    \item $X$ is closed in $M \ba d$,\label{st1}
    \item $X$ is $3$-separating in $M \ba d$,\label{st2}
    \item $X$ is not a cosegment in $M \ba d$, and\label{st5}
    \item there are no $4$-element circuits contained in $X$.\label{st4}
  \end{enumerate}
  Then $M$ has an $N$-detachable pair.
\end{lemma}
\begin{proof}
  Since $X$ is the union of two triads that meet at $u$, but $X$ is not a cosegment, $r_{M \ba d}^*(X) = 3$.  As $X$ is a $5$-element $3$-separating set, $r_{M \ba d}(X) = 4$.
  It follows that $E(M \ba d)-X$ is coclosed, due to \ref{st4}.
  Since $x \in \cocl_{M \ba d}(X-x)$ for each $x \in X$, we also have that $E(M \ba d)-X$ is closed.
  Hence each element in $X$ is doubly $N$-labelled in $M \ba d$ by \cref{doublylabel}.
  It follows that each $x \in X$ is not contained in an \unfortunate\ triangle or triad.

  Assume that $M$ does not contain an $N$-detachable pair.
\begin{sublemma}
    \label{sll1}
    For distinct $s \in \{s_1 , s_2, u\}$ and $t \in \{t_1 , t_2, u\}$, the matroid $M \ba d/s/t$ is $3$-connected and has an $N$-minor.
\end{sublemma}
\begin{slproof}
Let $s \in \{s_1, s_2\}$ and $t \in \{t_1, t_2, u\}$.
Since $X$ is a corank-$3$ circuit, and $s$ is not contained in a triangle,
the dual of \cref{r3cocirc} implies that $M\ba d /s$ is $3$-connected.
Moreover, $X-s$ is a corank-$3$ circuit in $M \ba d / s$, so $M \ba d / s / t$ is $3$-connected unless $\{s,t\}$ is contained in a $4$-element circuit of $M \ba d$. But, by orthogonality, such a circuit contains another element of $X$, and so, as $X$ is closed in $M \ba d$, the circuit is contained in $X$; a contradiction.  It follows by symmetry that $M \ba d / s /t$ is $3$-connected.

It remains to show that $M \ba d/s/t$ has an $N$-minor.
By swapping the labels on $\{s_1,s_2\}$ and $\{t_1,t_2\}$, if necessary, we may assume that $s \neq u$.
Recall that $M\ba d/s$ has an $N$-minor.
Now $M \ba d/s$ is $3$-connected by the dual of \cref{r3cocirc}, and $M \ba d \ba d' / s$ has an $N$-minor by \cref{m2.7}.
Applying \cref{doublylabel}(ii), we deduce that $M \ba d / s / t$ has an $N$-minor, since $t \in \cocl_{M\ba d}(\{t_1,t_2,u\}-t)$, so $t \notin \cl(E(M) - \{t_1,t_2,u\})$.
\end{slproof}

Now, as $M$ has no $N$-detachable pairs, \cref{sll1} implies that for each distinct pair $s,t$ with $s \in  \{s_1 , s_2, u\}$ and $t \in \{t_1 , t_2, u\}$, there is a circuit of $M$ containing $\{d,s,t\}$.
\begin{sublemma}
    \label{sll2}
    There are no $4$-element circuits of $M$ contained in $X \cup d$.
\end{sublemma}
\begin{slproof}
    Suppose $X \cup d$ contains a $4$-element circuit $C$.
    Then $d \in C$, by \ref{st4}.
    Let $S = \{s,s'\} \in \{\{s_1,s_2\}, \{t_1,t_2\}\}$, and $T = \{t,t',t''\} = X-S$.
    We may assume, without loss of generality, that $C=\{d,s,t,x\}$, where $x \neq t'$.
    Now $\{d,s,t'\}$ is also contained in a $4$-element circuit, $\{d,s,t',y\}$ say.  By circuit elimination, $\{s,t,t',x,y\}$ contains a circuit.  By \ref{st4}, $\{s,t,t',x\}$ is independent, so \ref{st1} implies that $y \in X$,
    and $y \notin \{s,t,t',x\}$,
    and thus $\{x,y\} = \{s',t''\}$.

    If $x = t''$, then $\{d,s,t,t''\}$ and $\{d,s,t',s'\}$ are circuits, but there is also a $4$-element circuit containing $\{d,s',t\}$.
    So let $\{d,s',t,z\}$ be a circuit, for some $z$.  Now $\{s,t,t'',s',z\}$ contains a circuit, by circuit elimination, 
    and it follows, by \ref{st1} and \ref{st4}, that $z \in X-\{s,s',t,t''\}$, so $z=t'$.
    But then circuit elimination on the circuits $\{d,s,t',s'\}$ and $\{d,s',t,t'\}$ implies that $\{s,s',t,t'\}$ contains a circuit; a contradiction.
    The argument is similar when $x=s'$.
\end{slproof}

Now, letting $t_3 = u$, 
for each $i \in \seq{3}$ there are elements $v_i, w_i \in \cl(X \cup d) - (X \cup d)$ such that $\{s_1, t_i, d, v_i\}$ and $\{s_2, t_i, d, w_i\}$ are circuits. 
Observe also that $d \notin \cl(X)$, since $v_i,w_i \notin \cl(X)$.

\begin{sublemma}
    \label{almostU36}
    Let $P = \{v_1,v_2,v_3,w_1,w_2,w_3\}$.
    Then $|P|=6$ and $r(P)=3$.
    Moreover, if $P$ contains a triangle~$T$, then $T$ is either $\{v_i,v_j,w_k\}$ or $\{v_i,w_j,w_k\}$ for some $\{i,j,k\}=\{1,2,3\}$.
\end{sublemma}
\begin{slproof}
If $v_i=v_{i'}$ for distinct $i,i' \in \seq{3}$, then $\{s_1, t_i,t_{i'},d\}$ contains a circuit, by the circuit elimination axiom, contradicting \ref{sll2}.
Similarly, the $w_i$ are pairwise distinct for $i \in \seq{3}$.
Say $v_i = w_j$ for some $i, j \in \seq{3}$. Then, again by circuit elimination, there is a circuit $\{s_1, s_2, t_i, t_j, v_i\}$.
But $X$ is closed in $M \ba d$, so $v_i \notin \cl(\{s_1, s_2, t_i, t_j\})$. Hence $\{s_1, s_2, t_i, t_j\}$ is a circuit of $M$, contradicting 
\ref{st4}.
Hence the elements $v_1, v_2, v_3, w_1, w_2, w_3$ are distinct.
By \ref{st2}, 
$\cl(X \cup d) - (X \cup d)$ has rank at most 3.
If $r(\{v_1,v_2,v_3\}) \le 2$, then $\{s_1,d,v_1,v_2,v_3\}$ has rank at most four, but spans the rank-$5$ set $X \cup d$; a contradiction.
A similar argument applies if $r(\{w_1,w_2,w_3\}) \le 2$, or, for some distinct $i,j \in [3]$ either $r(\{v_i,v_j,w_i\}) \le 2$ or $r(\{v_i,w_i,w_j\}) \le 2$.
\end{slproof}

The \lcnamecref{specifictriads} now follows from \cref{almostU36,prespecifictriads}.
\end{proof}

Finally, we come to the main result of this paper.
Recall that $d \in E(M)$,
the matroid $M \ba d$ is $3$-connected and has a cyclic $3$-separation $(Y, \{d'\}, Z)$ with $|Y| \ge 4$, and $M \ba d \ba d'$ has an $N$-minor with $|Y \cap E(N)| \le 1$.

\begin{theorem}
  \label{foundation}
  Suppose $M$ has no $N$-detachable pairs.
  Then there is a subset $X$ of $Y$ such that
  \begin{enumerate}
    \item $|X| \ge 4$ and 
      $X$ is $3$-separating in $M \ba d$, and\label{setup}
    \item either
      \begin{enumerate}[label=\rm(\alph*)]
        \item $X \cup \{c,d\}$ is an \pspider\ of $M$, a \twisted\ of $M$, or a \tvamoslike\ of $M^*$, for some $c \in \cocl_{M \ba d}(X)-X$;
          or\label{problemseps}
        \item for every $x \in X$, 
          \begin{enumerate}[label=\rm(\Roman*)]
            \item $\co(M \ba d \ba x)$ is $3$-connected,\label{ond}
            \item $M \ba d / x$ is $3$-connected, and\label{com}
            \item $x$ is doubly $N$-labelled in $M \ba d$.\label{onddl}
          \end{enumerate}\label{nonproblemseps}
      \end{enumerate}
  \end{enumerate}
\end{theorem}
\begin{proof}
    Choose $X \subseteq Y$ that is minimal with respect to 
    \ref{setup}.
    Let $W = E(M \ba d) -X$.
    Suppose that \cref{problemseps} does not hold; then, it remains to show that \cref{nonproblemseps} holds.
    By \cref{coseg}, we may 
    assume that $X$ is not a cosegment of $M \ba d$.

    \begin{sublemma}
        \label{dli}
        Every element in $Y \cup d'$ is $N$-deletable in $M \ba d$, and
        every element in $X$ is doubly $N$-labelled in $M \ba d$.
    \end{sublemma}
    \begin{slproof}
      If there is some element $x \in X \cap \cocl_{M \ba d}(Z)$, then $X-x$ is $3$-separating, by \cref{extendSep}.
      If $|X| > 4$, this contradicts the minimality of $X$.  On the other hand, if $|X| = 4$, then $X-x$ is a triad, since $X-x$ cannot be an \unfortunate\ triangle by \cref{doublylabel}(ii).  But then $X$ is a $4$-element cosegment, contradicting \cref{coseg}.

      Now we may assume that $Z \cup d'$ is coclosed in $M \ba d$.
      By \cref{doublylabel}(i), every element in $Y$ is $N$-deletable, while $d'$ is $N$-deletable by hypothesis.
      If there is some element $x \in X$ that is not $N$-contractible, then $x \in \cl(Z)$ 
      by \cref{doublylabel}(ii).  Then,
      the minimality of $X$ implies that $|X|=4$.

    Since $X-x$ is not an \unfortunate\ triangle, $X-x$ is a triad, and $X$ is a circuit.
    Moreover, $(X-x, \{x\},W)$ is a vertical $3$-separation, so
    $\co(M \ba d \ba x)$ is $3$-connected by Bixby's Lemma.
    Since $M$ has no $N$-detachable pairs, $x$ is in a triad of $M \ba d$ that meets $X-x$ and $W$.
    Let this triad be $\{x',x,w\}$ where $x' \in X-x$.
    Since $M \ba d \ba x'$ has an $N$-minor, and $\{x,w\}$ is a series pair in this matroid, up to an $N$-label switch the matroid $M \ba d / x$ has an $N$-minor after all, 
    thus completing the proof of \cref{dli}.
    \end{slproof}

    Note, in particular, that no triangle meets $X$. 
    
    \begin{sublemma}
      \label{dlflan}
      If $|X|=4$ and $X \cup f_5$ is a flan for some $f_5 \in W$, with flan ordering $(f_1,f_2,f_3,f_4,f_5)$ for some labelling $\{f_1,f_2,f_3,f_4\}$ of $X$,  
      then either $M$ has an $N$-detachable pair, or \cref{problemseps} holds.
    \end{sublemma}
    \begin{slproof}
      Suppose $f_5$ is $N$-deletable in $M \ba d$.
      Then, by \cref{flancorollary}, either $M$ has an $N$-detachable pair, or \cref{problemseps} holds.
      So we may assume that $f_5$ is not $N$-deletable in $M \ba d$.
      By \cref{dli}, $f_5 \in Z$.
      Now $(Y \cup f_5, \{d'\}, Z-f_5)$ is a 
      path of $3$-separations in $M \ba d$, by \cref{extendSep}.
      By \cref{swapSepSides,exactSeps}, $d' \in \cocl_{M \ba d}(Z-f_5)$.
      Moreover, $Z-f_5$ contains a circuit, since $Z$ contains a circuit and $f_5 \notin \cl(Z-f_5)$, so this path of $3$-separations is a cyclic $3$-separation, and $|(Y \cup f_5) \cap E(N)| \le 1$, by \cref{m2.7} and since $|Y \cap E(N)| \le 1$ and $|E(N)| \ge 4$.

      Suppose there is some $f_6 \in \cl(X \cup f_5) \cap (W-f_5)$ so that $X \cup \{f_5,f_6\}$ is a flan.
      Now $(Y \cup \{f_5,f_6\}, \{d'\}, Z-\{f_5,f_6\})$ is a path of $3$-separations 
      where $d'$ is a coguts element,
      using a similar argument as in the previous paragraph.
      To show this is a cyclic $3$-separation, we now require only that $r^*_{M \ba d}(Z-\{f_5,f_6\}) \ge 3$.
      Suppose not.
      Since $M \ba d \ba d'$ has an $N$-minor with $|(Y \cup f_5) \cap E(N)| \le 1$, \
      \cref{m2.7} implies that $|(Y \cup \{f_5,f_6\}) \cap E(N)| \le 1$.
      But now $r^*_{M \ba d \ba d'}(Z-\{f_5,f_6\}) \le 1$; a contradiction.
      By \cref{doublylabel}(ii), 
      since $f_5$ is not $N$-deletable we have $f_5 \in \cocl(Z - \{f_5,f_6\})$.
      But $f_6 \in \cl(Z - \{f_5,f_6\})$ and $d' \in \cocl(Z - \{f_5,f_6\})$, contradicting \cref{presingle}.
      So $X \cup f_5$ is a maximal flan.

      Note that $M \ba d \ba f_3$ has an $N$-minor, by \cref{dli}. 
      If $d$ fully blocks $X \cup f_5$, 
      then,
      by \cref{5flan}, $M$ has an $N$-detachable pair.
      Towards an application of \cref{flan5structure}, we show that $M /f_i/f_{i'}$ has an $N$-minor for all distinct $i,i' \in \seq{3}$.
      Let $i \in \{1,2\}$.
      By \cref{flanend,dli}, $M\ba d/f_i$ is $3$-connected and has an $N$-minor.
      Now $((Y-f_i) \cup f_5, \{d'\},Z-f_5)$ is a cyclic $3$-separation in $M \ba d/f_i$.  Since $\{f_3,f_4,f_5\}$ is a triad in $M \ba d$, we have $f_3 \notin \cl(Z-f_5)$, so $M \ba d/f_i/f_3$ has an $N$-minor by \cref{doublylabel}(ii).
      Now, $\{f_5,d'\} \subseteq \cocl_{M \ba d/f_i}(Z-f_5)$, so no element in $(Y-f_i) \cup f_5$ is also in $\cl_{M \ba d/f_i}(Z-f_5)$ by \cref{presingle}.
      Hence $M \ba d/f_1/f_2$ also has an $N$-minor by \cref{doublylabel}(ii).
      Now, by \cref{flan5structure}, either $M$ has an $N$-detachable pair or \cref{problemseps} holds, thus completing the proof. 
    \end{slproof}

    Next we prove that \ref{ond} holds for each $x \in X$.
    Towards a contradiction,
    let $x$ be an element of $X$ such that $\co(M \ba d \ba x)$ is not $3$-connected, and let $(P, \{x\}, Q)$ be a cyclic $3$-separation of $M \ba d$.
    \begin{sublemma}
        \label{ondi}
        $W \cap P \neq \emptyset$ and $W \cap Q \neq \emptyset$.
    \end{sublemma}
    \begin{slproof}
    Suppose that $W \cap Q = \emptyset$.
    Then $Q \cup x \subseteq X$ and $|Q| \ge 3$.
    But $Q$ and $Q \cup x$ are $3$-separating, so the minimality of $X$ implies that $X =Q \cup x$ and $|Q|=3$.  Since $Q$ contains a circuit, $Q$ is a triangle of $M\ba d$, and hence of $M$.  But, 
    by \cref{dli},
    $Q$ is not \unfortunate; a contradiction.  So $W \cap Q$ and, by symmetry, $W \cap P$ are non-empty.
    \end{slproof}

    \begin{sublemma}
        \label{ondii}
        Up to swapping $P$ and $Q$, 
        $|X \cap Q| = 2$ and
        $|W \cap P| \ge 2$.
    \end{sublemma}
    \begin{slproof}
      Since $|W| \ge 3$, we may assume that $|W \cap P| \ge 2$.
    By uncrossing, $X \cap Q$ and $(X \cap Q) \cup x$ are $3$-separating in $M \ba d$.
    If $|X \cap Q| \le 1$, then $|W \cap Q| \ge 2$, in which case $X \cap P$ and $(X \cap P) \cup x$ are also $3$-separating in $M \ba d$, by uncrossing.
    By the minimality of $X$, it follows that $|X|=4$, so
    either $X \cap Q = \emptyset$ and $|X\cap P| = 3$, or $|X \cap Q| = 1$ and $|X \cap P| = 2$.  In the first case, $X-x$ is a triad, since it cannot be an \unfortunate\ triangle, so $X$ is a $4$-element cosegment, contradicting \cref{coseg}.  In the latter case, \cref{ondii} holds after swapping $P$ and $Q$. 
    On the other hand, if $|X \cap Q| > 2$, then the minimality of $X$ implies that $X \cap P = \emptyset$.
    But then $X-x$ is a triad, so $X$ is a $4$-element cosegment, contradicting \cref{coseg}.
    \end{slproof}

    Now, note that if $|W \cap Q| = 1$, then $Q$ is a triangle in $M \ba d$, 
    but $Q$ is not an \unfortunate\ triangle since, by \cref{dli}, it contains an $N$-contractible element; a contradiction.
    So $|W \cap Q| \ge 2$.

    \begin{sublemma}
        \label{ondiii}
        $|X \cap P| = 2$.
    \end{sublemma}
    \begin{slproof}
    By uncrossing, $X \cap P$ and $(X \cap P) \cup x$ are $3$-separating.
    If $|X \cap P| > 2$, then this contradicts the minimality of $X$.
    So assume that $X \cap P = \{t\}$, say.
    Now $X-t$ is a triad, and
    $t \in \clstar(X-t)$.
    If $t \in \cocl(X-t)$, then $X$ is a $4$-element cosegment, contradicting \cref{coseg}.
    So $t \in \cl(X-t)$.
    By the dual of \cref{r3cocircsi}, $\co(M \ba d \ba t)$ is $3$-connected, so, as $M$ has no $N$-detachable pairs, $t$ is in a triad that, by orthogonality, meets $X-t$. If this triad does not contain $x$, then, by the dual of \cref{r3cocircsi} again, $\co(M \ba d \ba x)$ is $3$-connected; a contradiction.
    Let $f_5$ be the element of the triad in $W$, and let $X \cap Q = \{f_1,f_2\}$. 
    Now $X$ is contained in a $5$-element flan with ordering $(f_1,f_2,x,t,f_5)$.
    Thus, by \cref{dlflan}, either $M$ has an $N$-detachable pair or \cref{problemseps} holds; a contradiction.
    \end{slproof}

    \begin{sublemma}
        \label{ondiv}
        $X$ is closed in $M\ba d$.
    \end{sublemma}
    \begin{slproof}
        Suppose $c \in \cl(X) - X$.  We may assume that $c \in P$. 
        Since $|W \cap Q| \ge 2$, 
        both $X \cap P$ and
        $(X \cap P) \cup c$ are $3$-separating, by uncrossing. So
        $c \in \cl(X \cap P)$, and $(X \cap P) \cup c$ is a triangle.  Since this triangle contains an $N$-contractible element, it is not \unfortunate, which is contradictory.
    \end{slproof}

    \begin{sublemma}
        \label{ondv}
        $X$ contains no $4$-element circuits.
    \end{sublemma}
    \begin{slproof}
        Let $X \cap P = \{p_1,p_2\}$ and $X \cap Q = \{q_1,q_2\}$.
        Suppose $X$ has a $4$-element circuit.  Either this circuit contains $x$ or it does not.  Suppose that it does: without loss of generality, let $\{p_1,p_2,x,q_1\}$, be this circuit.  Since $\{p_1,p_2,x\}$ is a triad, $\{p_1,p_2,x,q_1\}$ is $3$-separating, contradicting the minimality of $X$.
        Now we may assume there is no $4$-element circuit in $X$ containing $x$.
        Thus $r(\{p_1,p_2,x,q_1,q_2\})=4$, and it follows, 
        by \cref{swapSepSides}, that $x \in \cocl(W)$, so $X-x$ is $3$-separating by \cref{extendSep},
        again contradicting the minimality of $X$.
    \end{slproof}

    Now, since \cref{ondii,ondiii,ondiv,ondv} hold, we can apply \cref{specifictriads} to deduce that $M$ has an $N$-detachable pair; a contradiction.  This proves that each $x \in X$ satisfies \ref{ond}.
    Recall that each $x \in X$ satisfies \cref{onddl} by \cref{dli}.

    It remains to consider \ref{com}.
    Suppose $M \ba d / x$ is not $3$-connected for some $x \in X$.
    Since $x$ is not in a triangle, $\si(M \ba d/x)$ is not $3$-connected, so $M \ba d$ has a vertical $3$-separation $(P, \{x\}, Q)$.
    We may assume, without loss of generality, that $|W \cap P| \ge 2$.
    Thus, by uncrossing, both $X \cap Q$ and $(X \cap Q) \cup x$ are $3$-separating.
    By the minimality of $X$, we have $|X \cap Q| \le 3$, and if $|X \cap Q|=3$, then $X \cap P = \emptyset$.
    If $|X \cap Q| = 2$, then $(X \cap Q) \cup x$ is a triangle or a triad, but as $x \in \cl(Q)$ and $X$ contains no triangles, 
    this leads to a contradiction.

    Suppose $|X \cap Q| \le 1$. Then $|W \cap Q| \ge 2$, in which case $X \cap P$ and $(X \cap P) \cup x$ are $3$-separating, by uncrossing.
    Now $|X\cap P| \ge 2$, but, by the minimality of $X$, $|X\cap P| \le 3$ and if $|X\cap P| = 3$ then $X \cap Q=\emptyset$.
    Moreover, 
    $|X \cap P| \neq 2$ since $x \in \cl(X \cap P)$ and $X$ does not contain any triangles.
    It follows that 
    $X \cap Q = \emptyset$ and $|X \cap P| = 3$.
    
    Now $\{|X \cap P|, |X \cap Q|\} = \{0,3\}$, $X-x$ is a triad, and $X$ is a circuit.
    Since $\co(M \ba d \ba x)$ is $3$-connected, but $M \ba d \ba x$ is not, $x$ is in a triad $T^*$ of $M \ba d$ that meets both $X-x$ and $W$, by orthogonality.
    Let $T^* \cap W=\{f_5\}$, and observe that $X \cup f_5$ is a $5$-element flan of $M \ba d$.
    By \cref{dlflan}, either $M$ has an $N$-detachable pair or \cref{problemseps} holds; a contradiction.
    So each $x \in X$ also satisfies \ref{com}, as required.  
    This completes the proof of \cref{foundation}.
\end{proof}

\section*{Acknowledgements}
    We thank the referees for their careful reading of the paper.

\bibliographystyle{abbrv}
\bibliography{lib}

\end{document}